\documentclass[12pt,centertags,oneside]{amsart}
\usepackage{amsmath,amstext,amsthm,amscd,typearea,hyperref}
\usepackage{amssymb}
\usepackage{a4wide}
\usepackage[mathscr]{eucal}
\usepackage{mathrsfs}
\usepackage{typearea}
\usepackage{charter}
\usepackage{pdfsync}
\usepackage[a4paper,width=16.5cm,top=3cm,bottom=3cm]{geometry}

\numberwithin{equation}{section}

\allowdisplaybreaks
\emergencystretch=\maxdimen
\usepackage[foot]{amsaddr} 

\usepackage{multicol}

\usepackage{xcolor}

\newtheorem{theorem}{Theorem}[section]
\newtheorem{definition}[theorem]{Definition}
\newtheorem{proposition}[theorem]{Proposition}
\newtheorem{corollary}[theorem]{Corollary}
\newtheorem{lemma}[theorem]{Lemma}
\newtheorem{remark}[theorem]{Remark}

\newtheorem{mainthm}{Theorem}

\newcommand{\cali}[1]{\mathscr{#1}}

\newcommand{\GL}{{\rm GL}}

\newcommand{\U}{{\rm U}}

\newcommand{\SL}{{\rm SL}}

\newcommand{\supp}{{\rm supp}}

\newcommand{\diff}{{\rm d}}

\newcommand{\del}{\partial}

\renewcommand{\Im}{\mathop{\mathrm{Im}}}
\renewcommand{\Re}{\mathop{\mathrm{Re}}}

\newcommand{\ddc}{{\rm dd^c}}

\newcommand{\dbar}{\overline\partial}

\renewcommand{\GL}{{\rm GL}}

\newcommand{\ep}{\epsilon}

\newcommand{\Leb}{{\rm Leb}}
\newcommand{\Lip}{{\rm Lip}}
\newcommand{\Arg}{{\rm Arg}}

\newcommand{\Cc}{\cali{C}}

\newcommand{\Hc}{\cali{H}}

\newcommand{\Tc}{\cali{T}}

\newcommand{\W}{\cali{W}}

\newcommand{\B}{\mathbb{B}}
\newcommand{\DD}{\mathbb{D}}
\newcommand{\D}{\mathbb{D}}
\newcommand{\C}{\mathbb{C}}

\newcommand{\N}{\mathbb{N}}
\newcommand{\Z}{\mathbb{Z}}
\newcommand{\R}{\mathbb{R}}

\renewcommand\P{\mathbb{P}}

\newcommand{\E}{\mathbf{E}}

\newcommand{\omegaFS}{ \omega_{\mathrm{FS}}}

\newcommand{\lp}{\langle}
\newcommand{\rp}{\rangle}

\newcommand{\norm}[1]{\lVert#1\rVert}

\newcommand{\oA}{\mathcal{A}}
\newcommand{\oB}{\mathcal{B}}

\newcommand{\oG}{\mathcal{G}}
\newcommand{\oP}{\mathcal{P}}
\newcommand{\oN}{\mathcal{N}}

\newcommand{\oE}{\mathcal{E}}

\newcommand{\oQ}{\mathcal{Q}}

\newcommand{\oL}{\mathcal{L}}

\usepackage{fancyhdr}

\title{Random walks on $\SL_2(\C)$: spectral gap and limit theorems}

\author{Tien-Cuong Dinh}
\address{T.-C. Dinh, L. Kaufmann \& H. Wu - Department of Mathematics,  National University of Singapore - 10, Lower Kent Ridge Road - Singapore 119076}
\email{matdtc@nus.edu.sg}

\author{Lucas Kaufmann}
\address{L. Kaufmann - Center for Complex Geometry - Institute for Basic Science (IBS) - 55 Expo-ro Yuseong-gu Daejeon 34126 South Korea \& Institut Denis Poisson, CNRS, Université d’Orléans, Rue de Chartres B.P.
6759, 45067 Orléans Cedex 2 France. }
\email{lucas.kaufmann@univ-orleans.fr}

\author[1]{Hao Wu}
\email{matwu@nus.edu.sg}

\thanks{This work was supported by the NUS and MOE grants  AcRF Tier 1 R-146-000-248-114, R-146-000-259-114, R-146-000-299-114 and MOE-T2EP20120-0010. L. Kaufmann was supported by the Institute for Basic Science (IBS-R032-D1)}

\begin{document}

\begin{abstract}
We obtain various new limit theorems for random walks on $\SL_2(\C)$ under low moment conditions.  For non-elementary measures with a finite second moment we prove a Local Limit Theorem for the norm cocycle,  yielding the optimal version of a theorem of É. Le Page.  For measures with a finite third moment, we obtain the Local Limit Theorem for the matrix coefficients, improving  a recent result of Grama-Quint-Xiao and the authors, and Berry-Esseen bounds with optimal rate $O(1 / \sqrt n)$ for the norm cocycle and the matrix coefficients. The main tool is a detailed study of the spectral properties of the Markov operator and its purely imaginary perturbations acting on different function spaces. We introduce, in particular, a new function space derived from the Sobolev space $W^{1,2}$ that provides uniform estimates. 
\end{abstract}

\clearpage\maketitle

\thispagestyle{empty}

{\centering\footnotesize \textit{------ \hskip5pt In memory of Professor Nessim Sibony \hskip5pt ------} \par}

\vskip20pt

\noindent\textbf{Keywords:} random walks on Lie groups, random matrices, local limit theorem, Berry-Eseen bound, spectral gap.

\medskip

\noindent\textbf{Mathematics Subject Classification 2010:} \texttt{60B15,60B20,37C30}.

\setcounter{tocdepth}{1}
\tableofcontents

\section{Introduction}

The theory of random walks on Lie groups is a classical and well developed topic. Initiated from the work of Furstenberg and Kesten in the 1960s \cite{furstenberg-kesten,furstenberg}, it was later developed by Guivarc'h, Kifer, Le Page, Raugi, Margulis, Goldsheid and others.  See \cite{bougerol-lacroix} for an account of the classical theory. Even after important progress in the last decades, there are questions whose answers were found only recently, see e.g. \cite{benoist-quint:CLT,grama-quint-xiao}, and others remain unsolved. This topic is still very active and a modern overview can be found in the book \cite{benoist-quint:book}. In this work we focus on the group $G :=\SL_2(\C)$ acting by linear transformations on $\C^2$ and its induced action on the complex projective line $\P^{1}:= \C \P^1$.

The general problem can be described as follows. Let $\mu$ be a probability measure on $G$. Then, $\mu$ induces a random walk on $G$: for $n \geq 1$ we let $$S_n: = g_n \cdots g_1,$$ where the $g_j$'s are independent random elements of $G$ with law given by $\mu$. One also has an induced random walk on $\P^1$: for a point $x \in \P^1$ we look at its trajectory under $S_n$, i.e.\, $S_n \cdot x =  g_n \cdots g_1 \cdot x$. The general goal is to describe the asymptotic behaviour of these random walks and study questions such as the growth and distribution of the norm of $S_n$ or of its coefficients. A remarkable fact is that these non-commutative random processes satisfy analogues of  many of the classical limit theorems for sums of  i.i.d.\ random variables.

The action of a matrix $g \in G$ on $\C^2$ can be described jointly by its action on the set of directions, i.e. $\P^1$, and by its effect on the length of vectors of $\C^2$.  The study of the action on $\P^1$ is done via the so-called stationary measures defined below. In order to study the length of vectors,  we consider the \textit{norm cocycle} defined by $$\sigma(g,x) = \sigma_g(x):= \log \frac{\norm{gv}}{\norm{v}}, \quad \text{for }\,\, v \in \C^2 \setminus \{0\}, \, x = [v] \in \P^1   \, \text{ and }\,  g \in G.$$
We use here the standard euclidean norm on $\C^2$.  Notice that $\|\sigma_g\|_\infty = \log \|g\|$, where $\|g\|$  denotes the operator norm of the matrix $g$. 

When $\mu$ has a finite first moment, that is, $\int_G \log \|g\| \, \diff \mu(g) < \infty$,  its \textit{(upper) Lyapunov exponent} is the finite number defined by 
$$\gamma := \lim_{n \to \infty} \frac1n \E \big( \log\|S_n\| \big)=\lim_{n \to \infty} \frac1n \int \log\|g_n \cdots g_1\| \, \diff \mu(g_1) \cdots \diff \mu(g_n).$$

 We note that the standard moment conditions for random matrices are expressed in terms of the quantity $\log N(g)$,  where $N(g):= \max(\|g\|,  \|g^{-1}\|)$,  but  since $\|g\| = \|g^{-1}\|$ for $g \in \SL_2(\C)$,  we only need to consider here the integrability of $\log \|g\|$.

A fundamental result of Furstenberg-Kesten says that we actually have that $\frac1n \log\|S_n\|$ converges to $\gamma$ almost surely as $n \to \infty$. This is the analogue in this setting of the Strong Law of Large Numbers for sums of i.i.d.\ random variables. 

In order to obtain finer limit theorems and study the action on $\P^1$, some conditions on $\mu$ need to be imposed. We need the support of $\mu$ to be somewhat non-trivial (which leads to the definition of non-elementary measures below) and, as in the case of sums of i.i.d.'s, the random variables must verify some moment conditions.

We say that $\mu$ is \textit{non-elementary} if its support does not preserve a finite subset of $\P^1$ and if the semi-group it generates is not relatively compact in $G$.  This notion corresponds to the more standard assumptions that the support of $\mu$ is proximal and strongly irreducible,  which readily generalizes to the higher dimensional setting,  see \cite{benoist-quint:book}.  Being non-elementary is a weak and easily verifiable condition that holds for generic $\mu$. It is an important result of Furstenberg that a non-elementary measure admits a unique \textit{stationary measure}, also called the \textit{Furstenberg measure}.  This is the only probability measure $\nu$ on $\P^1$ such that $$\int_G g_* \nu \, \diff \mu(g)= \nu.$$

To date, many limit theorems are known both for the sequences $\sigma(S_n,x)$ and $\log \|S_n\|$. However, most of them require the strong condition of \textit{finite exponential moment}, that is, $\int_G \|g\|^\alpha \, \diff \mu(g) < \infty$  for some $\alpha > 0$. These results include the Central Limit Theorem (CLT), the Law of Iterated Logarithm (LIL), the  Local Limit Theorem (LLT), etc. See \cite{benoist-quint:book} for a complete treatment  in the more general setting of reductive groups over local fields.  These results rely on the fundamental work of Le Page \cite{lepage:theoremes-limites}, which studies in detail the Markov operator and its perturbations for measures with a finite exponential moment.  

It is expected that the above exponential moment condition is too strong, since the classical versions of the aforementioned limit theorems all hold under the much weaker condition of finite second moment. A recent breakthrough by Benoist-Quint, based on martingale approximation, led to the proof of the optimal version of the CLT under the second moment condition $\int_G \log^2 \|g\| \diff \mu(g) < \infty$, see \cite{benoist-quint:CLT} and Theorem \ref{thm:CLT} below. See also \cite{cuny-dedecker-jan} for other results under low moment conditions based on their techniques. This kind of method, however, does not apply to other questions, such  as the Local Limit Theorem and sharp Berry-Esseen bounds, treated here, and  the decay of Fourier coefficients studied in \cite{li:fourier,DKW:fourier}, to name a few. For these questions, the spectral analysis of the Markov operator is more effective, and most likely indispensable.

Up to now, the above mentioned approach has not been successful when the moment conditions are weak. The main challenge is to obtain good spectral properties of the Markov operator and its perturbations without the exponential moment hypothesis.  Before this work,  there was no evidence supporting the belief that such good spectral properties should be valid  under low moment conditions.  The challenge is to find a Banach space equipped with a ``good norm'' on which the Markov operator has a spectral gap. Here,  by a ``good norm'' we mean a norm that either dominates the $\|\cdot\|_{\infty}$-norm  or one that allows us to obtain uniform estimates after additional arguments. A recent progress was made by the authors in \cite{DKW:PAMQ}, where the methods of \cite{DKW:IJM} were further developed and allowed us to obtain a spectral gap theorem for the Markov operator under a first moment condition acting on a Sobolev space, see Theorem \ref{thm:DKW-spectral-gap} below.

In this work, we continue to develop the ideas of \cite{DKW:PAMQ} and apply them to prove new limit theorems.

\vskip3pt

Our first main result is the Local Limit Theorem for the norm cocycle under the optimal second moment condition.  We recall that being non-elementary is equivalent to being proximal and strongly irreducible.

\begin{mainthm}[Local Limit Theorem for the norm cocycle] \label{thm:LLT}
Let $\mu$ be a non-elementary probability measure on $G =\SL_2(\C)$. Assume $\mu$ has a finite second moment, that is, $\int_G \log^2 \|g\| \, \diff \mu(g) < \infty$. Then, the associated norm cocycle satisfies the Local Limit Theorem.

More precisely, let $S_n =g_n \cdots g_1$ where the $g_i$ are i.i.d.\ with law $\mu$.  Let $\gamma$ be the Lyapunov exponent associated with $\mu$ and $\nu$ be the corresponding stationary measure. Then, there exists a number $a > 0$ such that for every continuous function $f$ with compact support on $\R \times \P^1$,   we have
\begin{equation} \label{eq:LLT}
\begin{split}
\lim_{n \to \infty} \sup_{(t,x) \in \R \times \P^1} \Big| \sqrt{2 \pi n}\, a \,  \E \Big( f\big( & t + \sigma(S_n, x) - n \gamma , S_n x\big) \Big)  \\ 
- \quad &e^{ -\frac{t^2}{2 a^2 n}} \int_{\R \times \P^1} f(s,  y)\, \diff s \, \diff \nu( y) \Big| = 0.
\end{split}
\end{equation}
\end{mainthm}

The square of the number $a > 0$ above coincides with the variance in the Central Limit Theorem of Benoist-Quint, see Theorem \ref{thm:CLT} below.   It is not difficult to see that Theorem \ref{thm:LLT} can be extended to all functions $f$ that can be approximated in $L^1(\Leb \otimes \nu)$, both from above and below, by continuous functions with compact support. In particular, by considering the indicator function of an interval $[b_1,b_2]$, it follows that, for any $-\infty<b_1<b_2<+\infty$ in $\R$ and  $x \in \P^1$, we have
\begin{equation} \label{eq:LLT-prob}
\lim_{n\to \infty} \sup_{t\in\R} \bigg|\sqrt{2\pi n}\,a \, \mathbf P\Big(  t+\sigma(S_n, x) - n\gamma \in [b_1,b_2] \Big) -e^{ -{ t^2\over 2a^2 n}}(b_2-b_1)\bigg|=0.
\end{equation}

When $\mu$ has a finite exponential moment this result is known since Le Page \cite{lepage:theoremes-limites}. His theorem also holds for matrices of any dimension. See also \cite[Chapter 16]{benoist-quint:book} for the case of random walks with finite exponential moments on more general Lie groups. The second moment condition in Theorem \ref{thm:LLT} is optimal, since this is also the case for the classical LLT for sums of i.i.d.'s. We can also deduce a weak version of Theorem \ref{thm:LLT} for the random variables $\log \|S_n\|$,  see Remark \ref{rmk:LLT-norms} below.

\vskip5pt

Next, we state a related LLT for the coefficients of the random matrices. Dealing with the coefficients is more challenging than studying the norm cocycle or the operator norm. Since the absolute value of a given coefficient can vanish, this  introduces singularities when taking the logarithm. Here, $\lp \cdot , \cdot \rp$ stands for the standard hermitian inner product on $\C^2$.

\begin{mainthm}[Local Limit Theorem for the matrix coefficients] \label{thm:LLT-coeff}
Let $\mu$ be a non-elementary probability measure on $G=\SL_2(\C)$. Assume that $\mu$ has a finite moment of order three, that is, $\int_G \log^3 \|g\| \, \diff \mu(g) < \infty$. Then, the coefficients of $S_n$ satisfy the Local Limit Theorem.

More precisely,  let $\gamma$ be the Lyapunov exponent associated with $\mu$ and $a>0$ be as in Theorem \ref{thm:LLT}.  Then, for any  $v,w\in\C^2 \setminus \{0\}$, $x=[v]$ and every continuous function $f$ with compact support on $\R \times \P^1$,   we have
$$\lim_{n\to \infty} \sup_{t\in\R} \bigg|\sqrt{2\pi n}\,a \, \mathbf E\Big(  f \big( t+\log {|  \lp S_n v, w \rp   |\over \|v\|\|w\|}-n\gamma , S_n x \big) \Big) -  e^{ -\frac{t^2}{2 a^2 n}} \int_{\R \times \P^1} f(s,  y)\, \diff s \, \diff \nu(y) \bigg|=0,$$  uniformly in $x$.
\end{mainthm}

As before,  it is not hard to deduce from the above theorem that \eqref{eq:LLT-prob} holds for $\log {|  \lp S_n v, w \rp   |\over \|v\|\|w\|}$ instead of $\sigma(S_n,x)$. The analogous result in higher dimensions is very recent and it is only known to hold in the case of finite exponential moment.  This was obtained by the authors in the work \cite{DKW:BE-LLT-coeff}, which generalizes a weaker version of the  LLT for the coefficients obtained in  Grama-Quint-Xiao \cite{grama-quint-xiao}.

\vskip7pt

Our methods also yield an optimal Berry-Esseen bound for the norm cocycle. Recall that the classical  Berry-Esseen Theorem says that, for a sequence of i.i.d.'s with a finite moment of order three, the rate of convergence in the Central Limit Theorem is of order $n^{-1 \slash 2}$. The following result is its analogue in our setting.

\begin{mainthm}[Berry-Esseen bound for the norm cocycle] \label{thm:berry-esseen}
Let $\mu$ be a non-elementary probability measure on $G=\SL_2(\C)$. Assume that $\mu$ has a finite moment of order three, that is, $\int_G \log^3 \|g\| \, \diff \mu(g) < \infty$. Let $\gamma$ be the Lyapunov exponent associated with $\mu$ and $a>0$ be as in Theorem \ref{thm:LLT}. Then, there is a constant $C>0$ such that, for any interval $J \subset \R$, any $x \in \P^1$ and all $n \geq 1 $, we have
\begin{equation}\label{eq:berry-esseen}
\bigg| \mathbf P \Big( \frac{\sigma(S_n,x) - n \gamma}{\sqrt n} \in J \Big)  -  \frac{1}{\sqrt{2 \pi} \, a}\int_{J} e^{-\frac{s^2}{2 a^2}} \diff s \bigg| \leq \frac{C}{\sqrt n}.
\end{equation}
\end{mainthm}

Our proof uses the Nagaev-Guivarch'h spectral method, as  presented in \cite{gouezel:spectral-methods}. Before our work, the optimal $O(n^{-1 \slash 2})$ bound was known only under a finite exponential moment condition by a result of Le Page, see \cite{lepage:theoremes-limites,bougerol-lacroix} and also \cite{xiao-grama-liu}. Under the optimal third moment condition, this result was expected, but the best known bound so far was  $O(n^{-1 \slash 4} \sqrt{\log n})$ obtained in \cite{cuny-dedecker-jan} using martingale approximation methods in the spirit of \cite{benoist-quint:CLT}.  After the appearance of the first version of our work,  Cuny-Dedecker-Merlev\`ede-Peligrad \cite{cuny-dedecker-merlevede-peligrad} announced the improvement to a  $O(n^{-1 \slash 2})$ (resp. $O((\log n)^{1/2}n^{-1 \slash 2})$) bound under a fourth (resp.  third) moment condition.  These results are valid in any dimension.  We note that the optimal bound for $2 \times 2$ matrices under a third moment condition given  by Theorem \ref{thm:berry-esseen} is still unknown for matrices of higher order. 

Theorem \ref{thm:berry-esseen} also yields a Berry-Essen bound for $\log \|S_n\|$ with rate $\log n \slash \sqrt n$, see Remark \ref{rmk:berry-esseen-norm}. In the exponential moment case, this is was proven by Xiao-Grama-Liu \cite{xiao-grama-liu:berry-eseen}. Cuny-Dedecker-Merlev\`ede-Peligrad \cite{cuny-dedecker-merlevede-peligrad} recently obtained the same rate under a third moment condition.

\medskip

We also obtain an optimal Berry-Esseen bound for the matrix coefficients under a third moment condition. We note that, in higher dimensions, the optimal $O(1 / \sqrt n)$ bound is only known in the exponential moment case \cite{DKW:BE-LLT-coeff}. An $O(1 / n^{(p-1)/2p})$ bound was recently announced in \cite{cuny-dedecker-merlevede-peligrad-2} under a $p^{th}$ moment condition,  $p \geq 3$.

\begin{mainthm}[Berry-Esseen bound for the matrix coefficients]\label{thm:BE-coeff}
	Let $\mu$ be a non-elementary probability measure on $G=\SL_2(\C)$. Assume that $\mu$ has a finite moment of order three, that is, $\int_G \log^3 \|g\| \, \diff \mu(g) < \infty$. Let $\gamma$ be the Lyapunov exponent associated with $\mu$ and $a>0$ be as in Theorem \ref{thm:LLT}.  Then, there is a constant $C>0$ such that, for any   $v,w\in\C^2 \setminus \{0\}$, any interval $J\subset\R$, and all $n\geq 1$, we have
	$$\bigg| \mathbf P \Big(   \log{ |\lp  S_n v , w \rp | \over \norm{v} \norm{w}} - n \gamma\in \sqrt n J \Big)  -  \frac{1}{\sqrt{2 \pi} \,  a }  \int_{J} e^{-\frac{s^2}{2  a ^2}} \, \diff s \bigg| \leq \frac{C}{\sqrt n}.$$
\end{mainthm}

\medskip

 The proofs of our main theorems partly follow the earlier works \cite{lepage:theoremes-limites,gouezel:spectral-methods,grama-quint-xiao,DKW:BE-LLT-coeff} and, as mentioned before, rely on the study of the perturbed Markov operators $$\oP_\xi u  :=  \int_G e^{i \xi \sigma_g} g^*u  \, \diff \mu(g), \quad \xi \in \R$$
 acting on different spaces of functions on $\P^1$. These perturbations allow us to encode characteristic functions of the random variables we are interested in, and enables the use of Fourier analysis. In this way, the spectral properties of $\oP_\xi$ provide an effective way of obtaining various limit theorems. See \cite{gouezel:spectral-methods} for an overview of this method and \cite{bianchi-dinh} for recent applications in complex dynamics.
 
In this work, we obtain various results about the family of operators $\oP_\xi$ acting on different function spaces. The key properties are the regularity of the family $\xi \mapsto \oP_\xi$ (Propositions \ref{prop:P_t-regularity} and \ref{prop:P_t-logp-regularity}) and the fact that the spectral radius of $\oP_\xi$ acting on the Sobolev space $W^{1,2}$ is smaller than one for all $\xi \neq 0$ (Theorem \ref{thm:P_t-contracting}). Along the way, we introduce in Section \ref{sec:W} a new function space, denoted by $\W$, that mixes the Sobolev norm with a $\log^p$-norm (see Section \ref{subsec:log^p}) and where we still have a spectral gap. It follows from the classical exponential estimate of Moser-Trudinger that the  new norm $\|\cdot\|_\W$ is stronger than the $\Cc^0$-norm (see Theorem \ref{thm:spectral-gap-W}). This allows us to obtain uniform estimates more directly. The main feature of our work is that we require low (and often optimal) moment conditions, whereas the available results on the spectrum of these operators, obtained mainly by Le Page \cite{lepage:theoremes-limites}, all require strong exponential moment conditions.

The spectral method can be used together with the above properties of the operators $\oP_\xi$ to prove an Almost Sure Invariance Principle (ASIP)  for the norm cocycle under a $(2+\varepsilon)$-moment condition.  This approach, however, do not provide the best bounds obtained in \cite{cuny-dedecker-jan} via martingale approximation. Their result also holds under a second moment condition.  See also \cite{cuny-dedecker-merlevede}.

\medskip

It is worth mentioning that our study also applies to random walks on $\SL_2(\R)$, that is, when the random matrices have real coefficients. This corresponds to the particular case where the support of $\mu$ is contained in $\SL_2(\R)$. The case of matrices in $\SL_2(\Z)$ can be studied similarly.  Random matrices in $\GL_2(\C)$ can be rescaled and the study of their products is similar to that of $\SL_2(\C)$-random walks,  treated here, together with that of sums of real  i.i.d.\ random variables, which is standard in classical probability theory. All the above theorems hold for $\GL_2(\C)$ instead of $\SL_2(\C)$, with the condition that the probability measure $\mu$ in $\GL_2(\C)$ is proximal and strongly irreducible (see e.g. \cite[Chapter 4]{benoist-quint:book} for the definition). The proofs are similar to the ones presented here, under some simple modifications. We chose to work in $\SL_2(\C)$ for simplicity.  Other Lie groups closely related to $\SL_2(\C)$, such as $\text{Sp}(2,\C)$, $\text{SO}(3,1)$ etc., can be easily fitted into our framework. We hope that our methods can be generalized to cover more general Lie groups as in \cite{benoist-quint:book} and provide a good spectral theory in the low moment case, which is currently missing in the literature. This would provide new tools for proving limit theorems for the corresponding random walks.  It seems challenging however to find good analogues of the function spaces used here in higher dimensions,  see the comments below.

In the case of $\SL_2(\R)$, further applications of the methods we develop here can be given. In this case, the action on $\P^1$ preserves the real projective line $\R \P^1$, which via a stereographic projection is naturally identified with a circle. It follows that the stationary measure $\nu$ is supported by $\R \P^1$. These observations allow us to study the Fourier coefficients of $\nu$ in the spirit of \cite{li:fourier}, where the spectral properties of $\oP_\xi$ also play a crucial role. We deal with this question in the paper \cite{DKW:fourier}. 
\medskip

 It is natural to ask whether our results can be extended to the case of higher dimensions and higher rank Lie groups.  Under exponential moment conditions,  most of the above results are known in any dimension, see for instance \cite{benoist-quint:book, bougerol-lacroix, cuny-dedecker-merlevede-peligrad-2,  DKW:BE-LLT-coeff,  lepage:theoremes-limites,  xiao-grama-liu:berry-eseen}.  This is mainly due to a good spectral theory in this setting.  Under low moment conditions, however,  such spectral theory is still missing.  Within the complex analytic point of view adopted here, this is reflected by the fact the intrinsic Sobolev norms can be defined on the Riemann surface $\P^1$ (see Subsection \ref{subsec:(1,0)-forms} and Proposition \ref{prop:equinorm}) and  a spectral gap result can be proved for them (Theorem \ref{thm:DKW-spectral-gap}),  but good intrinsic norms adapted to our problem are not yet available in higher dimensional manifolds.  We also note that,  since our spaces use the complex and differentiable structure of $\P^1$ (not only its metric structure) it is hard to transfer our proofs to the case of more general local fields as in \cite{benoist-quint:book}.

\medskip

\noindent \textbf{Organization.} The paper is organized as follows. In Section \ref{sec:preliminary}, we recall simple facts of the theory of random walks on $\SL_2(\C)$ and Fourier analysis.  In particular, we introduce a useful approximation result by functions with compactly supported Fourier transforms.

In Section \ref{sec:markov-op}, we discuss the Markov operator $\oP$ and study its action on the Sobolev space $W^{1,2}$ and on the space $\Cc^{\log^p}$ of  $\log^p$-continuous functions. The main properties are that $\oP$ has a spectral gap in $W^{1,2}$ and acts continuously in $\Cc^{\log^p}$.

Section \ref{sec:perturbed-markov-op} concerns the study of the perturbations $\oP_\xi$ acting on the same function spaces. We show that the family $\xi \mapsto \oP_\xi$ has some regularity in $\xi$ (continuity, Lipschitz continuity, differentiability, etc.) depending on the moment conditions on $\mu$. See Propositions \ref{prop:P_t-regularity} and \ref{prop:P_t-logp-regularity} below. This allows us to apply the theory of perturbations of linear operators \cite{kato:book} and obtain good spectral properties of $\oP_\xi$ for small values of $\xi$ (cf. Corollary \ref{cor:P_t-decomp}). 
 
 In Section \ref{sec:spec-Pt}, we study the spectrum of $\oP_\xi$ for \textit{all} non-zero (possibly large) values of $\xi$. The main result is that the spectral radius of $\oP_\xi$ is strictly smaller than one for all $\xi \neq 0$,  cf. Theorem \ref{thm:P_t-contracting}. This is a crucial ingredient in the proof of the LLTs and is related to the absence of some strong invariance properties of the cocycle $\sigma(g,x)$, see Proposition \ref{prop:non-arithmetic}.
 
 In Section \ref{sec:W}, we introduce the function space $\W$, that lies between $W^{1,2} \cap \Cc^{\log^{p-1}}$ and $W^ {1,2}\cap \Cc^0$ and where some good spectral properties from $W^{1,2}$ persist. This automatically yields uniform estimates, which are needed in the proofs of Theorems \ref{thm:LLT} and \ref{thm:LLT-coeff}. 
 
 Section \ref{sec:LLT-norm} is devoted to the proof of Theorem \ref{thm:LLT}. We follow Le Page's strategy, which works here thanks to the results from Sections \ref{sec:markov-op} to \ref{sec:W}. In Section \ref{sec:LLT-coeff}, we prove Theorem \ref{thm:LLT-coeff} by following the approach from \cite{DKW:BE-LLT-coeff}. Theorem \ref{thm:berry-esseen} is proved in Section \ref{sec:berry-esseen} using Nagaev-Guivarc'h's approach and the proof of Theorem \ref{thm:BE-coeff} is given in Section \ref{sec:berry-esseen-coeff}.
 
 An appendix is included in the end of the text, where some technical results about different notions of convergence and pointwise values of functions in $W^{1,2}$ are discussed.

\medskip

\noindent\textbf{Notations:} Throughout this paper, the symbols $\lesssim$ and $\gtrsim$ stand for inequalities up to a multiplicative constant.  The dependence of these constants on certain parameters (or lack thereof),  if not explicitly stated,  will be clear from the context.  We denote by $\mathbf E$ the expectation and $\mathbf P$ the probability.

\medskip

 \noindent\textbf{Acknowledgments:} The authors would like to thank the anonymous referees whose remarks helped us improve the presentation and clarify some of the arguments.

\section{Preliminary results} \label{sec:preliminary}
\subsection{Norm cocycle,  Lyapunov exponent and the CLT}
We start with some basic results and notations. We refer to \cite{bougerol-lacroix} for the proofs of the results described here.

 Let $\mu$ be a probability measure on $G=\SL_2(\C)$.   For $n \geq 1$,  we define the convolution measure by  $\mu^{*n} := \mu * \cdots * \mu$ ($n$ times) which is the push-forward of the product measure $\mu^{\otimes n}$ on $G^n$ by the map $(g_1, \ldots, g_n) \mapsto g_n \cdots g_1$.  If $g_j$ are i.i.d.  with law $\mu$ then $\mu^{*n} $ is the law of $S_n = g_n \cdots g_1$.

Denote by $\norm{g}$ the operator norm of the matrix $g$. Note that since $g \in \SL_2(\C)$ we have that $\|g\| \geq 1$. We say that  $\mu$  has a \textit{finite exponential moment} if  $\int_G \|g\|^\alpha \, \diff \mu(g) < \infty$  for some $\alpha > 0$, and that $\mu$ has a \textit{finite moment of order} $p >0$ (or a \textit{finite $p^{th}$ moment}) if  $\int_G \log^p \|g\| \, \diff \mu(g) < \infty$.

For measures with finite first moment we define the  \textit{(upper) Lyapunov exponent} by $$
\gamma := \lim_{n \to \infty} \frac1n \E \big( \log\|S_n\| \big)=\lim_{n \to \infty} \frac1n \int \log\|g_n \cdots g_1\| \, \diff \mu(g_1) \cdots \diff \mu(g_n)$$
and the \textit{norm cocycle} by $$\sigma(g,x) = \sigma_g(x):= \log \frac{\norm{gv}}{\norm{v}}, \quad \text{for }\,\, v \in \C^2 \setminus \{0\}, \, x = [v] \in \P^1   \, \text{ and } g \in G.$$

We say that $\mu$ is \textit{non-elementary} if its support does not preserve a finite subset of $\P^1$ and if the semi-group it generates is not relatively compact in $\SL_2(\C)$.  A non-elementary measure admits a unique \textit{stationary measure}, that is, a probability measure $\nu$ on $\P^1$ satisfying $$\int_G g_* \nu \, \diff \mu(g)= \nu.$$

The stationary measure,  the norm cocycle and the Lyapunov exponent are related by Furstenberg's formula $$\gamma = \int_{\P^1} \int_G \sigma_g(x) \diff \mu(g) \diff \nu(x).$$ Moreover, one can show that for non-elementary measures we always have $\gamma > 0$.

\textit{Cartan's decomposition} says that every matrix $g \in \SL_2(\C)$ can be written as
\begin{equation} \label{eq:cartan-decomposition}
g=k_g a_g \ell_g, \quad \text{ where } \, \, k_g,\ell_g \in \text{SU}(2) \,\, \text{ and } \,\, a_g=\left(\begin{smallmatrix} \lambda & 0 \\ 0 & \lambda^{-1} \end{smallmatrix}\right) \,\, \text{ for some } \,\, \lambda \geq 1.
\end{equation}
Observe that $\|g\| = \|g^{-1}\| = \lambda$.

\vskip5pt

Recall the \textit{Central Limit Theorem} of Benoist-Quint \cite{benoist-quint:CLT}. See  Dinh-Kaufmann-Wu \cite{DKW:PAMQ} for a proof using the spectral gap of the Markov operator.

\begin{theorem}  \label{thm:CLT}
Let $\mu$ be a probability measure on $G=\SL_2(\C)$. Assume that $\mu$ is non-elementary and has a finite second moment, i.e.\ $\int_G \log^2 \|g\| \, \diff \mu(g) < \infty$. Let $\gamma$ be the Lyapunov exponent of $\mu$. Then there exists a number $a >0$ such that for any $x \in \P^1$
\begin{equation} \label{eq:CLT}
\frac{1}{\sqrt n} \big( \sigma(g_n \cdots g_1,x) - n \gamma \big) \longrightarrow \cali N(0;a^2) \quad \text{ in law},
\end{equation}
where $\cali N(0;a^2)$ is the centered normal distribution with variance $a^2$.
\end{theorem}

Recall that $\P^1$ is equipped with a natural distance given by 
\begin{equation} \label{eq:distance-def}
d(x,y) : = \frac{| \det(v,w)\,  |}{\|v\| \|w\|}, \quad \text{where} \quad v,w \in \C^2 \setminus \{0\}, \, x = [v], \,\, y = [w] \in \P^1.
\end{equation}

We have that  $(\P^1, d)$ has diameter one and that $\text{SU}(2)$ acts by holomorphic isometries. We will denote by $\D(x,r)$ the associated open disc of center $x$ and radius $r$ in $\P^1$.

\subsection{Fourier transforms}

The proofs of our main theorems will depend on some Fourier analysis. We gather here some basic facts and simple results on the Fourier transform that will be useful later.

Recall that the Fourier transform of an integrable function $f$, denoted by $\widehat f$, is defined by
$$\widehat f(\xi):=\int_{-\infty}^{+\infty}f(u)e^{-i u\xi} \diff u$$ 
and the inverse Fourier transform is $$f(u)={1\over {2\pi}}\int_{-\infty}^{+\infty} \widehat f (\xi) e^{ i u\xi} \diff \xi.$$
With these definitions, the Fourier transform of $\widehat f(\xi)$ is $2\pi f(-u)$ and the convolution operator satisfies 
\begin{equation} \label{eq:convolution-fourier-product}
\widehat{f_1*f_2}=\widehat f_1\cdot \widehat f_2.
\end{equation}

\subsection{Approximation} \label{subsec:regularization} We will often have to approximate functions above and below by functions whose Fourier transforms are compactly supported. For that end, we introduce here a useful approximation procedure. See \cite{grama-condi} for a related approach.

We start with the following elementary lemma.

\begin{lemma} \label{l:vartheta}
There is a smooth strictly positive even function $\vartheta$ on $\R$ with $\int_\R \vartheta(u) \diff u=1$ such that its Fourier transform $\widehat\vartheta$ is  a smooth function supported by $[-1,1]$.
\end{lemma}
\begin{proof}
Choose a  smooth even function $\widehat\chi$ supported by $[-1/2,1/2]$ which is not identically zero.
Define $\widehat\vartheta(\xi):=(\widehat\chi*\widehat\chi)(\xi)\cdot (\sqrt{2\pi} e^{-\xi^2/2})$ and let $\vartheta$ its inverse Fourier transform. Clearly, the function $\widehat\vartheta(\xi)$ is  smooth, even and supported on $[-1,1]$. It follows that $\vartheta$ is smooth and integrable.  We have 
$$\vartheta(u)=\chi^2*e^{-u^2/2}$$
which is clearly positive and even.  Multiplying it by a positive constant gives us a function satisfying $\int_\R \vartheta(u) \diff u=1$.
\end{proof}

Let $\vartheta$ be as in the preceding lemma and define for $0<\delta\leq 1$,
$$\vartheta_\delta(u):=\delta^{-2}\vartheta(u/\delta^2).$$
Observe that  $\widehat{\vartheta_\delta}(\xi)=\widehat \vartheta(\xi \delta^2)$. Therefore,  $\supp(\widehat{\vartheta_\delta})$  is contained in  $[-\delta^{-2},\delta^{-2}]$ and $\|\widehat\vartheta_\delta\|_{\Cc^1}$ is bounded by a constant independent of $0<\delta\leq 1$. Notice that, since $\widehat\vartheta$ is smooth and supported in $[-1,1]$, $\vartheta$ has fast decay at infinity.

For any locally bounded function $\psi$ on $\R$ and $0<\delta\leq 1$, define
\begin{equation} \label{eq:def-phi-delta-sup}
\psi^+_{[\delta]}(u):=\sup_{|u'-u|\leq \delta} \psi(u') .
\end{equation}
We have the following simple lemma whose proof is left to the reader, see also \cite[Lemma 5.2]{grama-condi}.

\begin{lemma}\label{l:conv-Fourier}
	Let $\psi$ be a continuous integrable non-negative function on $\R$. Then, for any $0<\delta\leq 1$,  there exists a positive  constant $c_\delta = O(\delta)$ independent of $\psi$,  such that
	$$ \psi(u)\leq (1+c_\delta) \psi^+_{[\delta]} *\vartheta_\delta(u) \quad\text{for all}\quad u\in \R.$$
\end{lemma}

We can now state our main approximation result.

\begin{lemma} \label{lemma:conv-Fourier-2} 
Let $\psi$ be a continuous function with support in a compact set $K$ in $\R$.  Assume that $\|\psi\|_\infty \leq 1$. Then, for every  $0< \delta \leq 1$ there exist smooth functions $\psi^\pm_\delta$  such that $\widehat {\psi^\pm_\delta}$ have supports in $[-\delta^{-2},\delta^{-2}]$,  $$\psi^-_\delta \leq\psi\leq \psi^+_\delta,\quad \lim_{\delta \to 0} \psi^\pm_\delta =\psi    \quad \text{and} \quad  \lim_{\delta \to 0} \big \|\psi^\pm_\delta -\psi \big \|_{L^1} = 0.$$ 
Moreover,  $\norm{\psi_\delta^\pm}_\infty$, $\norm{\psi_\delta^\pm}_{L^1}$ and $\|\widehat{\psi^\pm_\delta}\|_{\Cc^1}$ are bounded by a constant which only depends on $K$.
\end{lemma}

\begin{proof}
Observe that $\psi=\max(\psi,0)-\max(-\psi,0)$ and $\|\max(\pm\psi,0)\|_\infty \leq 1$. Therefore, for simplicity, we can assume that $\psi\geq 0$.

Let $\vartheta$, $\vartheta_\delta$, $\psi^\pm_{[\delta]}$ and $c_\delta$ be as above. Since $\vartheta$ is strictly positive,  $\psi$ is supported by $K$ and $\|\psi\|_\infty \leq 1$,  one can find a constant $c>0$ depending only on $K$ such that $c\vartheta-\psi\geq 0$. 
Define
$$\psi^+_\delta(u) :=(1+c_\delta) \psi_{[\delta]}^+ *\vartheta_\delta(u) \quad\text{and}\quad \psi^-_\delta(u) :=c\vartheta(u)-(1+c_\delta)(c\vartheta-\psi)_{[\delta]}^+*\vartheta_\delta(u).$$
By Lemma \ref{l:conv-Fourier}  applied to $\psi$ and $c\vartheta-\psi$, we have $\psi^-_\delta \leq \psi\leq \psi^+_\delta$.  It is clear that $\psi^\pm_\delta$ converge pointwise to $\psi$ as $\delta$ tends to zero and the $L^1$ convergence follows from well-known properties of convolutions.  By the convolution formula \eqref{eq:convolution-fourier-product} and the fact that $\widehat\vartheta_\delta$  is supported by $[-\delta^{-2},\delta^{-2}]$,  it follows that  $\widehat{\psi^\pm_\delta}$ have support in $[-\delta^{-2},\delta^{-2}]$.  This proves the first assertions.

We now prove the last estimates. From the definition of convolution, one can easily see that $\norm{\psi_\delta^\pm}_\infty$ and $\norm{\psi_\delta^\pm}_{L^1}$ are bounded by a constant only depending on $c$, which depends only on $K$. For the last assertion, we note that,  since $|\psi|\leq 1$ and $\psi$ has compact support in $K$, we have that $|\psi^+_{[\delta]}|\leq 1$ and $\psi^+_{[\delta]}$ is supported by $K+[-1,1]$ for all $0< \delta \leq 1$. It follows that $\|\widehat{\psi^+_{[\delta]}}\|_{\Cc^1}$ is bounded by a constant depending only on $K$. It is then easy to see that $\|\widehat{\psi^+_{\delta}}\|_{\Cc^1}$ is bounded by a constant which only depends on $K$. We use here that $\|\widehat\vartheta_\delta\|_{\Cc^1}$ is bounded by a constant independent of $0<\delta\leq 1$.  Finally, since $\widehat\vartheta$ has compact support, $\vartheta$ has fast decay at infinity and the same holds for $c\vartheta-\psi$ and hence for  $(c\vartheta-\psi)^+_{[\delta]}$, uniformly in $\delta$.  We deduce that the Fourier transform of $(c\vartheta-\psi)^+_{[\delta]}$ has  a $\Cc^1$-norm bounded by a constant which only depends on $K$. 
Thus, by the same argument as above we get that $\|\widehat{\psi^-_{\delta}}\|_{\Cc^1}$ is bounded by a constant which only depends on $K$.  
\end{proof}

\section{The Markov operator $\oP$} \label{sec:markov-op}

Let $\mu$ be a non-elementary probability measure on $G=\SL_2(\C)$. Consider the associated Markov operator 
\begin{equation} \label{eq:makov-op}
\oP u := \int_G g^* u \, \diff \mu(g) 
\end{equation}
acting on functions on $\P^1$. Observe that, for every integer $n \geq 1$, the iterate $\oP^n$ corresponds to the Markov operator associated with the convolution power $\mu^{\ast n}$.

In this work, we'll study the action of $\oP$ on different spaces of functions and differential forms. We always assume $\mu$ to be non-elementary. However, the moment conditions required on $\mu$ will depend on the situation.

\subsection{Action on the Sobolev space and on $(1,0)$-forms} \label{subsec:(1,0)-forms}

The main result of \cite{DKW:PAMQ} is that, under a finite first moment condition, $\oP$ acts continuously on $W^{1,2}$ with a spectral gap. We now give more details. Consider first the space $L^2_{(1,0)}$ of $(1,0)$-forms on $\P^1$ with $L^2$ coefficients equipped with the intrinsic norm $$\|\phi\|_{L^2} : = \Big( \int_{\P^1} i \phi \wedge \overline \phi \,  \Big)^{1 \slash 2}.$$

Above,  $i \phi \wedge \overline \phi$ defines a positive measure on $\P^1$ which is then integrated over $\P^1$.  More precisely,  if in a local coordinate $z$ we have $\phi = f(z,\overline z) \, \diff z$ for an $L^2$-function $f$,  then $i \phi \wedge \overline \phi =  |f(z,\overline z)|^2 \,  i \diff z \wedge \diff \overline{ z}$.  One can easily check by using the change of variables formula that the integral of $i \phi \wedge \overline \phi$ is independent of the chosen holomorphic coordinate.  In particular, the above integral can be globally defined and does not require the choice of a background metric or measure on $\P^1$.

Then, we  define the Sobolev space $W^{1,2}$ as the space of complex-valued measurable functions on $\P^ 1$ with finite  $\|\cdot\|_{W^{1,2}}$-norm, where
$$\|u\|_{W^{1,2}} := \Big| \int_{\P^1}  u \, \omegaFS \Big| + \frac12 \|\partial u\|_{L^2} + \frac12 \|\partial \overline u\|_{L^2}$$ 
and $\omegaFS$ stands for the Fubini-Study form on $\P^1$,  that is,  the unique $\U(2)$ invariant volume form on $\P^1$ of total volume one.  Alternatively,  via the identification between $\P^1$ and the two-dimensional sphere,   $\omegaFS$ corresponds to the normalized volume form associated with the round metric.  As before,  the norms $\|\partial u\|_{L^2}$ and $ \|\partial \overline u\|_{L^2}$ are intrinsic and do not depend on the choice of a background metric or measure on $\P^1$.

 Observe that, when $u$ is real-valued, the last two terms in the definition of $\| u \|_{W^{1,2}}$ above add up to $\|\partial u\|_{L^2}$ and we recover the original Sobolev norm used in \cite{DKW:PAMQ}. In the present work, we need to work with complex-valued functions. We can easily verify that all the results from \cite{DKW:PAMQ} also apply to complex-valued functions.  In particular,  we stress the choice of the Fubini-Study form in the above definition is not important,  as different  smooth measures on $\P^1$  can be used and yield equivalent norms.    We can also replace the quantity $\big| \int_{\P^1}  u \, \omegaFS \big|$ in the  definition by $L^1$ or $L^2$ norms,  as the following result shows.

 \begin{proposition} [\cite{DKW:PAMQ}] \label{prop:equinorm}
		Let $U$ be a non-empty open subset of $\P^1$. Then the following norms on $W^{1,2}$ are equivalent to the norm $\|\cdot\|_{W^{1,2}}$.

\begin{enumerate}
        \item $\|u\|_1 :=  \int_{\P^1} |u| \, \omegaFS  +  \frac12 \|\partial u\|_{L^2} + \frac12 \|\partial \overline u\|_{L^2}$ \medskip
        \item $ \|u\|_2 :=  \Big(\int_{\P^1} |u|^2 \, \omegaFS \Big)^{1/2} +  \frac12 \|\partial u\|_{L^2} + \frac12 \|\partial \overline u\|_{L^2}$ \medskip
        \item $\|u\|_3 := |\int_U u \, \omegaFS|+  \frac12 \|\partial u\|_{L^2} + \frac12 \|\partial \overline u\|_{L^2}$ \medskip
        \item $\|u\|_4 := \int_U |u| \,  \omegaFS+  \frac12 \|\partial u\|_{L^2} + \frac12 \|\partial \overline u\|_{L^2}$.
\end{enumerate}
 \end{proposition}
 
We refer to the Appendix and \cite{dinh-sibony:decay-correlations, vigny:dirichlet,DKW:PAMQ,dinh-marinescu-vu} for further properties of $W^{1,2}$ space and its higher dimensional analogues.

 Recall that the essential spectrum of an operator is the subset of the spectrum obtained by removing its isolated points corresponding to eigenvalues of finite multiplicity.
The essential spectral radius $\rho_{\rm ess}$ is then the radius of the smallest disc centred at the origin which contains the essential spectrum.

The following spectral gap theorem is the main result of \cite{DKW:PAMQ}. We note that it also holds under the weaker condition of finite moment of order $\frac12$, see Remark \ref{rem:1/2-moment} below. 

\begin{theorem}[\cite{DKW:PAMQ}] \label{thm:DKW-spectral-gap}
Let $\mu$ be a non-elementary probability measure on $G=\SL_2(\C)$. Assume that $\mu$ has a finite first moment, i.e.,  $\int_G \log \|g\| \, \diff \mu(g) <\infty$. Then $\oP$ defines a bounded operator on $W^{1,2}$.  Moreover, $\oP$ has a spectral gap, that is,  $\rho_{\rm ess}(\oP)<1$ and $\oP$ has a single eigenvalue of modulus $\geq 1$ located at $1$. It is an isolated eigenvalue of multiplicity one.

If $\nu$ denotes the unique $\mu$-stationary measure, then the linear functional $\oN u := \int_{\P^1} u \, \diff \nu$ defined for smooth $u$ extends to $W^{1,2}$ as a continuous linear functional with respect to both the weak and strong topologies of $W^{1,2}$. Moreover, $\|\oP^n - \oN\|_{W^{1,2}} \leq c \rho^n$ for some constants $c > 0$ and $0< \rho < 1$.
\end{theorem}

We stress that, for random matrices of higher order, the only known general spectral gap result is the one of Le Page \cite{lepage:theoremes-limites}, which requires an exponential moment condition. 

The above theorem is a consequence of the following contraction result, which does not require any moment condition on $\mu$. 

\begin{proposition}[\cite{DKW:PAMQ}] \label{prop:L^2_1,0-contraction}
Let $\mu$ be a non-elementary probability measure on $G=\SL_2(\C)$. Then the operator $\oP:L^2_{(1,0)} \to L^2_{(1,0)}$ defined by $\oP \phi := \int_G g^* \phi \, \diff \mu(g)$ is bounded and, after replacing $\mu$ by some convolution power $\mu^{\ast N}$ if necessary,  it satisfies $\|\oP\|_{L^2_{(1,0)}} < 1$.
\end{proposition}

The following equidistribution result is a consequence of Theorem \ref{thm:DKW-spectral-gap}. We note that the moment condition in the next theorem can be relaxed to a $(\frac12 + \varepsilon)$-moment condition,  see Remarks \ref{rem:1/2-moment} and  \ref{rem:holder-equidistribution} below.

\begin{theorem}[\cite{DKW:PAMQ}] \label{thm:equi-dis}
		Let $\mu$ be a non-elementary probability measure on $G=\SL_2(\C)$. Assume that $\mu$ has a finite $(1+\varepsilon)$-moment for some $\varepsilon>0$,  i.e.,  $\int_G \log^{1+\varepsilon} \|g\| \, \diff \mu(g) <\infty$. Let $\nu$ be the associated stationary measure. Then, for every $u \in \Cc^1(\P^1)$, we have 
		$$\Big\| \oP^n u - \int_{\P^1} u \,\diff \nu\Big\|_\infty \leq C \lambda^n\norm{u}_{\Cc^1}$$ for some constants $C>0$ and $0<\lambda<1$ independent of $u$.
\end{theorem}

The next regularity property of the stationary measure follows from the work of Benoist-Quint \cite{benoist-quint:CLT}.  The result is valid in any dimension and is a simple consequence of the proof of \cite[Proposition 4.5]{benoist-quint:CLT}  after noting that, from the proof of \cite[Theorem 2.2]{benoist-quint:CLT},  the constants $C_n$ appearing there satisfy $\lim_{n\to \infty} n^{p-1} C_{n,\epsilon}=0$.

\begin{proposition}\label{regularity}
Let $\mu$ be a non-elementary probability measure on $G=\SL_2(\C)$. Assume that $\int_G \log^p \|g\| \, \diff \mu(g) <\infty$ for some $p >1$. Let $\nu$ be the associated stationary measure. Then $$\nu\big(\D(x,r)\big)= o\big(1/|\log r|^{p-1}\big) \quad\text{as}\quad r \to 0 \quad\text{uniformly in } \, x\in \P^1.$$
\end{proposition}

We note that the exponent $p-1$ in the above result was obtained by Benoist-Quint, while in \cite{DKW:PAMQ} the authors were able to obtain another positive exponent that has not been computed explicitly.

\medskip

Functions in $W^{1,2}$ satisfy the following fundamental exponential integrability property, known as \textit{ Moser-Trudinger's estimate}, see \cite{moser:trudinger}:  there exist constants $\alpha >0$ and $C>0$ such that
\begin{equation} \label{eq:moser-trudinger-estimate}
\int_{\P^1} e^{\alpha |u|^2} \, \omegaFS < C, \quad \text{ for all } u \in W^{1,2} \,   \text{ such that } \,\,  \|u\|_{W^{1,2}} \leq 1.
\end{equation}

We record here the following consequence that  will be frequently used in the sequel. Its proof is contained  in  the proof of Proposition 2.3 in \cite{DKW:PAMQ} and is a direct application of \eqref{eq:moser-trudinger-estimate}.

\begin{lemma} \label{lemma:g^*u-L^2}
Let $u \in W^{1,2}$ be such that $\|u\|_{W^{1,2}} \leq 1$ and let $g \in G$. Then, there exists a constant $\kappa > 0$ independent of $u$ and $g$ such that $\|g^*u\|_{L^2} \leq \kappa (1 + \log^{1 \slash 2} \|g\|)$.
\end{lemma}

\begin{remark} \label{rem:1/2-moment} Notice that in \cite{DKW:PAMQ} the stated bound is $\log \|g\|$, while its proof actually gives the better bound $\log^{1 \slash 2} \|g\|$ stated above. In particular, our spectral gap theorem (Theorem \ref{thm:DKW-spectral-gap}) also holds under the weaker moment condition $\int_G \log^{1 \slash 2} \|g\| \, \diff \mu(g) < \infty$.
\end{remark}

\subsection{Action on $\text{log}^p$-continuous functions} \label{subsec:log^p}

For $p>0$ we introduce the following semi-norm on functions on $\P^1$:
$$[u]_{\log^p} := \sup_{x \neq y \in \P^1} \big|u(x) - u(y)\big| \big(\log^\star d(x,y)\big)^p,$$
where from now on we'll use the notation $\log^\star t: = 1 + |\log t \, |$ for $t > 0$. Here, $d$ denotes the distance defined in \eqref{eq:distance-def}.

We say that $u$ is \textit{$\log^p$-continuous} if $[u]_{\log^p} < + \infty$. We denote by $\Cc^{\log^p}$ the space of $\log^p$-continuous functions and equip it with the norm $$\|u\|_{\log^p}: = \|u\|_{\infty} + [u]_{\log^p}.$$

\begin{lemma} \label{lemma:log^p-product}
Let $u$ and $v$ be  $\log^p$-continuous functions on $\P^1$. Then $uv$ is also $\log^p$-continuous and $[uv]_{\log^p} \leq \|u\|_\infty [v]_{\log^p} + [u]_{\log^p} \|v\|_\infty$. In particular, $\|uv\|_{\log^p} \leq \|u\|_\infty \|v\|_{\log^p} + \|u\|_{\log^p} \|v\|_\infty$.
\end{lemma}

\begin{proof}
For any $x,y \in \P^1$ we have
\begin{align*}
|u(x)v(x) - u(y)v(y)| &= |u(x)(v(x) - v(y)) - (u(y) - u(x)) v(y)| \\ &\leq \|u\|_\infty\cdot |v(x) - v(y)| + |u(y) - u(x)| \cdot \|v\|_\infty \\ & \leq  \big( \|u\|_\infty [v]_{\log^p} + [u]_{\log^p} \|v\|_\infty \big) \, (\log^\star d(x,y))^{-p}.
\end{align*}

The lemma follows.
\end{proof}

\begin{lemma} \label{lemma:norm-g^*u}
Let $u$ be a $\log^p$-continuous function on $\P^1$ and $g \in \SL_2(\C)$. Then,  there is a constant $C>0$ independent of $g$ and $u$ such that $[g^* u]_{\log^p} \leq C (1 + \log^ p \|g\|) [u]_{\log^p}$.
\end{lemma}
\begin{proof}
First of all, we notice that the Lipschitz norm of the map induced by $g$  is bounded by $\|g\|^2$, that is,
\begin{equation} \label{eq:lipschitz-g}
d(gx,gy) \leq  \|g\|^2 d (x,y),  \quad \text{ for all } x,y \in \P^ 1.
\end{equation}
This follows from \cite[III.4.2]{bougerol-lacroix} and is also valid in higher dimension.

In order to prove the desired estimate,  we may assume that  $[u]_{\log^p} = 1$. We need to show that $$|u(gx) - u(gy)| \cdot (\log^\star d (x,y))^p \leq C (1+ \log^ p \|g\|).$$

\noindent \textbf{Case 1:} $d(x,y) \geq \frac{1}{\|g\|^4}$.  It follows that $$(\log^\star d(x,y))^p \leq (\log^\star ( 1 \slash \|g\|^4))^p \leq c_1 (1 + \log^p \|g\|)$$ for some constant $c_1>0$. Since $[u]_{\log^p} = 1$, $|u(gx) - u(gy)|$   is bounded by a constant independent of $g$. Therefore, for some constants $c>0$ and $c_2>0$, we have that
\begin{align*}
|u(gx) - u(gy)| \cdot (\log^\star d (x,y))^p \leq c \, (\log^\star d (x,y))^p \leq  c_2  (1+\log^ p \|g\|).
\end{align*}
 
\noindent \textbf{Case 2:} $d(x,y) \leq \frac{1}{\|g\|^4}$. Then $\log \|g\|^2 = \frac12 \log \|g\|^4 \leq - \frac12 \log d(x,y)$. Hence
$$-\log (\|g^2\| d(x,y)) = - \log \|g\|^2  -\log d(x,y)  \geq - \frac12 \log d(x,y).$$

Using (\ref{eq:lipschitz-g}) one gets
\begin{align*}
|u(gx) - u(gy)| \cdot (\log^\star d (x,y))^p \leq (\log^\star d (gx,gy))^{-p}  (\log^\star d (x,y))^p \\ \leq (\log^\star (\|g\|^2 d (x,y) ))^{-p}  (\log^\star d (x,y))^p \leq \frac{ (1-\log d (x,y))^p}{(1-\frac12 \log d (x,y))^p}
\end{align*}
and the last quantity is bounded by a constant independent of $g$, $u$ $x$ and $y$. In particular, it is bounded by $C(1 + \log^ p \|g\|)$ for $C>0$ large enough. This finishes the proof.
\end{proof}

\begin{proposition} \label{prop:Pu-logp}
Let $p>0$ and let $u$ be a $\log^p$-continuous function on $\P^1$. Assume that $\mu$ has a finite moment of order $p$. Then, there is a constant $C>0$,  independent of $\mu$,  such that $$[\oP u]_{\log^p} \leq C \big(1+M_p(\mu)\big) [u]_{\log^p}, \quad\text{where}\quad M_p(\mu):= \int_G \log^p \|g\| \, \diff \mu(g).$$ In particular, $\oP$ defines a bounded operator on $ \Cc^{\log^p}$.
\end{proposition}
\begin{proof}
The first statement follows directly from Lemma \ref{lemma:norm-g^*u} and the triangle inequality. The last statement follows from the first one and the fact that $\|\oP u\|_{\infty} \leq \|u\|_{\infty}$ for every $u \in \Cc^0$.
\end{proof}

The next results are crucial for us, as they allow us to obtain pointwise estimates from a simultaneous control on the Sobolev and the $\log^p$-norms.

\begin{proposition} \label{prop:moser-trudinger-logp}
Let $p > 0$. Let $\mathcal F$ be a bounded subset of  $W^{1,2}$. Then, there exists a constant $A=A(\mathcal F) >0$ such that for every  function $u \in \mathcal F$ with $[u]_{\log^p} \leq M$, $M \geq 1$,  we have that $$\|u\|_\infty \leq AM^{\frac{1}{2p}}.$$
\end{proposition}

\begin{proof}
We can assume that $\mathcal F$ is contained in the unit ball of $W^{1,2}$. Suppose the conclusion is false. Then, for every constant $A>1$, we can find  a function $u \in \mathcal F$ and a  point $x$  such that $|u(x)| \geq  AM^{\frac{1}{2p}}$. 
Fix a constant $A$ large enough. We claim that there is a constant $0<r<1$ such that $|u| \geq  \frac{1}{2} A (- \log r)^{1 \slash 2}$ on the disc $\D(x,r)$ of radius $r$ centered at $x$. 

In order to prove the claim, let $r := e^{- M^{1 \slash p}} \leq 1 \slash e$ and consider an arbitrary point $y \in \D(x,r)$. Then, we have\begin{align*}
|u(y)| &\geq - |u(y) - u(x)| + |u(x)| \geq - M\big(  \log^\star d(y,x)\big)^{-p} +  A M^{\frac{1}{2p}} \\
&\geq -1 +   A M^{\frac{1}{2p}} \geq \frac{1}{2}  A M^{\frac{1}{2p}}  =  \frac{1}{2} A( -\log  r)^{1 \slash 2},
\end{align*}
which proves the claim.

From  Moser-Trudinger's estimate \eqref{eq:moser-trudinger-estimate} there exist $\alpha > 0$ and $C> 0$ independent of $u$ such that $\int_{\P^1} e^{\alpha |u|^2} \omegaFS \leq C$. On the other hand, the above lower bound for $u$ gives that $$\int_{\DD(a,r)} e^{\alpha |u|^2} \omegaFS \gtrsim r^{2 - \frac{\alpha A^2}{4}} \geq e^{-2 + \frac{\alpha A^2}{4}}.$$

 Since  $A$ can be taken arbitrarily large, we get a contradiction. This ends the proof of the proposition.
\end{proof}

\begin{corollary} \label{cor:logp-pre-equidistribution}
Let $p > 1 \slash 2$. Let $u_n$, $n \geq 1$, be a sequence of continuous functions. Assume that there are constants $c_1,c_2 > 0$, $0<\lambda<1$ and $q > 0$ such that $\|u_n\|_{W^{1,2}} \leq c_1 \lambda^n$ and $[u_n]_{\log^p} \leq c_2 n^q$. Then, there are constants $c > 0$ and $0< \tau < 1$ depending on $c_1,c_2,\lambda,p$ and $q$ such that  $$\|u_n \|_\infty \leq c \tau^n \quad \text{ for every } n \geq 1.$$ 
\end{corollary}

\begin{proof}
Let $\widetilde u_n := \lambda^{-n} u_n$. Then the family $\{\widetilde u_n : n \geq 1 \}$ is bounded in $W^{1,2}$. By  hypothesis  $[\widetilde u_n]_{\log^p} \leq c_2 \lambda^{-n} n^q$ and  from Proposition \ref{prop:moser-trudinger-logp}, there is a constant $A >0$ such that $\|\widetilde u_n\|_\infty \leq A c_2^{1 \slash (2p)}  n^{q \slash (2p)} \lambda^{-n \slash (2p)}$. This gives $$\| u_n\|_\infty \leq \lambda^n A c_2^{1 \slash (2p)}  n^{q \slash (2p)} \lambda^{-n \slash (2p)} = A c_2^{1 \slash (2p)} n^{q \slash (2p)} \lambda^{(1-\frac{1}{2p}) n}.$$

Since $p > 1 \slash 2$ one can find constants $c>0$ and $0<\tau<1$ such that  the last quantity is bounded by $c \tau^n$. This finishes the proof.
\end{proof}

When applied to $u_n = \oP^n u$ the above result gives an equidistribution theorem for  $\log^p$-continuous observables in $W^{1,2}$, improving Theorem \ref{thm:equi-dis} above.

\begin{corollary} \label{cor:logp-equidistribution}
Let $p > 1 \slash 2$ and assume that $\mu$ has a finite moment of order $p$, that is, $\int_G \log^p \|g\| \, \diff \mu(g) < \infty$. Let $u$ be a $\log^p$-continuous function on $W^{1,2}$ such that $\lp \nu, u \rp = 0$, $\|u\|_{W^{1,2}} \leq 1$ and $\|u\|_{\log^p} \leq 1$. Then there exist constants $c>0$  and $0<\tau<1$ independent of $u$ such that $$\|\oP^n u\|_\infty \leq c \tau^n \quad \text{ for all } n \geq 1 .$$ 
\end{corollary}

\begin{proof}  We may work with the norm $\|u\|_\nu = |\lp \nu, u \rp| +\frac12 \|\del u\|_{L^2} + \frac12 \|\del \overline u\|_{L^2}$, which by \cite[Corollary 2.13]{DKW:PAMQ} is equivalent to $\|\cdot\|_{W^{1,2}}$. The proof given there requires  $p \geq 1$, but by Lemma \ref{lemma:g^*u-L^2} and Remark \ref{rem:1/2-moment} above, the result also holds when $p \geq 1 \slash 2$. Let $u_n: = \oP^n u$. Since $\lp \nu, u \rp = 0$ and $\nu$ is stationary we get that  $\lp \nu, u_n \rp = 0$ for every $n$. Hence $\|u_n\|_\nu = \frac12 \|\del u_n\|_{L^2} + \frac12 \|\del \overline{u_n}\|_{L^2}$. By Proposition \ref{prop:L^2_1,0-contraction}, the last quantity is bounded by $c_1 \lambda^n$ for some constants $c_1>0$ and $0< \lambda <1$. From Proposition \ref{prop:Pu-logp} we deduce that $$[u_n]_{\log^p} \lesssim 1 + M_p(\mu^{\ast n}) \lesssim n^{p+1},$$ where the last inequality follows from the sub-additivity of $\log\|g\|$ and the inequalities $(t_1 + \cdots + t_n)^p \leq n^p (\max t_i)^p \leq n^p (t_1^p + \cdots + t_n^p)$. The result then follows from Corollary \ref{cor:logp-pre-equidistribution}.
\end{proof}

\begin{remark} \label{rem:holder-equidistribution}
By noting that $\Cc^1 \subset W^{1,2} \cap \Cc^{\log^p}$, Corollary \ref{cor:logp-equidistribution} shows that Theorem \ref{thm:equi-dis} also holds under the weaker moment condition $\int_G \log^{1 \slash 2 + \varepsilon} \|g\| \, \diff \mu(g) < \infty$ for some $\varepsilon >0$.
\end{remark}

\section{The perturbations $\oP_\xi$ of the Markov operator} \label{sec:perturbed-markov-op}
 For $\xi \in \R$ consider the \textit{perturbed Markov operator}
\begin{equation} \label{eq:def-P_t}
\oP_\xi u  (x) :=   \int_G e^{i \xi \sigma_g(x)} g^*u (x) \, \diff \mu(g).
\end{equation}

Notice that $\oP_0= \oP$ is the original Markov operator (\ref{eq:makov-op}). A direct computation using the cocycle relation $\sigma(g_2g_1,x) = \sigma(g_2,g_1 \cdot x) + \sigma(g_1,x)$ gives that
\begin{align}
\oP^n_\xi u (x) &= \int_{G^n} e^{i \xi \sigma_{g_n g_{n-1} \cdots g_1} (x) } \, u (g_n g_{n-1}  \cdots g_1 \cdot x)\, \diff \mu(g_1,\ldots,g_n) \nonumber \\ &= \int_G e^{i\xi \sigma_g (x)} \, u(gx) \, \diff \mu^{\ast n} (g). \label{eq:P_t^n}
\end{align} 
In other words, $\oP^n_\xi$ corresponds to the perturbed Markov operator associated with the convolution power $\mu^{\ast n}$.

The following estimates on $\sigma_g$ will be essential in the study of the action of $\oP_\xi$ on Sobolev and $\log^p$-continuous functions.

\begin{lemma} \label{lemma:sigma-estimates}
There exists a constant $C > 0$ such that the following estimates hold for any $g \in \SL_2(\C)$:
\begin{itemize}
\item[(1)] $\|\sigma_g\|_\infty = \log \|g\|$;
\item[(2)] $\|\del \sigma_g \|_{L^2} \leq C \,  (1+\log^{1 \slash 2} \|g\|)$;

\item[(3)] The positive measure $\Psi_g := i \del \sigma_g \wedge \overline {\del \sigma_g}$ has a density $\rho_g$ with respect to the Fubini-Study form (that is, $\Psi_g = \rho_g \, \omegaFS$)   satisfying $$\int_{\P^1} \rho_g \, \log \rho_g \, \omegaFS \leq C \, (1 + \log^2 \|g\| ).$$

\item[(4)] For $u \in W^{1,2}$ we have $$\|g^* u \, \del \sigma_g \|_{L^2}  \leq C \,   (1 + \log \|g\|) \cdot \|u\|_{W^{1,2}}.$$

\item[(5)] For $p > 0$, there is a constant $C_p >0$ such that $$[\sigma_g]_{\log^p} \leq C_p \, (1+\log^{p+1} \|g\|).$$
\end{itemize}
\end{lemma}

\begin{proof}
Identity (1) follows directly from the definition of $\sigma_g$. Estimate (2) is contained in the proof of \cite[Lemma A.6]{DKW:PAMQ}. Let us prove (3). Let $g = k_g a_g \ell_g$ be a Cartan's decomposition as in \eqref{eq:cartan-decomposition}.  Using that $k_g$ and $\ell_g$ preserve the euclidean norm we obtain $\sigma_g = \ell_g^* \sigma_{a_g}$. Since $\text{SU}(2)$ is compact we can assume for our problem that $g$ is of the form $\left(\begin{smallmatrix} \lambda & 0 \\ 0 & \lambda^{-1} \end{smallmatrix}\right)$ for some $\lambda \geq 1$. Observe that it is enough to consider $\lambda$ large.

Let $z = [z:1]$ be the standard affine coordinate in $\P^1 \setminus \{\infty\}$. In this coordinate, we have $g(z) = \lambda^2 z$, so  $$\sigma_g(z) = \frac{1}{2}\log \frac{\lambda^4 |z|^2 + 1}{|z|^2 + 1} = \frac{1}{2}\log (\lambda^4 |z|^2 + 1) - \frac{1}{2}\log( |z|^2 + 1)$$ and $$\del \sigma_g = \frac12 \Big( \frac{\lambda^4 \overline z}{\lambda^4 |z|^2 + 1} -   \frac{\overline z}{ |z|^2 + 1} \Big) \diff z= \frac{(\lambda^4-1) \overline z}{2(\lambda^4|z|^2+1)(|z|^2+1)} \, \diff z.$$ Hence $$ i\del \sigma_g \wedge\overline{\partial \sigma_g} =\frac{(\lambda^4-1)^2|z|^2}{4(\lambda^4|z|^2+1)^2(|z|^2+1)^2}i \diff z\wedge \diff \overline z =\frac{(\lambda^4-1)^2|z|^2}{(\lambda^4|z|^2+1)^2}\omegaFS.$$
In other words, $$\rho_g(z) = \frac{(\lambda^4-1)^2|z|^2}{(\lambda^4|z|^2+1)^2}.$$ From estimate (2) we have that $\|\rho_g \|_{L^1} = \|\del \sigma_g\|_{L^2}^2 \lesssim 1 + \log\|g\| = 1 + \log \lambda$. The function $s \mapsto  \frac{s}{(s+1)^2}$ has $\frac{1}{4}$ as its maximal value for $s \in [0,\infty)$. This gives that $\|\rho_g \|_{\infty} \leq \frac{(\lambda^4-1)^2}{4 \lambda^4} \lesssim \frac{\lambda^4}{4}$, so $\|\log \rho_g\|_\infty \lesssim \log \lambda$. We conclude that $$\|\rho_g \log \rho_g\|_{L^1} \leq \|\rho_g \|_{L^1} \|\log \rho_g\|_\infty  \lesssim 1 + \log^2 \lambda = 1 + \log^2\|g\|,$$ giving (3).

\vskip5pt

We now prove (4). We may assume that $\|u\|_{W^{1,2}} \leq 1$. From Moser-Trudinger's estimate we have $\int_{\P^1} e^{\alpha |u|^2} \, \omegaFS \leq A$ for some constants $\alpha, A >0$ independent of $u$. Notice that  $g_* \sigma_{g} = - \sigma_{g^{-1}}$, so $g_* \Psi_g = \Psi_{g^{-1}}.$ Using Young's inequality $ab\leq e^a+b\log b-b \leq  e^a+b\log b$ for $a,b >0$ we get
\begin{align*}
\|g^*u \, \del \sigma_g \|_{L^2}^2 &= \int_{\P^1} |u|^2 \, \Psi_{g^{-1}} =  \int_{\P^1} |u|^2 \, \rho_{g^{-1}} \, \omegaFS \\ &\leq \alpha^{-1} \int_{\P^1} (e^{\alpha |u|^2} + \rho_{g^{-1}} \, \log \rho_{g^{-1}} ) \, \omegaFS  \lesssim (1+ \log^2\|g\|),
\end{align*}
where we have used estimate (3) and the fact that $\|g^{-1}\| = \|g\|$ for $g \in \SL_2(\C)$. Taking the square root, we obtain (4).

\vskip5pt

Estimate (5) is proven in \cite[Lemma 5]{cuny-dedecker-jan}. The lemma follows.
\end{proof}

\begin{proposition} \label{prop:P_t-bounded}
Assume that $\int_G \log \|g\| \, \diff \mu(g) <\infty$. Then  $\oP_\xi: W^{1,2} \to W^{1,2}$ is a bounded operator for every $\xi\in \R$.
\end{proposition}

\begin{proof}
Let $u \in W^{1,2}$ such that $\|u\|_{W^{1,2}} \leq 1$. We need to show that $\|\oP_\xi u\|_{W^{1,2}}$ is bounded independently of $u$. Recall that $\| u \|_{W^{1,2}} \lesssim \| u \|_{L^2} + \frac12 \|\del u\|_{L^2} + \frac12 \|\del \overline u\|_{L^2}$. From the triangle inequality and the assumption that $\int_G \log \|g\| \, \diff \mu(g) <\infty$ it is enough to show that  $\|e^{i \xi \sigma_g} g^* u \|_{L^2}$, $\| \del \big(e^{i\xi \sigma_g} g^* u \big) \|_{L^2}$ and $\| \del \big(\overline{e^{i\xi \sigma_g} g^* u} \big) \|_{L^2}$ are bounded by a constant times $1 + \log \|g\|$.

We have,  by Lemma \ref{lemma:g^*u-L^2}, that $$\|e^{i \xi \sigma_g} g^* u \|_{L^2} = \| g^* u \|_{L^2} \lesssim 1 + \log^{1\slash 2} \|g\| \lesssim 1 + \log  \|g\|,$$ giving the first estimate.
 Now, \begin{equation} \label{eq:del-e^itsigma_g-g^*u}
   \del \big(e^{i\xi \sigma_g} g^* u \big)  =   i \xi \, e^{i\xi \sigma_g} \del \sigma_g \, g^* u  +  e^{i\xi \sigma_g} g^* \del u.
\end{equation}
For the first term in the right hand side we have $$ \|i \xi \, e^{i\xi \sigma_g} \del \sigma_g \, g^* u\|_{L^2}  = |\xi| \cdot  \| g^* u \, \del \sigma_g \|_{L^2}  \lesssim  |\xi|  (1 + \log \|g\|),$$ where in the last step we have used Lemma \ref{lemma:sigma-estimates}-(4). For the second term, we have $$ \|e^{i\xi \sigma_g} g^* \del u\| _{L^2}  =  \| g^*\del u\|_{L^2}   = \|\del u\|_{L^2} \leq 1,$$ where we have used that $g$ acts unitarily on $L^2_{(1,0)}$. This gives that $\| \del \big(e^{i\xi \sigma_g} g^* u \big) \|_{L^2} \lesssim 1 + \log \|g\|$. The estimate on $\| \del \big(\overline{e^{i\xi \sigma_g} g^* u} \big) \|_{L^2}$ is analogous. This finishes the proof.
\end{proof}

\begin{proposition} \label{prop:P_t-regularity}
Let $p \geq 1$ and assume that $\int_G \log^p \|g\| \, \diff \mu(g) <\infty$. Then  the family $\xi \mapsto \oP_\xi$ of bounded operators on $W^{1,2}$ is continuous in $\xi$. If $p \geq 2$ the  family is $\lfloor p-1 \rfloor$-times differentiable and $\frac{\diff^k}{\diff \xi^k} \oP_\xi = \oP^{(k)}_\xi$ for $k=0,1,\ldots, \lfloor p-1 \rfloor$, where
\begin{equation} \label{eq:def-P_t^k}
\oP^{(k)}_\xi u := \int_G (i  \sigma_g)^k \, e^{i\xi \sigma_g} g^* u \, \diff \mu(g).
\end{equation}
\end{proposition}

\begin{proof}
For the first assertion, we need to prove that $\lim_{\xi \to \xi_0} \|\oP_\xi - \oP_{\xi_0}\|_{W^{1,2}}=0$ for any $\xi_0 \in \R$ which is equivalent to show that  $\lim_{\xi \to \xi_0} \|\oP_\xi u - \oP_{\xi_0}u \|_{W^{1,2}} = 0$ uniformly in $\B = \{u : \|u\|_{W^{1,2}} \leq 1\}$. By Lebesgue's dominated convergence theorem and the moment condition, the last property will follow if we show that
\begin{itemize}
\item[(i)] For fixed $g \in G$, we have $\|e^{i\xi \sigma_g} \, g^*u - e^{i\xi_0 \sigma_g} \, g^*u \|_{W^{1,2}} \to 0$ as $\xi \to \xi_0$ uniformly in $u \in \B$, and
 \item[(ii)] $\| e^{i\xi \sigma_g} \, g^*u \|_{W^{1,2}} \leq C( 1 + \log \|g\|)$ for some $C>0$ independent of $u \in \B$ and $\xi$ close to $\xi_0$.
\end{itemize}

Notice that $e^{i\xi \sigma_g} \to e^{i\xi_0 \sigma_g}$ and  $\xi e^{i\xi \sigma_g} \to \xi_0 e^{i\xi_0 \sigma_g}$ uniformly as $\xi \to \xi_0$ when $g$ is fixed.  As $$\|e^{i\xi \sigma_g} \, g^*u - e^{i\xi_0 \sigma_g} \, g^*u\|_{L^2} \leq \|e^{i\xi \sigma_g}   - e^{i\xi_0 \sigma_g}\|_\infty \|g^* u\|_{L^2} \leq 2 \|g^* u\|_{L^2}$$ and $\|g^* u\|_{L^2}$ is bounded by $1 + \log^{1 \slash 2}\|g\|$ times a constant independent of $u \in \B$ (Lemma \ref{lemma:g^*u-L^2}) we conclude that $\|e^{i\xi \sigma_g} \, g^*u - e^{i\xi_0 \sigma_g} \, g^*u\|_{L^2}$ tends to zero uniformly in $u \in \B$. Using (\ref{eq:del-e^itsigma_g-g^*u}), we get
\begin{align*}
\|\del \big (e^{i\xi \sigma_g} \, g^*u \big) - \del \big(e^{i\xi_0 \sigma_g}  \, g^*u \big)\|_{L^2} \leq \|i (\xi \, e^{i\xi \sigma_g} - \xi_0 \, e^{i\xi_0 \sigma_g}  ) \, g^* u \, \del \sigma_g\|_{L^2}  + \|(e^{i\xi \sigma_g}   - e^{i\xi_0 \sigma_g}) g^* \del u\|_{L^2} \\ \leq  \|\xi \, e^{i\xi \sigma_g} - \xi_0 \, e^{i\xi_0 \sigma_g}\|_\infty \, \|g^*u \, \del \sigma_g\|_{L^2} + \|  e^{i\xi \sigma_g} -   \, e^{i\xi_0 \sigma_g}\|_\infty \, \|g^* \del u\|_{L^2}
\end{align*}
and, arguing as before, the last two terms tend to zero as $\xi \to \xi_0$ uniformly in $u \in \B$. An analogous argument applies to $\|\del \big (\overline{e^{i\xi \sigma_g} \, g^*u} \big) - \del \big( \overline{e^{i\xi_0 \sigma_g}  \, g^*u} \big)\|_{L^2}$. This yields (i). Estimate (ii) was shown in the proof of Proposition \ref{prop:P_t-bounded}. This proves the first part of the proposition.

\vskip5pt

Assume now that $p \geq 2$. We'll use induction. Fix an integer $1 \leq k \leq \lfloor p-1 \rfloor$ and assume that $\oP_\xi$ is $k-1$ times differentiable and $\frac{\diff^{k-1}}{\diff \xi^{k-1}} \oP_\xi = \oP^{(k-1)}_\xi$. If $k=1$, this means that $\xi \mapsto \oP_\xi$ is continuous, which we have shown in the first part of the proposition.

 Let us prove that $\oP_\xi$ is $k$ times differentiable and $\frac{\diff^{k}}{\diff \xi^k} \oP_\xi = \frac{\diff}{\diff \xi} \oP^{(k-1)}_\xi = \oP^{(k)}_\xi$, or equivalently, that $ \lim_{h \to 0} h^{-1} \big( \oP_{\xi+h}^{(k-1)} - \oP_{\xi}^{(k-1)} \big) = \oP_{\xi}^{(k)}$ in $W^{1,2}$. Similarly as before,  according to Lebesgue's dominated convergence theorem, this will follow if we show that 
\begin{itemize}
\item[(iii)] For fixed $g \in G$,  we have
$$ \lim_{h \to 0} \frac{(i \sigma_g)^{k-1} \, e^{i(\xi + h) \sigma_g} - (i \sigma_g)^{k-1} \, e^{i \xi \sigma_g}}{h} g^* u =  (i \sigma_g)^k \, e^{i \xi \sigma_g} g^* u \quad \text{in } W^{1,2} \, \text{ uniformly in }\,\, u \in \B; $$
 \item[(iv)] $\| (i \sigma_g)^k e^{i\xi \sigma_g} \, g^*u \|_{W^{1,2}} \lesssim 1 + \log^{k+1} \|g\|$ for $\xi$ in a fixed compact subset of $\R$.
\end{itemize}

In order to prove (iii) we observe that when $g$ is fixed $$h^{-1}\big((i \sigma_g)^{k-1} \, e^{i(\xi+h) \sigma_g} - (i \sigma_g)^{k-1} \, e^{i \xi \sigma_g}\big)$$ converges in $\Cc^1$-norm to $(i \sigma_g)^k \, e^{i\xi \sigma_g}$ as $h \to 0$. Arguing as in the proof of (i) we get (iii).

Let us now prove (iv). We first observe that $\|(i \sigma_g)^k \, e^{i\xi \sigma_g} g^* u\|_{L^2} \leq \|\sigma_g^k \|_\infty \|g^*u\|_{L^2} \lesssim 1 + \log^{k+1\slash 2}\|g\| \lesssim 1 + \log^{k+1}\|g\|$ by Lemmas \ref{lemma:g^*u-L^2} and \ref{lemma:sigma-estimates}-(1). In order to estimate the derivative, we first compute
\begin{align*}
\del \big( \sigma_g^k \, e^{i\xi \sigma_g} g^* u \big) &= k \, \sigma_g^{k-1} e^{i\xi \sigma_g} \, g^*u \, \del \sigma_g + i\xi \sigma_g^{k} e^{i\xi \sigma_g} g^*u \,\del \sigma_g + \sigma_g^k  e^{i\xi \sigma_g} \, g^* \del u \\  & =: \phi_g^{(1)} +  \phi_g^{(2)} +  \phi_g^{(3)}.
\end{align*}

We have that $\| \phi_g^{(1)} \|_{L^2} \leq k \, \|\sigma_g^{k-1}\|_\infty \, \|g^*u \, \del \sigma_g\|_{L^2} \lesssim 1 + \log^{k} \|g\|$ by Lemma \ref{lemma:sigma-estimates}. Analogously, we have $\| \phi_g^{(2)} \|_{L^2} \leq |\xi| \cdot  \|\sigma_g^{k}\|_\infty \, \|g^*u \, \del \sigma_g\|_{L^2} \lesssim 1 + \log^{k+1} \|g\|$ and also  $\| \phi_g^{(3)} \|_{L^2} \leq \|\sigma_g^{k}\|_\infty \, \|g^*\del u \|_{L^2}\lesssim 1 + \log^{k} \|g\|$. We conclude that $\big\|\del \big( (i\sigma_g)^k \, e^{i \xi \sigma_g} g^* u \big)\big\|_{L^2}$ $\lesssim 1 + \log^{k+1}\|g\|$, which, together with an analogous computation for $\del \big( \overline{\sigma_g^k \, e^{i\xi \sigma_g} g^* u} \big) $ and the previous $L^2$ estimate, gives (iv). This finishes the proof of the proposition.
\end{proof}

\begin{remark}
Observe that the proof of the above proposition also shows that for every $\xi \in \R$ and $0\leq k \leq p-1$, $\oP_\xi^{(k)}$ defines a bounded operator on $W^{1,2}$.
\end{remark}

The regularity of the family $\xi \mapsto \oP_\xi$ provided by the above proposition allows us to apply the theory of perturbations of linear operators, see \cite{kato:book} and also \cite{gouezel:spectral-methods}. The general theory implies that the spectral properties of $\oP_0 =\oP$ (namely, the existence of a spectral gap, cf. Theorem \ref{thm:DKW-spectral-gap}) persist for small values of $\xi$. Moreover,  the spectral decomposition inherits the same regularity properties. The spectral properties of $\oP_\xi$ for large values of $\xi$ will be studied in Section \ref{sec:spec-Pt}. 

\begin{corollary} \label{cor:P_t-decomp}
Let $\mu$ be a non-elementary probability measure on $G=\SL_2(\C)$. Assume that $\int_G \log^p \|g\| \, \diff \mu(g) <\infty$ for some $p \geq 2$. Then, there exists an $\epsilon_0 > 0$ such that, for $\xi \in [-\epsilon_0,\epsilon_0]$, one has a decomposition
\begin{equation} \label{eq:P_t-decomp}
\oP_\xi = \lambda_\xi \oN_\xi + \oQ_\xi,
\end{equation}
 where $\lambda_\xi \in \C$, $\oN_\xi$ and $\oQ_\xi$ are bounded operators on $W^{1,2}$ and 
\begin{enumerate}
\item $\lambda_0 = 1$ and $\oN_0 u = \int u \, \diff \nu$, where $\nu$ is the unique $\mu$-stationary measure;
\item $\rho:= \displaystyle \lim_{n \to \infty } \|\oP_0^n - \oN_0\|_{W^{1,2}}^{1 \slash n} < 1$;

\item $\lambda_\xi$ is the unique eigenvalue of maximum modulus of $\oP_\xi$, $\oN_\xi$ is a rank-one projection and $\oN_\xi \oQ_\xi = \oQ_\xi \oN_\xi = 0$;

\item the maps $\xi \mapsto \lambda_\xi$,  $\xi \mapsto \oN_\xi$ and $\xi \mapsto \oQ_\xi$ are $\lfloor  p-1 \rfloor $-times differentiable;

\item  $|\lambda_\xi| > \frac{2 + \rho}{3}$ and for every $k=0,\ldots,\lfloor  p-1 \rfloor$, there exists a constant $c > 0$ such that $$\bigg \| \frac{\diff^k \oQ_\xi^n}{\diff \xi^k} \bigg \|_{W^{1,2}} \leq c \Big( \frac{1 + 2 \rho}{3} \Big)^n \quad \text{ for every } n \geq 0.$$
\end{enumerate}
\end{corollary}

Notice that in Theorem \ref{thm:LLT} we require a second moment condition. Proposition \ref{prop:P_t-regularity} shows in this case that $\xi \mapsto \oP_\xi$ is one time differentiable and, by the above corollary, the same regularity holds for the leading eigenvalue $\lambda_\xi$ near zero.  However, the next lemma will show that $\lambda_\xi$ is slightly more regular, admitting an order two expansion. This will be crucial in the proof of Theorem \ref{thm:LLT}. Recall that $\gamma >0$ is the Lyapunov exponent (cf. Section \ref{sec:preliminary}).

\begin{lemma}  \label{lemma:lambda_t-expansion}
Assume that  $\int_G \log^2 \|g\| \, \diff \mu(g) <\infty$. Then, the leading eigenvalue $\lambda_\xi$ of $\oP_\xi$ satisfies
\begin{equation} \label{eq:lambda_t}
\lambda_\xi = \int_{\P^1} \int_G e^{i\xi \sigma_g(x)} \diff \mu(g) \diff \nu(x) + \phi(\xi),
\end{equation}
where $\phi$ is a continuous function such that $\xi^{-2}\phi(\xi)$ has a finite limit at $\xi=0$.

Moreover, there is a constant  $A \geq \gamma^2$ and a continuous function $\psi$ such that
\begin{equation} \label{eq:lambda-expansion}
\lambda_\xi = 1 + i  \gamma \xi - A \frac{\xi^2}{2}+ \psi(\xi),
\end{equation}
and $
\lim_{\xi\to 0 }\xi^{-2}\psi(\xi) =0$.
\end{lemma}

\begin{proof}
We keep the notations of Corollary \ref{cor:P_t-decomp}. Let $\psi_\xi := \frac{\oN_\xi \mathbf 1}{ \lp \nu, \oN_\xi \mathbf 1 \rp}$. Then $\psi_\xi$ is an eigenvector of $\oP_\xi$ associated with $\lambda_\xi$ and $\lp \nu, \psi_\xi \rp = 1$. Here $ \lp \nu, \psi_\xi \rp$ stands for the value at $\psi_\xi$ of the linear functional defined by $\nu$ (as in Theorem \ref{thm:DKW-spectral-gap}). By Lemma \ref{lemma:int-good-rep} in the Appendix, this value coincides with $\int_{\P^1} \psi_\xi\, \diff \nu$ for a good representative of $\psi_\xi$. Using that $\lp \nu , \oP_0 \psi_\xi \rp = \lp \nu , \psi_\xi \rp = 1$ we can write 
\begin{equation*}
\lambda_\xi = \lp \nu, \lambda_\xi \psi_\xi \rp = \lp \nu,  \oP_\xi \psi_\xi \rp =  \lp \nu, \oP_\xi \mathbf 1 \rp + \lp \nu, (\oP_\xi -\oP_0)(\psi_\xi - \mathbf 1)\rp.
\end{equation*}
The first term equals $\int_{\P^1} \int_G e^{i\xi \sigma_g(x)} \diff \mu(g) \diff \nu(x)$. Denote by $\phi(\xi)$ the second term. From Proposition \ref{prop:P_t-regularity} and Corollary \ref{cor:P_t-decomp} it follows that $\phi$ is continuous. We need to show that $\xi^{-2}\phi(\xi)$ has a limit at $\xi=0$. We have $$\frac{\phi(\xi)}{\xi^2} =  \frac{1}{ \lp \nu, \oN_\xi \mathbf 1 \rp} \Big \lp \nu, \Big( \frac{\oP_\xi -\oP_0}{\xi} \Big) \Big( \frac{\oN_\xi \mathbf 1 - \lp \nu, \oN_\xi \mathbf 1   \rp}{\xi} \Big) \Big \rp $$
and $\oN_\xi = \oN_0 + \xi \oN'_0 + o(|\xi|)$ as operators on $W^{1,2}$ by Corollary \ref{cor:P_t-decomp}-(4). It follows that $$\lim_{\xi \to 0} \frac{\oN_\xi \mathbf 1 - \lp \nu, \oN_\xi \mathbf 1 \rp}{\xi} = \oN'_0 \mathbf 1 - \lp \nu, \oN'_0 \mathbf 1\rp  \quad \text{in } W^{1,2}.$$
Therefore, $ \big( \frac{\oP_\xi -\oP_0}{\xi} \big) \big( \frac{\oN_\xi \mathbf 1 - \lp \nu, \oN_\xi \mathbf 1   \rp}{\xi} \big) $ tends to $\oP'_0 \big(\oN'_0 \mathbf 1 - \lp \nu,  \oN'_0 \mathbf 1\rp \big)$ in $W^{1,2}$. This shows that $\xi^{-2}\phi(\xi)$ has a finite limit at $\xi=0$, giving (\ref{eq:lambda_t}).

Set $$\beta:= 2 \lim_{\xi \to 0} \xi^{-2} \phi(\xi)\quad\text{and}\quad A :=  \int_{\P^1} \int_G \sigma^2_g(x) \diff \mu(g) \diff \nu(x) - \beta.$$
 Recall Furstenberg's formula $\gamma = \int_{\P^1} \int_G \sigma_g(x) \diff \mu(g) \diff \nu(x)$. From (\ref{eq:lambda_t}) we have
\begin{align*}
\lambda_\xi - 1 - i\gamma \xi + A \frac{\xi^2}{2} =  \int_{\P^1} \int_G \Big[ e^{i\xi \sigma_g(x)} - 1 -  i\xi \sigma_g(x) + \frac{\xi^2}{2} \sigma_g^2(x)  \Big] \, \diff \mu(g) \diff \nu(x) - \beta \frac{\xi^2}{2} + \phi(\xi).
\end{align*}

By the the definition of $\beta$, one has $\lim_{\xi \to 0}\frac{1}{\xi^2} \big(\beta \frac{\xi^2}{2} + \phi(\xi)\big) = 0$.
Observe that, for fixed $g$, we have 
$$\lim_{\xi\to 0} \frac{1}{\xi^2} \Big(e^{i \xi \sigma_g} -1 - i\xi \sigma_g + \frac{\xi^2}{2} \sigma^2_g\Big) = 0$$ uniformly on $\P^1$. Also 
$$\frac{1}{\xi^2}\Big| e^{i \xi \sigma_g} -1 - i\xi \sigma_g + \frac{\xi^2}{2} \sigma^2_g \Big| \leq   \frac{1}{\xi^2}\big| e^{i \xi \sigma_g} -1 - i\xi \sigma_g \big| + \frac12 \sigma^ 2_g \leq c \sigma_g^2$$ for some constant $c > 0$ independent of $g$, where the second inequality follows by looking at the error term on  Taylor's expansion of order one. By Lebesgue's dominated convergence theorem we conclude that $\xi^{-2} \phi_2(\xi)$ tends to zero as $\xi \to 0$, where $$\phi_2(\xi):=  \int_{\P^1} \int_G \Big[ e^{i\xi \sigma_g(x)} - 1 -  i\xi\sigma_g(x) + \frac{\xi^2}{2} \sigma_g^2(x)  \Big] \, \diff \mu(g) \diff \nu(x).$$ Putting $\psi(\xi):= \phi_2(\xi)   -  \beta \frac{\xi^2}{2} + \phi(\xi)$  gives  (\ref{eq:lambda-expansion}).

 It remains to show that $A$ is real  and $A \geq \gamma^2$.   We begin by noticing that,  for $\xi$ small,  if  $\psi_\xi$ is as above,  then  $\overline{\lambda_\xi \psi_\xi} = \overline{\oP_\xi \psi_\xi} = \oP_{-\xi} \overline{\psi_\xi}$.  Thus, the leading eigenvalue of $\oP_{-\xi}$ is $\lambda_{-\xi} = \overline{\lambda_{\xi}}$.   Observe now that $|\lambda_\xi| \leq 1$,  because $|\lambda_\xi| \lp \nu,  |\psi_\xi| \rp = \lp \nu,  |P_\xi \psi_\xi| \rp \leq   \lp \nu,  P_0 | \psi_\xi| \rp = \lp \nu, |\psi_\xi| \rp$.  Then,  using  (\ref{eq:lambda-expansion}) we get $0 \leq 1 - |\lambda_\xi|^2  = 1 - \lambda_\xi \lambda_{-\xi}  = (A-\gamma^2) \xi^2 + o(|\xi^2|)$, so $A-\gamma^2 = \lim_{\xi \to 0} \xi^{-2} (1 - |\lambda_\xi|^2) \geq 0$.  Hence $A$ is real and $A \geq \gamma^2$.  The proof is finished.
\end{proof}

\begin{lemma} \label{lemma:lambda-expansion-2}
Consider the sequence of random variables $Z_n:= \frac{1}{\sqrt n} (\sigma(g_n \cdots g_1,x) - n\gamma )$,  where the $g_j$ are i.i.d. with law $\mu$ and $x$ has law $\nu$.  Then $Z_n$ converges in law to a centered normal distribution  $\cali N(0;a^2)$ with variance  $a^2  := A - \gamma^2$. 
\end{lemma}

\begin{proof}
We note that the first part follows from the CLT for the norm cocycle (Theorem \ref{thm:CLT}).  We provide here another proof for the CLT for $Z_n$,  which will give the desired formula for the variance.

For a continuous function $\varphi$ on $G^{\N^*} \times \P^1$, denote $\mathbf E_\nu[\varphi] := \int \varphi(\omega,x) \, \diff \mu^{\N^*}(\omega) \, \diff \nu(x)$,  where  $\mu^{\N^*}$ is the product measure on $G^{\N^*} $ obtained from $\mu$.  For small $\xi$ we have $$ \mathbf E (e^{i \xi Z_n}) =  \mathbf E_\nu \big[ e^{i\xi (\sigma(g_n \cdots g_1,x) - n\gamma )} \big] = e^{-in\xi \gamma} \lp \nu, \oP^n_\xi \mathbf 1 \rp =  e^{-in\xi \gamma} \lambda_\xi^n \lp \nu, \oN_\xi \mathbf 1 \rp +  e^{-in\xi \gamma} \lp \nu,\oQ^n_\xi \mathbf 1 \rp.$$

Fix $\tau \in \R$ and let $\xi_n = \frac{\tau}{\sqrt n}$ for $n$ large enough so that the above identity  holds for $\xi=\xi_n$. From the expansions $$\lambda_{\tau \slash \sqrt n}= 1 + i \gamma \frac{\tau}{\sqrt n} - \frac{A}{2}\cdot \frac{\tau^2}{n} + o \big(\frac{1}{n}\big)\quad\text{and} \quad\log (1 + s) = s - \frac{s^2}{2} + o(s^2)$$ we obtain $$\log \lambda_{\tau \slash \sqrt n}^n = n \log \lambda_{\tau  \slash \sqrt n} = n \Big( i \gamma \frac{\tau}{\sqrt n}+ (-A  + \gamma^2) \frac{\tau^2}{2n} + o \big(\frac{1}{n}\big)  \Big),$$ 
so $\lambda_{\tau \slash \sqrt n}^n = e^{i \sqrt n  \gamma \tau} e^{- \frac{1}{2} (A - \gamma^2) \tau^2 + o(1) }$.
This gives
$$\mathbf E_\nu \big[ e^{i \frac{\tau}{\sqrt n} (\sigma(g_n \cdots g_1,x) - n\gamma )} \big] =    e^{- \frac{1}{2} (A - \gamma^2) \tau^2 + o(1) }  \lp \nu, \oN_{\tau \slash \sqrt n} \mathbf 1 \rp +  e^{-i \tau \sqrt n \gamma} \lp \nu, \oQ^n_{\tau \slash \sqrt n} \mathbf 1 \rp.$$

As $\oN_{\tau \slash \sqrt n} \mathbf 1 \to \oN_0 \mathbf 1 = \mathbf 1$ and $\oQ^n_{\tau \slash \sqrt n} \mathbf 1  \to 0$ in $W^{1,2}$ as $n \to \infty$  and $\nu$ acts continuously on $W^{1,2}$ (Proposition \ref{p:WPC}),  we conclude that $  \mathbf E( e^{i \tau Z_n}) =   \mathbf E_\nu \big[ e^{i \frac{\tau}{\sqrt n} (\sigma(g_n \cdots g_1,x) - n\gamma )} \big]$  converges to $ e^{- \frac{1}{2} (A - \gamma^2)\tau^2} = e^{- \frac{1}{2} a^2 \tau^2} $ as $n \to \infty$.  Since $e^{- \frac{1}{2} a^2 \tau^2}$ is the characteristic function of the centered normal distribution with variance $a^2$,  we conclude that  $Z_n$ converges in law to $\cali N(0;a^2)$.  This concludes the proof. 
\end{proof}

\begin{remark}
The number  $a^2$ above coincides with the one in Theorem \ref{thm:CLT}.  In particular it is strictly positive.
\end{remark}

The following version of \eqref{eq:lambda-expansion} will be used in the proof of Theorem \ref{thm:LLT-coeff}.

\begin{lemma}  \label{lemma:lambda_t-expansion-3}
	Assume that  $\int_G \log^3 \|g\| \, \diff \mu(g) <\infty$. Then
	\begin{equation*} 
	\lambda_\xi = 1 + i  \gamma \xi - A \frac{\xi^2}{2}+O(|\xi|^3) \quad\text{as}\quad \xi\to 0,
	\end{equation*}
	where $A = a^2 + \gamma^2$.
\end{lemma}

\begin{proof}
	We keep the notations of the proofs of Lemmas \ref{lemma:lambda_t-expansion} and \ref{lemma:lambda-expansion-2}. We have
	$$   \lambda_\xi - 1 - i\gamma \xi + A \frac{\xi^2}{2} =  \phi_2(\xi)   - \beta \frac{\xi^2}{2} + \phi(\xi). $$
	Using the Taylor's expansion and Lebesgue's dominated convergence theorem, we see that 
	$$\lim_{\xi\to 0 } {\phi_2(\xi)\over\xi^3} =\int_{\P^1}\int_G  -{i\over 6}\sigma_g^3(x)   \, \diff \mu(g) \diff \nu(x),$$
	which is finite due to our moment assumption. Therefore, it is enough to show that $\xi^{-3}\big( \phi(\xi)-\beta {\xi^2\over 2}\big)$ is bounded when $\xi \to 0$.
	
	From  Corollary \ref{cor:P_t-decomp}, we have
	$$\oP_\xi =\oP_0+\xi\oP_0' + \frac{\xi^2}{2}\oP_0''+o(\xi^2)\quad \text{and}\quad \oN_\xi =\oN_0+\xi\oN_0' + \frac{\xi^2}{2} \oN_0''+o(\xi^2)$$
	as operators on $W^{1,2}$. It follows that 
  \begin{align*}
	\phi(\xi)-\beta {\xi^2\over 2}&=   \frac{1}{ \lp \nu, \oN_\xi \mathbf 1 \rp} \big\lp \nu, ( \oP_\xi -\oP_0 ) \big( \oN_\xi \mathbf 1 - \lp \nu, \oN_\xi \mathbf 1   \rp \big)\big \rp   -\beta {\xi^2\over 2}\\
	&= \frac{\xi^2}{ \lp \nu, \oN_\xi \mathbf 1 \rp}\oP'_0 \big(\oN'_0 \mathbf 1 - \lp \nu,  \oN'_0 \mathbf 1\rp \big) +O(|\xi|^3) -\beta{\xi^2\over 2} =O(|\xi|^3),
	\end{align*}
	where in the last step we have use that $\beta= 2 \lim_{\xi \to 0} \xi^{-2} \phi(\xi)$ and that $\lp \nu, \oN_\xi \mathbf 1 \rp = 1 + O(|\xi|)$.  We deduce that $\xi^{-3}\big( \phi(\xi)-\beta {\xi^2\over 2}\big)$ is bounded when $\xi \to 0$. The proof of the lemma is finished.
\end{proof}

As a consequence, we get the  following useful estimates.
\begin{lemma} \label{lemma:lambda-estimates}
	Assume that  $\int_G \log^3 \|g\| \, \diff \mu(g) <\infty$. Then, there exists a constant $\xi_0>0$ such that, for all $n \in \N$ large enough, one has $$\big|\lambda_{{\xi\over \sqrt n}}^n\big|\leq e^{-{a^2\xi^2\over 3}} \quad\text{for}\quad  |\xi|\leq \xi_0\sqrt n,$$
	$$\Big|  e^{-i\xi\sqrt n \gamma}\lambda_{{\xi\over \sqrt n}}^n-e ^{-{a^2\xi^2\over 2}}  \Big|\leq {c\over \sqrt n}|\xi|^3e^{-{a^2\xi^2\over 2}} \quad\text{for}\quad |\xi|\leq \sqrt[6] n,$$  
	$$\Big| e^{-i\xi\sqrt n \gamma} \lambda_{{\xi\over \sqrt n}}^n-e ^{-{a^2\xi^2\over 2}}  \Big|\leq {c\over \sqrt n}e^{-{a^2\xi^2\over 4}} \quad\text{for}\quad \sqrt[6] n<|\xi|\leq \xi_0\sqrt n,$$  
	where  $c>0$ is a constant independent of $n$.
\end{lemma}

\begin{proof}
	By Lemma \ref{lemma:lambda_t-expansion-3}, we have the following expansion for $|\xi|$ small
	$$\lambda _\xi=1+ i\gamma \xi -(a^2+\gamma^2)\xi^2/2+O(|\xi|^3).$$  
	It follows that,
	$$\log |\lambda_\xi|=\log \sqrt{1-a^2\xi^2+O(|\xi|^3)}=-{a^2\xi^2\over 2} + O(|\xi|^3) .$$
	In particular, by replacing $\xi$ by $\xi/\sqrt n$ and taking $\xi_0 > 0$ small enough, we deduce that
	$$
	\big|\lambda_{{\xi\over \sqrt n}}\big|^n=e^{n\log |\lambda_{\xi/ \sqrt n}|}\leq e^{-{a^2\xi^2\over 3}}   \quad\text{for}\quad |\xi|\leq \xi_0\sqrt n.
	$$
	This proves the first inequality.
	
	When $|\xi|\leq \sqrt[6]{n}$, the sequence  $n|\xi/\sqrt n|^3$ is bounded. Using the above estimate on $\log |\lambda_\xi|$ and that $e^{O(|t|)} = 1 + O(|t|)$ for $|t|$ bounded, we have 
	$$|\lambda_{\xi\over \sqrt n}|^n=e^{n\log |\lambda_{\xi/\sqrt n}|}= e^{-{a^2\xi^2\over 2} +O(n|\xi/\sqrt n|^3)}=e^{-{a^2\xi^2\over 2}}+O\big(n|\xi/\sqrt n|^3e^{-{a^2\xi^2\over 2}}\big).$$
	The above expansion of $\lambda_\xi$ implies that 
		$$\Arg \lambda_\xi -\gamma\xi =\arctan {\gamma\xi +O(|\xi|^3)\over 1+O(\xi^2)} -\gamma\xi =\arctan\big(\gamma\xi+O(|\xi|^3)\big)-\gamma\xi=O(|\xi|^3),$$ which gives
	\begin{align*}
	&e^{-i \xi \sqrt n\gamma}\lambda_{\xi\over \sqrt n}^n=\big|\lambda_{\xi\over \sqrt n}\big|^n e^{in(\Arg \lambda_{\xi/\sqrt n} \, -\gamma\xi/\sqrt n)}=\big|\lambda_{\xi\over \sqrt n}\big|^n e^{in\,O(|\xi/\sqrt n|^3)}\\
	&= e^{-{a^2\xi^2\over 2}}+O\big(n|\xi/\sqrt n|^3e^{-{a^2\xi^2\over 2}}\big)
	=e^{-{a^2\xi^2\over 2}}+{1\over\sqrt{n}} O\big(|\xi|^3e^{-{a^2\xi^2\over 2}}\big).
	\end{align*}
	The second inequality follows.

	When $|\xi| > \sqrt[6]{n}$. The above bound on $\lambda_{{\xi\over \sqrt n}}$ yields 
	$$\big| e^{-i \xi\sqrt n \gamma} \lambda_{{\xi\over \sqrt n}}^n \big| \leq e^{-{a^2\xi^2\over 3}}= e^{-{a^2\xi^2\over 12}}e^{-{a^2\xi^2\over 4}}  \leq e^{-{a^2\sqrt[3]{n}\over 12}} e^{-{a^2\xi^2\over 4}} \leq  {c\over \sqrt{n}} e^{-{a^2\xi^2\over 4}}$$ 
	and since, in this case,
	$$ e^{-{a^2\xi^2\over 2}}  = e^{-{a^2\xi^2\over 4}} e^{-{a^2\xi^2\over 4}}  \leq e^{-{a^2\sqrt[3]{n}\over 4}}e^{-{a^2\xi^2\over 4}} \leq {c\over \sqrt n}  e^{-{a^2\xi^2\over 4}}$$ 
	for some constant $c>0$, the third inequality follows. This finishes the proof of the lemma.
\end{proof}

\vskip5pt

We now study the action of the operators $\oP_\xi$ on the space of $\log^p$-continuous functions.

\begin{proposition} \label{prop:P_t-logp}
Let $p > 1$ and let $u$ be a $\log^{p-1}$-continuous function on $\P^1$. Assume that $M_{p}(\mu):= \int_G \log^{p} \|g\| \, \diff \mu(g) < \infty $. Then,  $$\|\oP_\xi u\|_{\log^{p-1}} \leq c \, (1 + |\xi|) M_{p}(\mu) \|u\|_{\log^{p-1}}$$ for some some constant $c>0$ independent of $u,\xi$ and $\mu$.  In particular, for every $\xi \in \R$, the operator $\oP_\xi$ acts continuously on the space of $\log^{p-1}$-continuous functions.  Moreover, for every $\xi \in \R$ and every $n \geq 1$ we have $$\|\oP^n_\xi\|_{\log^{p-1}} \leq c \, ( 1 + |\xi|) \, n^{p} \, M_{p}(\mu).$$
\end{proposition}

\begin{proof}
We may assume that $\|u\|_{\log^{p-1}} \leq 1$. By Lemma \ref{lemma:log^p-product} we have that $$[e^{i \xi \sigma_g}g^*u]_{\log^{p-1}} \leq [g^*u]_{\log^{p-1}} + [e^{i \xi \sigma_g}]_{\log^{p-1}} \|g^* u\|_\infty$$ and $[g^*u]_{\log^{p-1}} \leq C (1+ \log^{p-1}\|g\|)$ by Lemma \ref{lemma:norm-g^*u}. Now $[\sigma_g]_{\log^{p-1}} \leq C_{p-1} (1 + \log^{p} \|g\|)$ by Lemma \ref{lemma:sigma-estimates}-(5). Since $t \mapsto e^{i\xi t}$ is $|\xi|$-Lipschitz, it follows that $$[e^{i \xi \sigma_g}]_{\log^{p-1}}  \leq C_{p-1} |\xi| (1+ \log^{p} \|g\|).$$ We conclude that $[e^{i \xi \sigma_g}g^*u]_{\log^{p-1}} \lesssim (1+ |\xi|)(1 +  \log^{p} \|g\|)$. As $\|e^{i \xi \sigma_g}g^* u\|_{\infty} = \|u\|_{\infty} \leq 1$ one gets  that $$\|e^{i \xi \sigma_g}g^*u\|_{\log^{p-1}}   \leq c(1 + |\xi|)(1+ \log^{p} \|g\|)$$ for some constant $c > 0$ independent of $u,\xi$ and $g$. The first assertion follows by applying the triangle inequality to the integral (\ref{eq:def-P_t}) defining $\oP_\xi u$.

The last assertion follows from the fact that $\oP_\xi^n$ is the perturbed Markov operator associated with $\mu^{\ast n}$ (cf. (\ref{eq:P_t^n})) and that $M_p(\mu^{\ast n}) \leq n^p M_p(\mu)$ due to the sub-additivity of $\log \|g\|$ and the convexity of $t \mapsto t^p$ for $p >1$.
\end{proof}

\begin{remark}
The main feature of the last part of the above proposition is that the $\log^{p-1}$-norm of $\oP^n_\xi u$ grows polynomially in $n$, instead of the exponential growth obtained by iterating the estimate obtained in the first part.
\end{remark}

\begin{proposition} \label{prop:P_t-logp-regularity}
Let $p > 1$ and assume that $\int_G \log^{p} \|g\| \, \diff \mu(g) < \infty $ (resp. $\int_G \log^{p+1} \|g\| \, \diff \mu(g) < \infty $). Then, the family $\xi \mapsto \oP_\xi$ acting on $\Cc^{\log^{p-1}}$ is continuous (resp. locally Lipschitz continuous) in $\xi$.
\end{proposition}

\begin{proof}

For the continuity, we need to prove that  for any $\xi_0 \in \R$ we have $\lim_{\xi \to \xi_0} \|\oP_\xi u - \oP_{\xi_0} u \|_{\log^{p-1}} = 0$, uniformly in $u$ with $\|u\|_{\log^{p-1}} \leq 1$. Assume, from now until the end of the proof, that $\|u\|_{\log^{p-1}} \leq 1$. By Lebesgue's dominated convergence theorem and the moment condition, the above convergence will follow if we show that
\begin{itemize}
\item[(i)] For fixed $g \in G$, we have $\lim_{\xi \to \xi_0} e^{i\xi \sigma_g} \, g^*u = e^{i\xi_0 \sigma_g} \, g^*u$ in $\Cc^{\log^{p-1}}$ and
 \item[(ii)] $\| e^{i\xi \sigma_g} \, g^*u \|_{\log^{p-1}} \lesssim 1 + \log^{p} \|g\|$.
\end{itemize}

The first assertion is a consequence of the fact that for fixed $g \in G$ we have   $e^{i \xi \sigma_g} \to e^{i \xi_0 \sigma_g}$  in $\Cc^1$-norm as $\xi \to \xi_0$ and that $\|\cdot\|_{\log^ p} \lesssim \|\cdot\|_{\Cc^ 1}$, see also Lemma \ref{lemma:log^p-product}. Estimate (ii) is given in the proof of Proposition \ref{prop:P_t-logp}. This gives the first part of the proposition.

\medskip

We now show the Lipschitz continuity. It is enough to work with $\xi$ in a fixed bounded interval $J$ of $\R$. We need to show that $\|(\oP_\xi - \oP_\eta) u\|_{\log^{p-1}} \lesssim |\xi- \eta|$ for $\xi, \eta \in J$  uniformly in $u$. We have $$(\oP_\xi - \oP_\eta) u  = \int_G \big(e^{i \xi \sigma_g } - e^{i \eta \sigma_g} \big)  g^*u  \,  \diff \mu(g).$$
	
	Using Lemma \ref{lemma:sigma-estimates}-(1), we have $$\big\|\big(e^{i\xi \sigma_g} - e^{i\eta \sigma_g}\big) g^*u \big\|_\infty \leq \|e^{i\xi \sigma_g} - e^{i\eta \sigma_g}\|_\infty \leq |\xi- \eta|\cdot \log \|g\|,$$ so $\|(\oP_\xi - \oP_\eta) u \|_\infty \leq M_1(\mu) |\xi- \eta|$.
	
	 It remains to bound $[(\oP_\xi - \oP_\eta) u ]_{\log^{p-1}}$. From Lemma \ref{lemma:log^p-product},
	we have  $$[(e^{i\xi \sigma_g} - e^{i\eta \sigma_g}) g^*u]_{\log^{p-1}} \leq [(e^{i\xi \sigma_g} - e^{i\eta \sigma_g})]_{\log^{p-1}} + \|e^{i\xi \sigma_g} - e^{i\eta \sigma_g}\|_\infty [g^*u]_{\log^{p-1}}.$$
	
	As before, $\|e^{i\xi \sigma_g} - e^{i\eta \sigma_g}\|_\infty \leq |\xi- \eta|\cdot \log \|g\|$ and $[g^*u]_{\log^{p-1}} \lesssim 1+\log^{p-1}\|g\|$  by Lemma \ref{lemma:norm-g^*u}, so the last term above is bounded by a constant times $|\xi- \eta| \cdot (1+\log^p \|g\|)$.
	
	We now estimate the quantity $[(e^{i\xi \sigma_g} - e^{i\eta \sigma_g})]_{\log^{p-1}}$.  Observe that 
$$ e^{i\xi \sigma_g}-e^{i\eta \sigma_g}=(\xi-\eta)\int_0^1  i\sigma_g e^{i(\eta+t(\xi-\eta))\sigma_g}dt . $$ 

 Arguing as in the proof of Proposition \ref{prop:P_t-logp} and recalling that $\eta,\xi$ belong to a bounded  interval we get that $[i\sigma_g e^{i(\eta+t(\xi-\eta))\sigma_g}  ]_{\log^{p-1}}$ is bounded by 
 $$ [\sigma_g ]_{\log^{p-1}}+ \|\sigma_g \|_\infty  [e^{i(\eta+t(\xi-\eta))\sigma_g}  ]_{\log^{p-1}} \lesssim 1+  \log^p \norm{g} + \log\norm{g}  (1+\log^p\norm{g})\lesssim 1+\log^{p+1}\norm{g}.$$ 
  By the triangle inequality applied to the above integral, we obtain 
$$[(e^{i\xi \sigma_g} - e^{i\eta \sigma_g})]_{\log^{p-1}} \lesssim   |\xi-\eta|\cdot(  1+\log^{p+1}\norm{g}).$$ Integrating with respect to $g$, we conclude that $\|(\oP_\xi - \oP_\eta) u\|_{\log^{p-1}} \leq c |\xi- \eta| \cdot M_{p+1}(\mu)$ for $\xi, \eta \in J$, where $c>0$ is independent of  $\xi, \eta$ and $u$ in the unit ball of $\Cc^ {\log^{p-1}}$. This concludes the proof.	
\end{proof}

\section{Spectrum of $\oP_\xi$ on the Sobolev space $W^{1,2}$} \label{sec:spec-Pt}

In this section, we study the spectrum of $\oP_\xi$ acting on $W^{1,2}$. The main result is that, when $\xi \neq 0$, the spectral radius of $\oP_\xi$ is strictly less than one. This is a crucial property in the proofs of all of our main theorems.
 
 \begin{theorem} \label{thm:P_t-contracting}
Let $\mu$ be a probability measure on  $G=\SL_2(\C)$.   Assume that $\mu$ is non-elementary and that $\int_G \log^2 \|g\| \, \diff \mu(g) <\infty$.  Let $\xi$ be a non-zero real number. Then the spectral radius of $\oP_\xi$ acting on $W^{1,2}$ is less than one. Furthermore, for every compact subset $K \subset \R \setminus \{0\}$, there exist $0 < \delta < 1$ and $C>0$  such that $\norm{\oP_\xi^n}_{W^{1,2}} <C( 1 - \delta)^n$ for every $\xi \in K$ and $n\in\N$.
 \end{theorem}

When $\mu$ has a finite exponential moment,  an analogous result holds for $\oP_\xi$ acting on a space of H\"older continuous functions, see  \cite{lepage:theoremes-limites}, \cite[Chapter 15]{benoist-quint:book} and also \cite{li:fourier-2}.

The proof of Theorem \ref{thm:P_t-contracting} is divided into several steps.  The idea is to first derive basic estimates to show that the sequence $\|\oP_\xi^n\|_{W^{1,2}}$ is bounded for $n\geq 0$,  then prove weak convergence to zero and,  in the end,  combine this convergence with the first basic estimates to improve it to strong convergence,  which will ultimately give the desired result.  For this reason,  similar reasoning and related estimates are used repeatedly. 

We now begin the proof of Theorem \ref{thm:P_t-contracting}. The next lemma is the analogue of Proposition \ref{prop:L^2_1,0-contraction} for $\oP_\xi$.

\begin{lemma} \label{lemma:Pt-L^2_1,0-contraction}
Let $\mu$ be a non-elementary probability measure on $G =\SL_2(\C)$. Then the operator $\oP_\xi:L^2_{(1,0)} \to L^2_{(1,0)}$ defined by $\oP_\xi \phi := \int_G e^{i \xi \sigma_g} g^* \phi \, \diff \mu(g)$ is bounded. Furthermore, there is an $N \geq 1$ independent of $\xi$ such that, after replacing $\mu$ by the convolution power $\mu^{\ast N}$, one has $\|\oP_\xi\|_{L^2_{(1,0)}} < 1 \slash 2$.
\end{lemma}

\begin{proof}
The case $\xi = 0$ follows from Proposition \ref{prop:L^2_1,0-contraction}.  The case of general $\xi \in \R$ can be proven as in \cite[Proposition 2.9]{DKW:PAMQ}. The extra factor $e^{i \xi \sigma_g}$ doesn't affect the proof.
\end{proof}

 \begin{lemma} \label{lemma:iterates-bounded}
 Assume that $\mu$ is non-elementary and that $\int_G \log^2 \|g\| \, \diff \mu(g) <\infty$. Fix $\xi \in \R$. Then the sequence $\|\oP_\xi^n\|_{W^{1,2}}$ is bounded for $n\geq 0$.
\end{lemma}
\begin{proof}
We'll work with the norm  $\|u\|_\nu = |\lp \nu, u \rp| +\frac12 \|\del u\|_{L^2} + \frac12 \|\del \overline u\|_{L^2}$, which is equivalent to $\|\cdot\|_{W^{1,2}}$ by \cite[Corollary 2.13]{DKW:PAMQ}.

Let $u \in W^{1,2}$ be such that $\|u\|_{W^{1,2}}\leq 1$. Using the invariance of $\nu$, we have 
\begin{equation} \label{e:abs-Pt}
 \langle\nu, |\oP^n_\xi u|\rangle \leq \langle\nu,\oP^n_0|u|\rangle =  \langle\nu, |u|\rangle,
\end{equation}
which is bounded by a constant. Therefore, we only need to show that $\|\partial \oP_\xi^n u\|_{L^2}$ and $\|\partial \overline{\oP_\xi^n u}\|_{L^2}$ are bounded for $n\geq 0$.  By Lemma \ref{lemma:Pt-L^2_1,0-contraction} we can replace $\oP_\xi$ by an iterate and assume that $\|\oP_\xi\|_{L^2_{(1,0)}}<1/2$. Then,   $\|\oP_\xi^n\partial u\|_{L^2}$ and  $\|\oP_\xi^n\partial \overline u\|_{L^2}$ tend to zero exponentially fast. Therefore,  in order to show that $\|\partial \oP_\xi^n u\|_{L^2}$ (resp. $\|\partial \overline{\oP_\xi^n u}\|_{L^2}$) is bounded in $n$, it is enough to show the same property for  $\|\Theta_\xi^{(n)}u \|_{L^2}$ (resp. $\|\Xi_\xi^{(n)}u \|_{L^2}$), where $ \Theta_\xi^{(n)}u : = \partial \oP_\xi^n u -\oP_\xi^n\partial u$ (resp. $ \Xi_\xi^{(n)}u : = \partial \overline{\oP_\xi^n u} -\oP_\xi^n\partial \overline u$). We'll only handle $\Theta_\xi^{(n)}u$, since the estimates for  $\Xi_\xi^{(n)}u$ are analogous.

We have
$$\oP_\xi^n u =\int_{G^n} e^{i\xi(\sigma_{g_1}+ g_1^* \sigma_{g_2}+\cdots+ g_1^*\ldots g_{n-1}^*\sigma_{g_n})} g_1^*\ldots g_n^*(u) \, \diff \mu^{\otimes n},$$
so, using that $g^* \partial u = \partial (g^*u)$ for all $g \in G$, we get $$\Theta_\xi^{(n)}u = \partial \oP_\xi^n u -\oP_\xi^n\partial u = i\xi \sum_{k=1}^n \mathcal R^{(k)}_\xi u, $$ where $$\mathcal R^{(k)}_\xi u  := \int_{G^n} e^{i\xi(\sigma_{g_1}+ g_1^* \sigma_{g_2}+\cdots+ g_1^*\ldots g_{n-1}^*\sigma_{g_n})} g_1^*\ldots g_n^*(u) \, g_1^*\ldots g_{k-1}^*(\partial \sigma_{g_k})\, \diff \mu^{\otimes n}.$$

The above expression can be rewritten as $$\mathcal R^{(k)}_\xi u = \oP_\xi^{k-1} \phi_{n-k},\quad \text{where} \quad \phi_{\ell}:= \int_G e^{i \xi \sigma_g} g^*(\oP_\xi^{\ell} u) \partial \sigma_g \, \diff \mu(g) \in L^2_{(1,0)}.$$

From the fact that $\|\oP_\xi\|_{L^2_{(1,0)}}<1/2$ we get
\begin{equation} \label{eq:R_t-contraction}
\|\mathcal R^{(k)}_\xi u\|_{L^2} \leq \frac{1}{2^{k-1}} \|\phi_{n-k}\|_{L^2}.
\end{equation}

By  Cauchy-Schwarz inequality, one has $$i \phi_\ell \wedge \overline {\phi_\ell} \leq i  \int_G g^*|\oP_\xi^{\ell} u|^2 \,  \partial \sigma_g \wedge \overline{\partial \sigma_g}\, \diff \mu(g)$$

Using that $g_* \sigma_{g} = - \sigma_{g^{-1}}$ and the fact that $g$ acts unitarily on $L^2_{(1,0)}$, we see that $\|\phi_\ell\|^2_{L^2}$ is bounded by the mass of the positive measure $$\int_G |\oP_\xi^{\ell} u|^2 \,  i \partial \sigma_{g^{-1}} \wedge \overline{\partial \sigma_{g^{-1}}}\, \diff \mu(g) = | \oP_\xi^\ell u|^2 \check \rho \, \omegaFS,$$
where $\check \rho: = \int_G \rho_{g^{-1}} \, \diff \mu(g)$. See Lemma \ref{lemma:sigma-estimates}-(3) for the notation. Notice that, from that lemma, we have $\check \rho \log \check \rho \in L^1(\omegaFS)$ and $\|\check \rho \log \check \rho\|_{L^1}\lesssim 1 + \int_G \log^2 \|g\| \, \diff \mu(g)$. Recall that $\|g^{-1}\| = \|g\|$.

From $|\oP_\xi^\ell u| \leq \oP_0^\ell |u|$ we have  $|\oP_\xi^\ell u|^2 \leq \big( \oP_0^\ell |u| \big)^2$, so
\begin{equation} \label{eq:phi-ell-phi-ell-bar}
\|\phi_\ell\|_{L^2}^2 \leq \int_{\P^1}  | \oP_\xi^\ell u|^2 \,  \check \rho \, \omegaFS \leq \int_{\P^1}  \big(\oP_0^\ell |u|\big)^2 \check \rho \, \omegaFS.
\end{equation}

Since $\|u\|_{W^{1,2}}\leq 1$ we have $\|\, |u| \, \|_{W^{1,2}}\leq M$ for some constant $M >0$ (see \cite[Proposition 4.1]{dinh-sibony:decay-correlations}). Hence $\|\, P^{n-k}_0 |u|\, \|_{W^{1,2}}\leq M$.  From Moser-Trudinger's estimate we have $\int_{\P^1} e^{\alpha (P_0^{n-k}|u|)^2} \, \omegaFS \leq A$ for some constants $\alpha, A >0$ independent of $u$, $n$ and $k$. Using Young's inequality $ab\leq e^a+b\log b-b \leq  e^a+b\log b$ for $a,b >0$ and \eqref{eq:phi-ell-phi-ell-bar} we get $$\|\phi_{n-k}\|_{L^2}^2 \leq  \alpha^{-1} \Big( \int_{\P^1} e^{\alpha (\oP_0^{n-k}|u|)^2} \, \omegaFS + \int_{\P^1} \check \rho \log \check \rho \, \omegaFS \, \Big),$$ so  $$\|\phi_{n-k}\|_{L^2} \lesssim  \Big(1 + \int_G \log^2 \|g\| \, \diff \mu(g)\Big)^{1/2},$$ which is a constant independent of $n,k$ and $u$.

We conclude, using (\ref{eq:R_t-contraction}), that
\begin{equation} \label{eq:norm-Q^n}
\|\Theta_\xi^{(n)} u\|_{L^2} \leq |\xi| \sum_{k=1}^n \|\mathcal R^{(k)}_\xi u\|_{L^2} \leq |\xi| \sum_{k=1}^n C \frac{1}{2^{k-1}} \leq 2 C |\xi|,
\end{equation}
showing that $\|\Theta_\xi^{(n)} u\|_{L^2} $ is bounded in $n$. This finishes the proof.
\end{proof} 

Observe that the previous proposition already implies that $\rho(\oP_\xi) \leq 1$ for every $\xi \in \R$. The main obstruction to the non-existence of eigenvalues of unit modulus of $\oP_\xi$ with $\xi \neq 0$ is a strong invariance of the cocycle $\sigma$. This invariance  is encoded by the so-called cohomological equation.  

\begin{definition} \rm 
 We say that $\sigma$ admits a \textit{(strong) solution} to the cohomological equation at a parameter $\xi \in \R$ if there exist a continuous function $\varphi$ on $\P^1$, not identically zero on $\supp \, \nu$, and a number $s \in [0,2 \pi)$ such that  
\begin{equation} \label{eq:cohomological}
 e^{i\xi \sigma_g(x)} \varphi (gx) = e^{is} \varphi(x) \quad \text{ for all }  (g,x) \in \supp\, \mu \times \supp \, \nu.
\end{equation}
 
 We say that $\sigma$ admits a \textit{weak solution} to the cohomological equation  at a parameter $\xi \in \R$ if the there exist a measurable function $\varphi$ on $\P^1$, not identically zero nor $\pm \infty$, $\nu$-almost everywhere, and a number $s \in [0,2 \pi)$ such that
\begin{equation} \label{eq:weak-cohomological}
 e^{i\xi \sigma_g(x)} \varphi (gx) = e^{is} \varphi(x) \quad \text{ for $\mu \otimes \nu$-almost all }  (g,x) \in G \times \P^1.
\end{equation}
\end{definition}

The following result is well-known and is related to some ``non-arithmeticity'' of the support of non-elementary measures.

\begin{proposition} \label{prop:non-arithmetic}
Let $\mu$ be a non-elementary probability measure on $\SL_2(\C)$. Then, for every $\xi \neq 0$, the cohomological equation (\ref{eq:cohomological}) admits no non-zero continuous solution.
\end{proposition}

\begin{proof}
See \cite[Proposition 3.2]{conze-guivarch} or \cite[Chapter 17]{benoist-quint:book}.
\end{proof}

In what follows,  we will prove several auxiliary results which aim to rule out the existence  of eigenvalues of unit modulus of $\oP_\xi$.  The idea is to use the non-existence of solutions to the cohomological equation given by Proposition \ref{prop:non-arithmetic}.  A main difficulty is that,  at first,  we work with Sobolev functions. This gives identities that hold only almost everywhere and provide only weak solutions.  However, using the estimates of Lemma \ref{lemma:iterates-bounded},  we can ultimately produce strong solutions,  which is done in the end of the proof of Theorem \ref{thm:P_t-contracting}.  When doing that,  we often need to apply similar argument more than once,  but using stronger forms of convergence at the each step.

 \begin{lemma} \label{lemma:|P_t^nu|=1}
 Let $\xi \in \R$. Let $\varphi$ be a measurable function such that $|\varphi|= 1 = |\oP_\xi^ n \varphi|$, $\nu$-almost everywhere for every $n \geq 1$. Then $\varphi$ is a weak solution to the cohomological equation  at $\xi$.
 \end{lemma}
 
 \begin{proof}
We begin by observing that $1 = |\oP_\xi \varphi | \leq \oP_0 |\varphi| = 1$, $\nu$-almost everywhere. This is only possible if $e^{i\xi \sigma_g(x)}\varphi(gx)$ is independent of $g$ almost surely for $\nu$-almost every $x$.

The equality $|\oP_\xi \varphi |  = 1 = |\varphi|$ implies that $\oP_\xi \varphi (x) = e^{is(x)} \varphi(x)$ for some real measurable function $s(x)$ with values in $[0,2\pi)$. We claim that $s$ is constant  $\nu$-almost everywhere. By hypothesis, we have  $1 = |\oP_\xi^2 \varphi| = |\oP_\xi(e^ {i s}\varphi)| \leq \oP_0|\varphi| = 1$, so  the last inequality must be an equality, that is  
$$\big|\oP_\xi^2 \varphi(x)\big| = \Big|\int_G e^ {i\xi \sigma_g(x)} \varphi(gx) e^{i s(gx)} \, \diff \mu(g)\Big| = 1.$$ As $e^{i\xi \sigma_g(x)}\varphi(gx)$ is independent of $g$ by the above,  we get that  $\big|\int_G  e^{i s(gx)} \, \diff \mu(g)\big| = 1$, which is only possible if $e^{is(gx)}$ is independent of $g$ almost surely  for $\nu$-almost every $x$. Repeating the above argument for every $n$ gives that $e^{is(gx)}$ is independent of $g$ almost surely with respect to $\mu^{\ast n}$  for $\nu$-almost every $x$.

Let $\Omega = G^{\N^ \ast}$ and consider the the Furstenberg map $Z:\Omega \to \P^1$ defined by the property that $\lim_{n \to \infty} g_1(\omega)_* \cdots g_n(\omega)_* \nu = \delta_{Z(\omega)}$, see \cite[III.3]{bougerol-lacroix}. We have $Z_* \mu^{\N^\ast}  = \nu$. Let $\widetilde s := s \circ Z$ and $T: \Omega \to \Omega$ be the Bernoulli shift. The above property of $s$ implies that  $e^{i\widetilde s}$ is constant almost surely on the fibers $T^{-n}(\omega)$ for every $n \geq 1$ and $\mu^{\N^\ast}$-almost every $\omega$. This invariance and the fact that $T$ is an exact endomorphism implies that  $e^{i \widetilde s}$ is constant $\mu^{\N^\ast}$-almost everywhere, see e.g. \cite{walters:ergodic-theory}. As $Z_* \mu^{\N^\ast}  = \nu$ one concludes that $e^{is}$ is constant $\nu$-almost everywhere. Since $s$ has values in $[0,2 \pi)$, we obtain that $s$ is constant $\nu$-almost everywhere.

We have seen above that $e^{i\xi \sigma_g(x)} \varphi(gx)$ is independent of $g$ almost surely  for $\nu$-almost every $x$. Together with $\oP_\xi \varphi (x) = e^{is(x)} \varphi(x)$ and the fact that $s$ is constant $\nu$-almost everywhere one concludes that   $$e^{i\xi \sigma_g(x)}\varphi(gx) =  e^{is} \varphi(x) ,\quad\mu \otimes \nu\text{-almost everywhere}.$$ In other words, $\varphi$ is a weak solution to the cohomological equation  at  $\xi$. This proves the lemma.
 \end{proof}
 
 \begin{lemma} \label{eq:coboundary-unique}
Let $\varphi_1$ and $\varphi_2$ be two  measurable functions such that $|\varphi_j| = 1$, $\nu$-almost everywhere, $j=1,2$. Assume that $\varphi_1$ and $\varphi_2$ satisfy (\ref{eq:weak-cohomological}) for some $s_1$, $s_2$, respectively, instead of $s$. Then there is a complex number $b$ such that $|b|=1$ and $\varphi_1 = b \,  \varphi_2$, $\nu$-almost everywhere. 
\end{lemma}

\begin{proof}
By assumption, we have $$ e^{i\xi \sigma_g(x)} \varphi_1 (gx) = e^{is_1} \varphi_1(x)\quad\text{and}\quad e^{i\xi \sigma_g(x)} \varphi_2 (gx) = e^{is_2} \varphi_2(x)$$ for $\mu \otimes \nu$-almost every $(g,x)$ and, after iterating and using the cocycle property of $\sigma_g$,  the same is true for  $\mu^{\ast n} \otimes \nu$-almost every $(g,x)$ for all $n\geq 1$, where $s_1,s_2$ are replaced by $ns_1,ns_2$ modulo $2\pi \Z$. Let $\psi := \varphi_1 \slash \varphi_2$. Then 
$$\psi(gx) = e^{in(s_1 - s_2)}\psi (x)  \quad\text{for }\mu^{\ast n} \otimes \nu\text{-almost every } (g,x) \text{ for all }n\geq 1.$$ 
In particular,  $\psi(gx)$ is independent of $g$ almost surely with respect to $\mu^{\ast n}$  for $\nu$-almost every $x$. Arguing as in the proof of Lemma \ref{lemma:|P_t^nu|=1} above  one concludes that $\psi$ equals a constant $b$, $\nu$-almost everywhere and $s_1 = s_2$. Therefore,  $\varphi_1 = b \varphi_2$, $\nu$-almost everywhere finishing the proof. 
\end{proof}

From now until the end of this section, we fix  $\xi \neq 0$. We also assume that $\mu$ is non-elementary and that $\int_G \log^2 \|g\| \, \diff \mu(g) <\infty$.
\medskip

 Let $\B$ be the unit ball in $W^{1,2}$ and set $$\mathcal F_\infty : = \big\{\varphi_\infty : \varphi_\infty \text{ is a weak limit of } (\oP_\xi^{n_\ell}\varphi_\ell  )_{\ell \geq 1} \text{  with } \varphi_\ell \in \B \text{ and } n_\ell \nearrow \infty \text{ as } \ell \to \infty \big\}.$$

See the Appendix for the notion of weak convergence in $W^{1,2}$. Since $\|\oP_\xi^n\|_{W^{1,2}}$  is uniformly bounded by Lemma \ref{lemma:iterates-bounded}, one has that $\mathcal F_\infty$ is a bounded and weakly compact subset of $W^{1,2}$.

\begin{lemma} \label{lemma:F-inftty-bounded}
Every function in $\mathcal F_\infty$ is bounded.
\end{lemma}

\begin{proof}
Let $\varphi_\infty = \lim \oP_\xi^{n_\ell} \varphi_{\ell} $ be an element of  $\mathcal F_\infty$, where we consider the weak limit in $W^{1,2}$.  Notice that, for every $k \geq 1$ there is $\psi_k \in  \mathcal F_\infty $ such that $\varphi_\infty = \oP_\xi^k \psi_k$. We can choose $\psi_k$ as a limit of $\oP_\xi^{n_\ell - k} \varphi_{\ell}$.  In particular, $|\varphi_\infty| \leq \oP_0^{k} |\psi_k|$. Since $\mathcal F_\infty$ is a bounded subset of $W^{1,2}$ and $\nu$ acts continuously on $W^{1,2}$, the sequences $\| \, |\psi_k| \, \|_{W^{1,2}}$  and $c_k = \lp \nu, |\psi_k| \rp$ are bounded by a constant. Up to passing to a subsequence, we may assume that $c_k$ converges to $c_\infty$. From Theorem \ref{thm:DKW-spectral-gap}, we have that $ \oP_0^{k} |\psi_k|$ converges in $W^{1,2}$-norm to $c_\infty$. It follows that $|\varphi_\infty| \leq c_\infty$. This finishes the proof.
\end{proof}

\begin{remark}
 Functions in $W^{1,2}$ are, by definition,  only defined up to a set of zero Lebesgue measure. However, they admit good representatives outside a polar set.  See the definition in the Appendix.   These are locally the sets included in the pole set of a subharmonic function and are the negligible sets from a complex analytic point view.  We note that polar sets are always of zero Lebesgue measure,  so they should be thought as ``thin'' sets.  We refer to the Appendix for more details. Every function appearing from now until the end of the section is a good representative. The results given in the Appendix show that this can be done rigorously. In particular, notice that for bounded functions, these representatives are preserved by the operators $\oP_\xi$ (Proposition \ref{prop:good-rep-P_t}).  Also, since $\nu$ does not charge polar sets, cf. \cite[Theorem 2.11]{DKW:PAMQ}, statements of the type ``$\varphi = \psi$, $\nu$-almost everywhere'' are meaningful and the value $\lp \nu, \varphi \rp$ coincides with the integral of a good representative of $\varphi$ against $\nu$, see Lemma \ref{lemma:int-good-rep}.
\end{remark}

Define $\mathcal F^M_\infty \subset \mathcal F_\infty$ as the set of functions $\varphi \in \mathcal F_\infty$ such that $\langle \nu, |\varphi| \rangle$ is maximal and denote by $\mathbf{m}$ the corresponding maximal value. As $ \mathcal F_\infty$ is weakly compact and $\nu$ is continuous with respect to the weak topology (cf.  Theorem \ref{thm:DKW-spectral-gap}), this maximum exists.

\begin{lemma} \label{lemma:maximal-functions}
Let $\varphi_\infty \in \mathcal F^M_\infty$. Then $|\varphi_\infty| \leq \mathbf m$ with equality $\nu$-almost everywhere. If $\mathbf m =0$ then $\mathcal F_\infty$ contains only the zero function. If $\mathbf m \neq 0$ then there exists a non-zero $\varphi \in W^{1,2}$ satisfying (\ref{eq:weak-cohomological}) and $|\varphi| =\mathbf m$, $\nu$-almost everywhere.
\end{lemma}

\begin{proof}
Let $\varphi_\infty = \lim \oP_\xi^{n_\ell} \varphi_{\ell} $ be an element of  $\mathcal F^M_\infty$.  As in the proof of the previous lemma, we have $\varphi_\infty = \oP_\xi^k \psi_k$ and $|\varphi_\infty| \leq \oP_0^{k} |\psi_k|$ for every $k \geq 1$ and some $\psi_k \in  \mathcal F_\infty $.

Since $\nu$ is stationary we have $\langle \nu, |\psi_k| \rangle = \langle \nu, \oP_0^{k } |\psi_k| \rangle  \geq \langle \nu, |\varphi_\infty| \rangle = \mathbf m$. By the maximality  of $\mathbf m$, we get $\langle \nu, |\psi_k| \rangle = \mathbf m$ and $\oP_0^{k } |\psi_k| = |\varphi_\infty|$, $\nu$-almost everywhere for every $k \geq 1$. Now, by  Theorem \ref{thm:DKW-spectral-gap}, $\oP_0^{k}|\psi_k|$ converges to $\lim_{k \to \infty} \langle \nu, |\psi_k| \rangle = \mathbf m$ in $W^{1,2}$-norm as $k \to \infty$. We conclude that $|\varphi_\infty| \leq \mathbf m$, which together with $\lp \nu, |\varphi_\infty| \rp = \mathbf m$ gives that $|\varphi_\infty| = \mathbf m$, $\nu$-almost everywhere, proving the first part of the lemma. The second assertion follows directly.

Suppose now that $\mathbf m \neq 0$. Let $\psi_\infty$ be a weak limit of $\psi_k$ as $k \to \infty$, where $\psi_k$ is as above. We have seen that $\psi_k \in \mathcal F^M_\infty$ for every $k$, so $\psi_\infty \in \mathcal F^M_\infty$. We claim that $\psi_\infty$ satisfies (\ref{eq:weak-cohomological}). In order to see that, notice first that, for $n \leq k$
\begin{equation*}
|\varphi_\infty|  = |\oP_\xi^k \psi_k| = |\oP_\xi^{k-n} \oP_\xi^n  \psi_k| \leq \oP_0^{k - n} |\oP^n_\xi \psi_k| \leq \oP_0^{k} |\mathcal \psi_k|.
\end{equation*}

From the maximality of $\mathbf m$ and the fact that both $\varphi_\infty$ and $\psi_k$ belong to $\mathcal F^M_\infty$,  we get that  $\oP_\xi^n \psi_k \in \mathcal F^M_\infty$ for $n \leq k$. Letting $k \to \infty$ and using the continuity of $\oP^n_\xi$ and $\nu$ with respect to the weak convergence  give that $\oP_\xi^n \psi_\infty \in \mathcal F^M_\infty$. From the first part of the lemma one has $ |\oP_\xi^n \psi_\infty| = \mathbf m$, $\nu$-almost everywhere. From Lemma \ref{lemma:|P_t^nu|=1}, we conclude that $\psi_\infty$ satisfies (\ref{eq:weak-cohomological}). This finishes the proof.
\end{proof}

From now on, assume that $\mathbf m \neq 0$. In the end, we'll show that this case cannot occur. We'll do so by contradiction by producing a non-trivial solution to (\ref{eq:cohomological}). Let $\varphi \in W^{1,2}$ be as in the preceding lemma.  Recall that we consider a good representative $\varphi$ as in Definition \ref{def:good-rep}.  Since $\varphi$ satisfies (\ref{eq:weak-cohomological}), one has $\oP_\xi \varphi = e^{is} \varphi$, $\nu$-almost everywhere. Let $\widetilde \oP_\xi := e^{-is} \oP_\xi$, so that $\widetilde \oP_\xi \varphi = \varphi$, $\nu$-almost everywhere. Recall that $|\varphi|= \mathbf m$, $\nu$-almost everywhere.

 Consider the subspace
\begin{equation} \label{eq:def-subspace-E}
E := \{ u \in W^{1,2} : \lp  \varphi \, \nu, \overline u \rp = 0\}.
\end{equation}
Notice that, since $|\varphi| = \mathbf m$, $\nu$-almost everywhere, $E$ is a complex hyperplane in $W^{1,2}$ not containing $\varphi$.

\begin{lemma} \label{lemma:phi-nu-continuous}
The functional $u \mapsto \lp \varphi \, \nu, \overline u \rp $ is continuous in $W^{1,2}$ with respect to both the norm topology and the weak topology in $W^{1,2}$. In particular, $E$ is a closed complex hyperplane of $W^{1,2}$ which is also closed with respect to the weak topology in $W^{1,2}$.   Moreover, $E$ is invariant under $\widetilde \oP_\xi$.
\end{lemma}

\begin{proof}
This first statement follows from Corollary \ref{c:WPC} from the Appendix. So, we only need to show the invariance of $E$ under $\widetilde \oP_\xi$. This is a consequence of the identity
\begin{equation} \label{eq:P_t-tilde-invariance}
\lp \varphi \nu, \overline u \rp = \lp \varphi \nu, \overline {\widetilde \oP_\xi u} \rp \quad \text{for all } u \in W^{1,2}.
\end{equation}

In order to obtain the above identity,  it is enough to consider $u$ smooth,  since such $u$ are dense in $W^{1,2}$.  As $\varphi$ is a good representative, $\varphi \overline u$ is a good representative and the value  $\lp \varphi \nu, \overline u \rp$ coincides with the integral of $\varphi \overline u$ with respect to $\nu$ (see Lemma \ref{lemma:int-good-rep} in the Appendix).  The same is true for $\overline {\widetilde \oP_\xi u}$ instead of $\overline u$, since the former function is continuous. Using these facts and  the cohomological equation (\ref{eq:weak-cohomological}) we get
\begin{align*}
\lp \varphi \nu, \overline {\widetilde \oP_\xi u} \rp & =\int_{\P^1}  \varphi(x) \overline {\widetilde \oP_\xi u}(x)  \, \diff \nu(x) = \int_{\P^1}   \varphi(x) e^{is}  \int_G e^{-i\xi \sigma_g(x)} \overline u(gx) \, \diff \mu(g) \diff \nu(x) \\ &=  \int_{\P^1} \int_G e^{i\xi \sigma_g(x)} \varphi (gx) e^{-i\xi \sigma_g(x)} \overline u(gx) \, \diff \mu(g) \diff \nu(x) = \lp \nu, \oP_0 (\varphi \overline u) \rp = \lp \varphi \nu, \overline u \rp,
\end{align*}
yielding \eqref{eq:P_t-tilde-invariance}. This finishes the proof of the lemma.
\end{proof}

\begin{lemma} \label{lemma:weak-convergence}
Let $u \in E$. Then $\widetilde \oP_\xi^n u \to 0$ weakly and uniformly in $u \in E \cap \B$ as $n \to \infty$. In particular, for every $p \geq 1$, we have that $\| \oP_\xi^n u\|_{L^p} \to 0$  uniformly in $u \in E \cap \B$.
\end{lemma}

\begin{proof}
We argue as in Lemma \ref{lemma:maximal-functions}, but we work on the invariant subspace $E$ instead of $W^{1,2}$. Notice that the extra factor $e^{is}$ in the definition of $\widetilde \oP_\xi$ and the complex conjugate in the definition of $E$ do not affect the identities and estimates from that lemma. 
Denote by $\mathcal G_\infty$ the family consisting of all weak limits of sequences of the type $(\widetilde \oP_\xi^{n_\ell} u_\ell  )_{\ell \geq 1}$ with $u_\ell \in E \cap \B$  and $n_\ell \nearrow \infty$. We need to show that $\mathcal G_\infty$ contains only the zero function.   Let $\mathcal G^M_\infty \subset \mathcal G_\infty$ be the set of functions $u \in \mathcal G_\infty$ such that $\langle \nu, |u| \rangle$ is maximal and denote by $\mathbf m'$ the corresponding maximal value. If $\mathbf m'=0$ the arguments of Lemma \ref{lemma:maximal-functions} give that $\mathcal G_\infty = \{0\}$ and the first part of the lemma follows.

Suppose now, by contradiction, that $\mathbf m' \neq 0$. As in Lemma  \ref{lemma:maximal-functions} we can produce a function $\eta \in \mathcal G_\infty$ satisfying (\ref{eq:weak-cohomological}). By Lemma \ref{lemma:phi-nu-continuous}, $\eta \in E$ and Lemma \ref{eq:coboundary-unique} gives that $\eta = b \, \varphi$, $\nu$-almost everywhere for some non-zero complex number $b$. In particular, this shows that $\varphi \in E$, which contradicts the definition of $E$. The first assertion follows.

The second assertion follows from  $\|\oP_\xi^n u\|_{L^p} = \| \widetilde \oP_\xi^{n} u\|_{L^p} $ and the fact that weak convergence in $W^{1,2}$ implies strong convergence in $L^p$ for every $p \geq 1$, see \cite{dinh-marinescu-vu}. 
\end{proof}

The next result upgrades weak convergence to strong convergence. 

\begin{lemma} \label{lemma:strong-convergence}
We have $\|\widetilde \oP_\xi^n u\|_{W^{1,2}} \to 0$ as $n \to \infty$ for every $u \in E$. The convergence is uniform in $E \cap \B$. In particular, we have that $\rho(\widetilde \oP_\xi|_E) < 1$.
\end{lemma}

\begin{proof}
We keep the same notations as in the proof Lemma \ref{lemma:iterates-bounded}, such as $\Theta_\xi^{(n)}$, $\phi_\ell$ and $\check \rho$. We need to show that $\|\Theta_\xi^{(n)} u\|_{L^2} \to 0$ as $n \to \infty$ uniformly in $E \cap \B$. Recall that $\|\phi_\ell\|_{L^2}$ is uniformly bounded and, from (\ref{eq:phi-ell-phi-ell-bar}),  $$\|\phi_\ell\|_{L^2}^2 \leq \int_{\P^1}  | \oP_\xi^\ell u|^2 \,  \check \rho \, \omegaFS.$$
\vskip5pt
\noindent\textbf{Claim.} $\|\phi_\ell\|_{L^2} \to 0$ as $\ell \to \infty$.
\proof[Proof of Claim]
Let $\varepsilon > 0$. From the above estimate we have that $$\|\phi_\ell\|^2_{L^2} \leq \int_{\P^1} | \oP_\xi^\ell u|^2 \, \check \rho \, \omegaFS = \int_{ \check \rho \leq M } | \oP_\xi^\ell u|^2 \, \check \rho \, \omegaFS + \int_{ \check \rho \geq M } | \oP_\xi^\ell u|^2 \, \check \rho \, \omegaFS, $$ where $M>0$ is a fixed large constant.

 Using Young's inequality as in the proof of Lemma \ref{lemma:iterates-bounded}, we see that the last integral in the right hand side is bounded by a constant times $$\int_{ \check \rho \geq M } e^{\frac{\alpha}{2} |\oP_\xi^\ell u|^2} \, \omegaFS + \int_{ \check \rho \geq M }  \check \rho \, \log \check \rho  \, \omegaFS.$$ The contribution of the last integral is $< \varepsilon \slash 4$ for $M$ large because $\check \rho \log \check \rho \in L^1$. By Cauchy-Schwarz inequality and Moser-Trudinger's estimate, the first integral  is bounded by 
 $$ \Big( \int  e^{\alpha |\oP_\xi^\ell u|^2} \, \omegaFS \Big)^{1 \slash 2}\cdot \text{Area}\{\check \rho \geq M\}^{1\slash 2} \leq C \,  \text{Area}\{\check \rho \geq M\}^{1 \slash 2},$$ which is also small for $M$ large because $\check \rho \in L^1$.

By the above discussion, we can choose $M> 0$ so that $\int_{ \check \rho \geq M } | \oP_\xi^\ell u|^2 \, \check \rho \, \omegaFS < \varepsilon \slash 2 $.   Now, $\int_{ \check \rho \leq M } | \oP_\xi^\ell u|^2 \, \check\rho \, \omegaFS$ is bounded by $M \| \oP_\xi^\ell u\|_{L^2}$, which tends to zero as $\ell \to \infty$ by Lemma \ref{lemma:weak-convergence}. It follows that  $\|\phi_\ell\|^2_{L^2} < \varepsilon$ for $\ell$ large enough. This proves the claim. \endproof

\vskip5pt

Now, from (\ref{eq:R_t-contraction}) and (\ref{eq:norm-Q^n}) we get
 \begin{equation*}
\|\Theta_\xi^{(n)} u\|_{L^2} \leq |\xi| \sum_{k=1}^n \|\mathcal R^{(k)}_\xi u\|_{L^2} \leq |\xi| \sum_{k=1}^n  \frac{\|\phi_{n-k}\|_{L^2}}{2^{k-1}} .
\end{equation*}

Let $\varepsilon > 0$ and pick $\ell_0$ such that $\|\phi_\ell\|_{L^2} < \varepsilon$ for $\ell \geq \ell_0$. This is possible by the above  claim. Then, for large $n$ we have $$\sum_{k=1}^n  \frac{\|\phi_{n-k}\|_{L^2}}{2^{k-1}} = \sum_{k=1}^{n-\ell_0}  \frac{\|\phi_{n-k}\|_{L^2}}{2^{k-1}} + \sum_{k=n- \ell_0 +1}^n  \frac{\|\phi_{n-k}\|_{L^2}}{2^{k-1}} \leq \varepsilon \sum_{k=1}^\infty \frac{1}{2^{k-1}} + C \sum_{k=n- \ell_0 +1}^\infty  \frac{1}{2^{k-1}} < 3 \varepsilon.$$

We conclude that $\|\Theta_\xi^{(n)} u\|_{L^2} \to 0$ as $n \to \infty$. This finishes the proof.
\end{proof}

The last conclusion of the above lemma gives a precise spectral description of the operator $\widetilde \oP_\xi$: it fixes $\varphi$ and has a small spectral radius in the invariant hyperplane $E$. In other words, $\widetilde \oP_\xi$ \textit{has a spectral gap}. In particular, the iterates of $\widetilde \oP_\xi$ converge exponentially fast  to the projection onto the one-dimensional vector space spanned by $\varphi$. This is summarized by the following corollary. Recall that we are assuming that $\mathbf m \neq 0$, which will be shown to be impossible afterwards.

\begin{corollary} \label{cor:convergence-to-phi}
Let $u \in W^{1,2}$. Then $\widetilde \oP_\xi^n  u \to c_u \varphi$ in $W^{1,2}$-norm as $n \to \infty$, where  $c_u := \frac{1}{\mathbf m^2} \lp  \varphi \, \nu, \overline u \rp$. The convergence is exponentially fast. In particular, $c_u = 0$ if and only if $u \in E$.
\end{corollary}

We can now finish the proof of Theorem \ref{thm:P_t-contracting}.

\begin{proof}[Proof of Theorem \ref{thm:P_t-contracting}]
Fix $\xi \neq 0$. Recall our notations: $\mathcal F_\infty$ is the family consisting of all weak limits of sequences of the type $(\oP_\xi^{n_\ell} \varphi_\ell  )_{\ell \geq 1}$ with $\varphi_\ell \in \B$  and $n_\ell \nearrow \infty$ as $\ell \to \infty$, and $\mathcal F^M_\infty \subset \mathcal F_\infty$ is the set of functions $\psi \in \mathcal F_\infty$ such that $\langle \nu, |\psi| \rangle =: \mathbf m $ is maximal.

\vskip5pt
\noindent\textbf{Claim.} $\mathbf m =0$.
\proof[Proof of Claim]
If not, Lemma \ref{lemma:maximal-functions} gives a  weak solution $\varphi$ to the cohomological equation (\ref{eq:weak-cohomological}). Using this $\varphi$, we can define $E$ as in (\ref{eq:def-subspace-E}), which is a proper subspace of $W^{1,2}$. Let $u$ be a $\Cc^1$ function such that $u \notin E$. In particular, $u \in W^{1,2} \cap \Cc^{\log^1}$ and $c_u \neq 0$. By Corollary \ref{cor:convergence-to-phi}, $\widetilde \oP_\xi^n  u \to c_u \varphi$ exponentially fast in $W^{1,2}$-norm as $n \to \infty$. This implies that  $\|\widetilde \oP_\xi^n  u- \widetilde \oP_\xi^{n-1} u\|_{W^{1,2}} \lesssim \lambda^n$ for some $0< \lambda <1$ and $c_u \varphi = u + \sum_{n=1}^\infty (\widetilde \oP_\xi^n  u- \widetilde \oP_\xi^{n-1} u)$. By Proposition \ref{prop:P_t-logp}, we have $[\widetilde \oP_\xi^n  u- \widetilde \oP_\xi^{n-1}u ]_{\log^1} \lesssim n^{2}$. From Corollary \ref{cor:logp-pre-equidistribution}, we conclude that 
$$\|\widetilde \oP_\xi^n  u- \widetilde \oP_\xi^{n-1}u\|_\infty \leq c \tau^n \quad \text{for some constants } c>0\text{ and }0<\tau<1.$$ 
In particular, the series $\sum_{n=1}^\infty (\widetilde \oP_\xi^n  u- \widetilde \oP_\xi^{n-1} u)$ defines a continuous function. We conclude that $\varphi$ is continuous.  As (\ref{eq:weak-cohomological}) holds $\mu \otimes \nu$-almost everywhere, the same identity should hold for all $(g,x)$ in the support of $\mu \otimes \nu$. In other words,  (\ref{eq:cohomological}) holds,  contradicting Proposition \ref{prop:non-arithmetic}. This contradiction shows that $\mathbf m = 0$, proving the claim. \endproof

\vskip5pt

As $\mathbf m =0$, Lemma \ref{lemma:maximal-functions} shows that $\mathcal F_\infty$ contains only the zero function.   In other words,  $\oP_\xi^n  u \to 0$ weakly and uniformly in $u \in \B$ as $n \to \infty$. By \cite{dinh-marinescu-vu} one gets that $\| \oP_\xi^n u\|_{L^p} \to 0$  as $n \to \infty$ for every $p \geq 1$, uniformly in $u \in \B$.  Repeating the arguments of the proof of Lemma  \ref{lemma:strong-convergence} for $u$ in $W^{1,2}$ instead of $E$ gives that $\rho(\oP_\xi) < 1$. This proves the first part of the theorem.

 From the continuity of the family $\xi \mapsto \oP_\xi$ acting on $W^{1,2}$, cf. Proposition \ref{prop:P_t-regularity},  one gets that $\rho(\oP_\xi)$ is bounded away from one for $\xi$ in any compact subset of $\R \setminus \{0\}$.  The last assertion of the theorem follows.
\end{proof}

\section{The function space $\W$} \label{sec:W}

In this section, we introduce a new space of functions on $\P^1$, denoted by $\W$, contained both in the Sobolev space $W^{1,2}$ and in the space of continuous functions. The operators $\oP_\xi$  will act continuously on $\W$ and their spectral properties persist. The advantage of working with $\W$ instead of just $W^{1,2}$ is that it automatically yields pointwise and uniform estimates, which are needed in the proof of our main theorems.

\vskip5pt

The definition of the norm $\W$ will mix the Sobolev norm and the $\log^{p-1}$-norm,  where $p$ is such that $\mu$ has a finite moment of order $p$, and will also take into account the spectral gap of $\oP = \oP_0$ on $W^{1,2}$. The main results  are summarized in the following theorem.

\begin{theorem} \label{thm:spectral-gap-W}
Let $\mu$ be a non-elementary probability measure on $G =\SL_2(\C)$. Assume that $\int_G \log^p \|g\| \, \diff \mu(g) <\infty$, where $p > 3 \slash 2$. Set $\oN u := \int_{\P^1} u \, \diff \nu$, where $\nu$ is the unique stationary measure. Then, there exists a Banach space $( \W, \|  \cdot \|_\W)$ of functions in $\P^1$ with the following properties: 

\begin{enumerate}
\item The following inequalities hold: $\| \cdot \|_{W^{1,2}} \lesssim \| \cdot \|_\W $, $\| \cdot \|_{\infty} \lesssim \| \cdot \|_\W $ and  $ \| \cdot \|_\W \lesssim \max \lbrace \| \cdot \|_{W^{1,2}}, \| \cdot \|_{\log^{p-1}} \rbrace$. In particular, we have continuous embeddings $$ W^ {1,2} \cap \Cc^{\log^{p-1}} \subset\W \subset W^ {1,2}\cap \Cc^0;$$

\item There exists an $n_0 \geq 1$ such that $\| \oP^{n_0} - \oN\|_\W < 1$;
\item For every $\xi \in \R$, $\oP_\xi$ defines a bounded operator on $\W$; 
\item The family $\xi \mapsto \oP_\xi$ acting on  $\W$ is  continuous.
\end{enumerate}
\end{theorem}

Let us set up the proof of Theorem \ref{thm:spectral-gap-W}. Set $\oQ: = \oP - \oN$. Recall that, under our hypothesis, $\oQ$ acts continuously both on $W^{1,2}$ and $\Cc^{\log^{p-1}}$, see Theorem \ref{thm:DKW-spectral-gap} and Proposition \ref{prop:Pu-logp}. Observe that, as $\oN u$ is a constant function, we have $[\oQ u]_{\log^{p-1}} = [\oP u]_{\log^{p-1}}$ for $u \in \Cc^{\log^{p-1}}$. By Theorem \ref{thm:DKW-spectral-gap}  we can assume, up to taking iterates, that $$   \varrho_0 := \|\oQ\|_{W^{1,2}}   \leq 1 \slash 4 . $$

From  Proposition \ref{prop:Pu-logp}  there is a constant $C>0$ independent of $\mu$ such that 
 $$\|\oQ\|_{\log^{p-1}} \leq  B, \quad \text{where} \quad B: = C(1+M_{p-1}(\mu)).$$

 Set $\lambda:= 2\varrho_0$ and $\kappa := 2B$. If we replace $\oP$ by $\oP^n$, $\|\oQ\|_{W^{1,2}}$ is replaced by $\|\oQ^n\|_{W^{1,2}} \leq \varrho_0^n$ and $B$ is replaced by $C(1+M_{p-1}(\mu^{\ast n})) \lesssim n^{p-1}$,  see the proof of Corollary \ref{cor:logp-equidistribution}. Hence, by taking $n$ sufficiently large, we may assume that $\kappa <  \lambda^{3-2p}$. Recall that we are assuming $p>3\slash 2$.
 
For a measurable function $u$ on $\P^1$ we define
\begin{equation} \label{eq:def-norm-W}
\|u\|_\W := \inf \bigg \lbrace c>0 : u = c \, \sum_{n=0}^\infty u_n,  \quad \|u_n\|_{W^{1,2}}  \leq \lambda^n \, \text{ and } \|u_n\|_{\log^{p-1}} \leq \kappa^n, \, \forall n \geq 0 \bigg \rbrace.
\end{equation}

If a constant $c>0$ as above does not exist we set $\|u\|_\W := +\infty$. The space $\W$ is defined as the set of measurable functions $u$ such that $\|u\|_\W < +\infty$. Observe that the $u_n$ above are automatically continuous and belong to $W^{1,2}$, 

Observe that $\|u \|_{W^{1,2}} \leq 2 \| u  \|_\W $ for $u \in \W$.  Indeed,  if  $u \in\W$, we can write  $u = c \sum u_n$ as in (\ref{eq:def-norm-W}). As $\|u_n\|_{W^{1,2}} \leq \lambda^n$ and $\lambda = 2 \varrho_0 \leq 1 \slash 2$, the sum converges in $W^{1,2}$ and $\|u\|_{W^{1,2}} \leq 2c$. As $\|u\|_\W$ is the infimum of such $c$, we see that $\|u\|_{W^{1,2}} \leq 2 \|u\|_\W$.

The following lemma is analogous to Corollary \ref{cor:logp-pre-equidistribution}.

\begin{lemma} \label{lemma:sobolev-log^p-exponential}
Let $p,\lambda,\kappa$ be as above. Let $u$  be a   function in  $W^{1,2} \, \cap \, \Cc^{\log^{p-1}}$. Assume that $\|u\|_{W^{1,2}} \leq   \lambda^n$ and $ \|u\|_{\log^{p-1}} \leq  \kappa^n$ for some $n \geq 1$. Then, there are constants $c > 0$ and $0< \tau < 1$ independent of $u$ and $n$ such that  $$\|u \|_\infty \leq c \tau^n.$$ 
\end{lemma}

\begin{proof}
Let $\widetilde u := \lambda^{-n} u$. Then $u$ belongs to the unit ball of $W^{1,2}$. By hypothesis  $\|\widetilde u\|_{\log^{p-1}} \leq \lambda^{-n} \kappa^n = (\lambda^{-1} \kappa)^n$. From Proposition \ref{prop:moser-trudinger-logp}, there is a constant $A >0$ independent of $u$ and $n$ such that $\|\widetilde u\|_\infty \leq A (\lambda^{-1} \kappa)^{n \slash (2p-2)} $. This gives $$\| u\|_\infty \leq \lambda^n  A (\lambda^{-1} \kappa)^{n \slash (2p-2)} = A  (\lambda^{2p-3} \kappa)^{n \slash (2p-2)}.$$

Setting $\tau:= (\lambda^{2p-3} \kappa)^{1 \slash (2p-2)}$ and recalling that  $\kappa < \lambda^{3-2p}$  yields the result.
\end{proof}

\begin{corollary} \label{cor:uniform-W}
Let $u \in \W$. Then $u$ is continuous. Moreover, there is a constant $D >0$ such that $\|u\|_\infty \leq  D \|u\|_\W$ for all $u \in \W$.
\end{corollary}

\begin{proof}
We can assume $\norm{u}_{\mathscr W}=1/2$ and write $u =  \sum_{n=0}^\infty u_n$ with $\|u_n\|_{W^{1,2}}  \leq \lambda^n$ and $  \|u_n\|_{\log^{p-1}} \leq \kappa^n$. By Lemma \ref{lemma:sobolev-log^p-exponential}, one has $\|u_n \|_\infty \leq c \tau^n$ and $0<\tau <1$ so the series converges uniformly. In particular, $u$ is continuous and $\|u\|_\infty \leq D/2$, for some constant $D >0$ independent of $u$. Thus, $\|u\|_\infty \leq D \|u\|_\W$ and the proof is complete.
\end{proof}

We can now prove  Theorem \ref{thm:spectral-gap-W}.

\begin{proof}[Proof of Theorem \ref{thm:spectral-gap-W}]
The proof of first inequality in (1) was already obtained after the definition of $\| \cdot \|_\W$.   The second inequality is the content of Corollary \ref{cor:uniform-W} above. For the third inequality, we notice that if $u \in W^ {1,2} \cap \Cc^{\log^{p-1}}$ and $ \max \lbrace \| u \|_{W^{1,2}}, \| u \|_{\log^{p-1}} \} = c \neq 0$ then we can write $u = c \sum_{n \geq 0} u_n$ where $u_0 = c^{-1} u$ and $u_n = 0$ for $n \geq 1$. Then  $\|u_n\|_{W^{1,2}}  \leq \lambda^n$ and $\|u_n\|_{\log^{p-1}} \leq \kappa^n$ for $n \geq 0$, giving $\|u\|_\W \leq c =  \max \lbrace \| u \|_{W^{1,2}}, \| u\|_{\log^{p-1}} \} $. This concludes the proof of (1).   The completeness of $\W$ can be easily obtained from the completeness of $W^{1,2}$ and $\Cc^{\log^{p-1}}$.

Let us now prove (2). We claim that $\|\oQ\|_\W \leq 1 \slash 2$, where $\oQ=\oP-\oN$. Recall that in the set-up of the proof we have already replaced $\oP$ by an iterate. It is enough to show that if  $u = c \sum u_n$ as in (\ref{eq:def-norm-W}) and $v = \oQ u$ then $\|v\|_\W \leq c \slash 2$. 

We have
\begin{equation} \label{eq:v=Qu}
v = \oQ u = c \, \sum_{n=0}^\infty \oQ u_n = \frac{c}{2} \, \sum_{n=0}^\infty 2 \oQ u_n = \frac{c}{2} \, \sum_{n=0}^\infty v_n,
\end{equation}
 where $v_0:= 0$ and $v_n:= 2 \oQ u_{n-1}$ for $n \geq 1$.
 
 As $\|\oQ\|_{W^{1,2}} = \varrho_0 = \lambda \slash 2$ we get $$\|v_n\|_{W^{1,2}} = 2\|\oQ u_{n-1}\|_{W^{1,2}} \leq 2 \varrho_0 \lambda^{n-1} = \lambda^n$$ and from $\|\oQ\|_{\log^{p-1}} \leq B$ and $\kappa = 2B  $  we get $$\|v_n\|_{\log^{p-1}} = 2\|\oQ u_{n-1}\|_{\log^{p-1}} \leq  2 B \kappa^{n-1} = \kappa^n.$$
 
  Equation (\ref{eq:v=Qu}) together with the definition of $\|\cdot\|_\W$ in (\ref{eq:def-norm-W}) gives $\|v\|_\W \leq c \slash 2$, which is what we wanted. This proves (2).
  
  It follows directly from the definition of $\|\cdot\|_\W$ that if an operator $\mathcal T$ acts continuously both on $W^{1,2}$ and $\Cc^{\log^{p-1}}$ then  it also acts continuously on $\W$ and 
\begin{equation} \label{eq:op-norm-W}
\| \mathcal T\|_\W \leq \max \{ \| \mathcal T\|_{W^{1,2}}, \| \mathcal T\|_{\log^{p-1}} \}.
\end{equation}
 
 Then,  the fact that $\oP_\xi: \W \to \W$ is bounded follows directly from the   fact that it is bounded both in $W^{1,2}$ and in $\Cc^{\log^{p-1}}$ (cf. Propositions \ref{prop:P_t-bounded} and \ref{prop:P_t-logp}). This gives (3).

Applying \eqref{eq:op-norm-W}  to $\mathcal T = \oP_\xi - \oP_{\xi_0}$ when $\xi \to \xi_0$ and using the continuity of $\xi \mapsto \oP_\xi$ both in $W^{1,2}$ and in $\Cc^{\log^{p-1}}$ (cf.   Propositions \ref{prop:P_t-regularity} and \ref{prop:P_t-logp-regularity}) gives (4). This completes the proof of the theorem. 
\end{proof}

We now have an analogue of Corollary \ref{cor:P_t-decomp}. 

\begin{corollary} \label{cor:P_t-decomp-W}
Let $\mu$, $p$ and $\W$ be as in Theorem \ref{thm:spectral-gap-W}. Then, the decomposition (\ref{eq:P_t-decomp}) holds on $\W$. More precisely,
there exists an $\epsilon_1 > 0$ such that, for $\xi \in [-\epsilon_1,\epsilon_1]$ one has a decomposition
\begin{equation} \label{eq:P_t-decomp-W}
\oP_\xi = \lambda_\xi \oN_\xi + \oQ_\xi,
\end{equation}
 as bounded operators on $\W$, where $\lambda_\xi$ is as in Corollary \ref{cor:P_t-decomp} and $\oN_\xi$ and $\oQ_\xi$ are the restriction to $\W$ of the operators on $W^{1,2}$ given by Corollary \ref{cor:P_t-decomp}.
 
 Furthermore,
  
\begin{enumerate}
\item $\lambda_0 = 1$ and $\oN_0 u = \int u \, \diff \nu$, where $\nu$ is the unique $\mu$-stationary measure;
\item $\rho':= \displaystyle \lim_{n \to \infty } \|\oP_0^n - \oN_0\|_\W^{1 \slash n} < 1$;

\item $\lambda_\xi$ is the unique eigenvalue of maximum modulus of $\oP_\xi$ in $\W$, $\oN_\xi$ is a rank-one projection and $\oN_\xi \oQ_\xi = \oQ_\xi \oN_\xi = 0$;

\item the maps $\xi \mapsto \lambda_\xi$,  $\xi \mapsto \oN_\xi$ and $\xi \mapsto \oQ_\xi$ are continuous as operators on $\W$;

\item  $|\lambda_\xi| > \frac{2 + \rho'}{3}$ and there exists a constant $c > 0$ such that $\|\oQ_\xi^n\|_\W \leq c \Big( \frac{1 + 2 \rho'}{3} \Big)^n$  for every  $n \geq 0.$

\item If $p \geq 2$, then the map $\xi \mapsto \lambda_\xi$ is differentiable and has an expansion
\begin{equation} \label{eq:lambda-expansion-2}
\lambda_\xi = 1 + i  \gamma \xi - A \frac{\xi^2}{2}+ \psi(\xi),
\end{equation}
for a  continuous function $\psi$  such that $
\lim_{\xi\to 0 }\xi^{-2}\psi(\xi) =0$ and $A = a^2 + \gamma^2$, where $a > 0$ is the variance in the CLT (Theorem \ref{thm:CLT}) and $\gamma$ is the Lyapunov exponent. 
\end{enumerate}
\end{corollary}

\begin{proof}
By the properties of $\oP_\xi$ given by Theorem \ref{thm:spectral-gap-W} and the general theory of perturbations of linear operators \cite{kato:book}, one has a decomposition $\oP_\xi = \alpha_\xi \mathcal M_\xi + \mathcal R_\xi$, where $\alpha_\xi$, $\mathcal M_\xi$ and $\mathcal R_\xi$ satisfy the analogues of (1)--(5). We claim that $\alpha_\xi = \lambda_\xi$, $\mathcal M_\xi = \oN_\xi$ and $\mathcal R_\xi = \oQ_\xi$. This will give assertions (1)--(5). 

Let $\varphi_\xi = \oN_\xi \mathbf 1 \in W^{1,2}$ be the eigenfunction associated with $\lambda_\xi$ and $\psi_\xi = \mathcal M_\xi \mathbf 1 \in \W \subset W^ {1,2}$  be the eigenfunction associated with $\alpha_\xi$.   It follows from the decomposition (\ref{eq:P_t-decomp}) and the statements of Corollary \ref{cor:P_t-decomp}  that $\lambda_\xi$ is the only eigenvalue of $\oP_\xi$ in $W^{1,2}$ that is close to one for $\xi$ sufficiently small. This implies that $\alpha_\xi = \lambda_\xi$ for $\xi$ near zero and consequently $\varphi_\xi = \psi_\xi$ up to a constant multiple, as the associated eigenspaces of $\oN_\xi$ and $\mathcal M_\xi$ are both one-dimensional. Therefore, the projections $\oN_\xi$ and $\mathcal M_\xi$ agree on $\W$ for $\xi$ near zero and $\oQ_\xi = \oP_\xi - \lambda_\xi \oN_\xi = \oP_\xi - \alpha_\xi \mathcal M_\xi = \mathcal R_\xi$. This completes the proof of assertions (1)--(5).

As $\lambda_\xi$ coincides with the leading eigenvalue of $\oP_\xi$ acting on $W^{1,2}$, (6) follows from Lemmas \ref{lemma:lambda_t-expansion} and \ref{lemma:lambda-expansion-2}.
\end{proof}

In the proofs of Theorems \ref{thm:LLT-coeff}, \ref{thm:berry-esseen}  and \ref{thm:BE-coeff} in the forthcoming sections we'll need the following stronger regularity for $\oP_\xi$ acting on $\W$.

\begin{lemma} \label{lemma:Pt-regularity-3-moments}
Let $\mu$ be a non-elementary probability measure on $G=\SL_2(\C)$. Assume that $\int_G \log^{p+1} \|g\| \, \diff \mu(g) <\infty$, where $p > 3 \slash 2$. Let $\W$ be as in Theorem \ref{thm:spectral-gap-W}. Then the family $\xi \mapsto \oP_\xi$ acting on $\W$ is locally Lipschitz continuous in $\xi$. The same is true for $\xi \mapsto \oN_\xi$. 
\end{lemma}

\begin{proof}
In view of \eqref{eq:op-norm-W}, in order to prove that $\xi \mapsto \oP_\xi$ is locally Lipschitz in $\W$ it is enough to show that the same property holds both in $W^{1,2}$ and in $\Cc^{\log^{p-1}}$. Since $p+1 > 5 \slash 2$ we know from Proposition \ref{prop:P_t-regularity} that $\xi \mapsto \oP_\xi$ is  differentiable with locally bounded derivative as a family of operators on $W^{1,2}$. In particular, this family is  locally Lipschitz continuous. Now, Proposition \ref{prop:P_t-logp-regularity} gives that $\xi \mapsto \oP_\xi$ is locally Lipschitz continuous in $\Cc^{\log^{p-1}}$. The same property for $\W$ follows. 

The operator $\oN_\xi$ is given by the integral $(2\pi i)^{-1} \int_\Gamma R_\xi(z) \diff z$ where $\Gamma$ is a small circle enclosing $1$ and $R_\xi(z) := (z - \oP_\xi)^{-1}$ is the resolvent of $\oP_\xi$, see \cite{kato:book}. Since $\xi \mapsto \oP_\xi$ is locally Lipschitz by the first part of the lemma, the same property is true for $\xi \mapsto \oN_\xi$. This completes the proof of the lemma.
\end{proof}

\section{LLT for the norm cocycle} \label{sec:LLT-norm}

This section is devoted to the proof of Theorem \ref{thm:LLT}. We follow Le Page's approach, \cite{lepage:theoremes-limites}, which is inspired by the proof of the classical LLT for sums of i.i.d.'s via Fourier transforms. The key ingredients are the fact that the spectral radius of $\oP_\xi$ acting on $W^{1,2}$  is smaller than one for $\xi \in \R \setminus \{0\}$, cf. Theorem \ref{thm:P_t-contracting}, and the spectral properties of $\oP_\xi$ acting on $\W$ near $\xi=0$ obtained in Section \ref{sec:W}, which will yield pointwise estimates.

Let $\mathscr H$ be the space of  functions $\psi \in L^1(\R)$ such that $\psi$ is the inverse Fourier transform of a continuous function $\widehat \psi$ with compact support in $\R$, i.e., $\psi(u) = (2 \pi)^{-1} \int_{-\infty}^{+\infty} e^{iu\xi} \widehat \psi(\xi) \, \diff \xi$. We prove first a version of Theorem \ref{thm:LLT}.

\begin{lemma}\label{lemma:LLT-product}
	Let $\psi\in\mathscr H$ and $\varphi\in \Cc^1(\P^1)\subseteq \mathscr W$, then the conclusion of Theorem \ref{thm:LLT} holds for $f(u,x)=\psi(u)\cdot\varphi(x)$.
\end{lemma}

\begin{proof}
Replacing $u$ by $t+\sigma(S_n,x)-n\gamma$ in the Fourier inversion  formula for $\psi$ and using Fubini's theorem, we get
\begin{align*}
& \E \Big ( \psi\big( t + \sigma(S_n,x) - n \gamma\big) \cdot \varphi( S_n \cdot  x) \Big ) =\int_{G}  \psi\big( t + \sigma(g,x) - n \gamma\big) \cdot \varphi( g  x)  \diff\mu^{*n}(g)
\\  
&=\frac{1}{2 \pi} \int_{-\infty}^{+ \infty}  \widehat \psi(\xi) \int_{G} e^{i \xi  (t+\sigma(g,x) - n \gamma)} \varphi(gx) \diff \mu^{*n}(g) \diff \xi  =\frac{1}{2 \pi} \int_{-\infty}^{+ \infty} e^{i\xi t}\, \widehat \psi(\xi) e^{-i \xi n \gamma} \oP^n_{\xi} \varphi(x)  \diff \xi,
\end{align*}
where we have used \eqref{eq:P_t^n}.  Recall that $\widehat \psi$ is continuous and compactly supported, so the use of Fubini's theorem is justified.

 Let $a^2 > 0$ be the variance in the CLT for the norm cocycle (Theorem \ref{thm:CLT}).  Using that $\widehat \psi(0)=\int_{-\infty}^{+ \infty} \psi(s)\,\diff s$ and that the inverse Fourier transform of $2\pi e^{-\xi^2/2}$ is $\sqrt{2\pi}e^{-s^2/2}$, we get
\begin{equation*}
\sqrt{2 \pi} e^{-\frac{t^2}{2a^2  n}} \int_{-\infty}^{+\infty} \psi(s) \,  \diff s = a\,\widehat{\psi}(0) \int_{-\infty}^{+\infty} e^{i\frac{\xi t}{\sqrt n}}  e^{-\frac{a^2\xi^2}{2}} \,\diff \xi.
\end{equation*}

Let
\begin{equation*} 
T_n(t,x): =   \sqrt{2 \pi n}\, a \,  \E \Big (\psi\big( t + \sigma (S_n, x) - n \gamma\big) \cdot \varphi( S_n \cdot  x) \Big )  - e^{-\frac{t^2}{2 a^2 n}}  \int_{\P^1} \int_{-\infty}^{+\infty} \psi(s) \varphi(y)\, \diff s  \diff \nu( y)
\end{equation*}
be the expression on the left hand side of \eqref{eq:LLT}. Our goal is to prove that $|T_n(t,x)|$ tends to zero uniformly in $(t,x) \in \R \times \P^1$ as $n$ tends to infinity.  

Fix $\varepsilon >0$. We'll show that $|T_n(t,x)| < \varepsilon$ for $n$ large independently of $(t,x) \in \R \times \P^1$.  From the above computations we have
\begin{align*}
{\sqrt{2 \pi}\over a} \, T_n(t,x) = \sqrt {n}  \int_{-\infty}^{+\infty} e^{i\xi t} \widehat \psi(\xi)  e^{-i \xi n \gamma} \oP^n_{\xi} \varphi (x) \, \diff \xi   - \widehat{\psi}(0) \int_{-\infty}^{+\infty} e^{i\frac{\xi t}{ \sqrt n}}  e^{-\frac{a^2\xi^2}{2}} \,\diff \xi \int_{\P^1} \varphi \, \diff \nu.
\end{align*}

Let $M >0$ be such that the support of $\widehat \psi$ is contained in $[-M , M]$. From the expansion $\lambda_\xi = 1 + i\gamma \xi  - (a^2 +\gamma^ 2)\xi^2/2 + o(\xi^2)$ given by Corollary \ref{cor:P_t-decomp-W}, we obtain $$|\lambda_\xi| = \sqrt{\lambda_\xi \overline{\lambda_\xi}}= \sqrt{1 - a^2 \xi^2  + o(\xi^2)} = 1 -\frac{a^2 \xi^2}{2} + o(\xi^2) =  e^{-a^2 \xi^2 \slash 2 + o(\xi^2)}.$$ We can thus find   $0 < \xi_0  < \epsilon_1$, where $\epsilon_1$ is the constant from Corollary \ref{cor:P_t-decomp-W}, such that $\xi_0<M$ and   
\begin{equation} \label{eq:exp-times-lambda-n}
\Big| e^{-i \xi \sqrt n \, \gamma } \lambda_{\frac{\xi}{ \sqrt n}}^n \Big| \leq e^{-\frac{a^2\xi^2}{4}}  \quad \text{ for} \quad  |\xi| < \xi_0   \sqrt n. 
\end{equation}

By  Theorem \ref{thm:spectral-gap-W}-(1),  there is a constant $D>0$ such that $\| \cdot \|_{\infty} \leq D \| \cdot \|_\W $. Set $\varepsilon' := \varepsilon \slash ( 10 D \sqrt \pi/a  \|\widehat \psi\|_\infty \|\varphi\|_\W).$ By taking a smaller $\xi_0 $ if necessary we also have, by the continuity of $\xi \mapsto \oN_\xi$ (cf. Corollary \ref{cor:P_t-decomp-W}-(4)), that
\begin{equation} \label{eq:N_s-nu}
\big \| \oN_{\frac{\xi}{\sqrt n}}  - \oN_0 \big \|_\W  \leq  \varepsilon' \quad \text{ for} \quad  |\xi| < \xi_0   \sqrt n. 
\end{equation}

Recall from Corollary \ref{cor:P_t-decomp-W} that $\oP^n_\xi = \lambda^n_\xi \oN_\xi + \oQ^n_\xi$  as bounded operators on $\W$ for $|\xi| \leq \epsilon_1$.
 We  now split $\sqrt{2\pi} /a\, T_n$ into five parts
$${\sqrt{2\pi}\over a} \,  T_n(t,x): = T^1_n(t) + T^2_n(t,x) + T^3_n(t,x) + T^4_n(t,x) + T^5_n(t),$$ where
$$T^1_n(t): = \sqrt n  \int_{|\xi| \leq \xi_0 } e^{i \xi t} \widehat \psi (\xi) \, e^{-in \xi \gamma} \lambda_\xi^n  \int_{\P^1} \varphi \, \diff \nu  \diff \xi - \widehat{\psi}(0) \int_{|\xi| < \xi_0  \sqrt n} e^{i\frac{\xi t}{ \sqrt n}}  e^{-\frac{a^2\xi^2}{2}}\,\diff \xi \int_{\P^1} \varphi \, \diff \nu,$$

$$T^2_n(t,x): =  \sqrt n   \int_{|\xi| \leq \xi_0 } e^{i \xi t} \widehat \psi (\xi) \, e^{-in \xi \gamma} \lambda_\xi^n  \Big(  \oN_ \xi \varphi (x) - \int_{\P^1} \varphi \, \diff \nu \Big) \, \diff \xi, $$

$$T^3_n(t,x) :=  \sqrt n  \int_{|\xi| \leq \xi_0 } e^{i \xi t} \widehat \psi (\xi) \, e^{-in \xi \gamma} \,   \oQ^n_\xi \varphi(x) \,\diff \xi, $$

$$T^4_n(t,x): = \sqrt n  \int_{\xi_0  \leq |\xi| \leq M} e^{i \xi t} \widehat \psi(\xi) e^{-i n \xi \gamma} \oP^n_\xi \varphi (x) \, \diff \xi, $$ and 
$$T^5_n(t) := -  \widehat \psi (0) \int_{|\xi| \geq \xi_0  \sqrt n} e^{i\frac{\xi t}{ \sqrt n}}  e^{-\frac{a^2\xi^2}{2}}\,\diff \xi \int_{\P^1} \varphi \, \diff \nu.$$
Recall that $\widehat \psi$ has support in $[-M,M]$. 

We will show separately that  $|T^j_n| < \varepsilon \slash 5$, $j=1,\ldots,5$, uniformly in $(t,x)$ for $n$ large.

\vskip5pt

(1) -- Using the change of coordinates $\xi \mapsto \xi/ \sqrt n$ in the first integral defining $T^1_n$, we get
$$T^1_n(t) = \int_{|\xi| < \xi_0 \sqrt n} e^{i  \frac{\xi t}{ \sqrt n}} \Big[ \widehat{\psi} \big( \frac{\xi}{ \sqrt n}\big) e^{ -i\sqrt n \xi \gamma} \lambda_{ \frac{\xi}{ \sqrt n}}^n - \widehat \psi(0) e^{-\frac{a^2 \xi^2}{2}} \Big] \, \diff \xi \int_{\P^1} \varphi \, \diff \nu.$$

Using the expansion $\lambda_{\xi} = 1 + i \gamma \xi - (a^2 + \gamma^2)\xi^2/2  + o(\xi^2)$  again,  we get that $\lambda_{\xi} = e^{i\gamma \xi - \frac{a^2}{2} \xi^2 + o(\xi^2)}$ for $\xi$ small, so $\lambda_\xi^n = e^{i n \gamma \xi - n \frac{a^2}{2} \xi^2 + n \cdot o(\xi^2)}$. 
It follows that 
$$\lim_{n \to \infty} e^{ -i \sqrt n \xi \gamma} \lambda_{ \frac{\xi}{ \sqrt n}}^n=  e^{-\frac{a^2\xi^2}{2}}.$$ 
Therefore, the integrand in the first integral defining $T^1_n(t)$ converges to zero pointwise in $\xi$ as $n \to \infty$. By (\ref{eq:exp-times-lambda-n}), the integrand is bounded by a fixed integrable function for every $n$. Therefore, Lebesgue's dominated convergence theorem implies that $|T^1_n(t) | \to 0$ as $n \to \infty$, uniformly in $t$. In particular, we have $|T^1_n(t) | <  \varepsilon \slash 5$ for $n$ large, uniformly in $t$.

\vskip5pt

(2) -- Using the change of coordinates $\xi \mapsto \xi/\sqrt n$ again we get
$$T^2_n(t,x) =  \int_{|\xi| \leq \xi_0  \sqrt n }  e^{i \frac{ \xi t}{ \sqrt n}} \widehat{\psi} \big( \frac{\xi}{ \sqrt n} \big) \, e^{-i \sqrt n \xi \gamma}\lambda_{\frac{\xi}{ \sqrt n}}^n  \Big(  \oN_{\frac{\xi}{ \sqrt n}}\varphi (x) - \int_{\P^1} \varphi \, \diff \nu \Big) \, \diff \xi. $$

From (\ref{eq:N_s-nu}) and the fact that $\| \cdot \|_{\infty} \leq D \| \cdot \|_\W $ we get $$ \big\|  \oN_{\frac{\xi}{ \sqrt n}}\varphi - \oN_0 \varphi \big\|_\infty \leq  D \big\|  \oN_{\frac{\xi}{ \sqrt n}}\varphi - \oN_0 \varphi \big\|_\W \leq  D \, \varepsilon' \|\varphi\|_\W.$$

Recall that $\oN_0 \varphi = \int_{\P^1} \varphi \, \diff \nu$.  Together with (\ref{eq:exp-times-lambda-n}), this gives $$ \sup_{(t,x) \in \R \times \P^1} |T^2_n(t,x) | \leq  D \varepsilon'  \|\widehat \psi\|_\infty \|\varphi\|_\W \int_{-\infty}^{+\infty} e^{-\frac{a^2\xi^2}{4} } \, \diff \xi = D \varepsilon' {2 \sqrt \pi\over a}     \|\widehat \psi\|_\infty \|\varphi\|_\W = \varepsilon \slash 5.$$

\vskip5pt

(3) -- From the fact that $\|\oQ_\xi^n\|_{\W}$ is exponentially small (cf. Corollary \ref{cor:P_t-decomp-W}-(5))  and  that $\| \cdot \|_{\infty} \leq B \| \cdot \|_\W $  we get that $\|\oQ^n_\xi \varphi\|_\infty \leq c \tau^n \, \|\varphi\|_\W$ for some $0< \tau <1$ and $c>0$. Hence $$ \sup_{(t,x) \in \R \times \P^1} |T^3_n(t,x) | \leq c_2 \sqrt  n \, \xi_0  \|\widehat \psi\|_{\infty} \tau^n   \|\varphi \|_\W$$
for some $c_2 >0$. In particular, we have $|T^3_n(t,x) | <  \varepsilon \slash 5$ for $n$ large, uniformly in $(t,x)$.

\vskip5pt

(4) -- From Theorem \ref{thm:P_t-contracting}  there is a constant $c_3 >0$ such that for $\xi_0 < |\xi| < M$, we have $\|\oP^n_{\xi}\|_{W^{1,2}} \leq c_3 \beta_M^n$ for some $0<\beta_M<1$. We also have, from Proposition \ref{prop:P_t-logp}, that $\|\oP^n_{\xi}\|_{\log^{p-1}} \leq c_4 n^p$ for some constant $c_4 >0$. Corollary \ref{cor:logp-pre-equidistribution} yields 
\begin{equation}\label{lambda-4-ineq}
\|\oP^n_{\xi}\varphi\|_\infty \leq c_5 \rho_M^n  \|\varphi \|_{\W}
\end{equation}
 for some $c_5 >0$ and $0<\rho_M<1$, both independent of $\xi$. Therefore, we have
$$ \sup_{(t,x) \in \R \times \P^1} |T^4_n(t,x) | \leq c_5 \sqrt  n \, \|\widehat \psi\|_{\infty} \rho_M^n  \|\varphi \|_{\W},$$ so 
 $|T^4_n(t,x) | <  \varepsilon \slash 5$ for $n$ large, uniformly in $(t,x)$.

\vskip5pt

(5) --  Finally, we have
$$ \sup_{t \in \R} |T^5_n(t) | \leq |\widehat \psi (0)|   \int_{|\xi| \geq \xi_0  \sqrt n } e^{-\frac{a^2\xi^2}{2}}\diff \xi \, \Big | \int_{\P^1} \varphi \, \diff \nu \Big| \leq  {c_6\over \sqrt n} \norm{\widehat \psi}_{\infty}\norm{\varphi}_{W^{1,2}}  .$$ 
The last quantity tends to $0$ as $n \to \infty$, yielding $|T^5_n(t) | <  \varepsilon \slash 5$ for $n$ large, uniformly in $t$.

\vskip5pt

The estimates in (1)---(5) give (\ref{eq:LLT}) for $f$. This finishes the proof of the lemma.
\end{proof}

In order to reduce the proof of Theorem \ref{thm:LLT} to the case of Lemma \ref{lemma:LLT-product},  we'll need the following approximation lemma.  The result  is standard and a similar approximation is used in  \cite[p. 297-298]{lepage:theoremes-limites}.  We include a proof here for completeness.

\begin{lemma} \label{lemma:weierstrass}
Let $f$ be a continuous function with compact support on $\R \times \P^1$.  Then,  for  every $\varepsilon > 0$ there exist an integer $N = N(\varepsilon) >0$ and functions $f^\pm$ of the form  $f^\pm (u,x): =\sum_{1\leq j\leq N} \psi_j^\pm(u) \varphi_j^\pm(x)$ where $\psi_j^\pm \in \mathscr H$ and $\varphi_j^\pm \in \Cc^1(\P^1) \subseteq \W$, such that 
\begin{equation*}
  f^- \leq f \leq f^+ \quad \text{and} \quad \|f^\pm - f\|_{L^1(\Leb \otimes \nu)} \leq \varepsilon.
\end{equation*}
\end{lemma}

\begin{proof}

\noindent \textbf{Case 1:} Suppose first that $f(u,x) = \psi(u)$,  where $\psi$ is a smooth function with compact support on $\R$.  In this case,  the result follows directly from Lemma \ref{lemma:conv-Fourier-2}.

\medskip

\noindent \textbf{Case 2:} Suppose now that $f (u,x) =\psi(u) \varphi(x)$,  where $\psi$ is a smooth function  with compact support on $\R$ and $\varphi \in \Cc^1$.  By adding a constant to $\varphi$,  multiplying it by a constant and using Case 1, we can assume that $\frac12 \leq \varphi \leq 1$. Given $\varepsilon >0$,   Lemma \ref{lemma:conv-Fourier-2} yields functions $\psi^\pm \in \Hc$ with $\psi^- \leq \psi \leq \psi^+$ and $\| \psi^\pm - \psi \|_{L^1} \leq \varepsilon$.  Setting $f^\pm(u,x) = \psi^\pm(u) \varphi(x)$ gives the result in this case.

\medskip
\noindent \textbf{Case 3:} Consider now the general case.  After rescaling the $u$ coordinate and fixing a partition of unity subordinated to local coordinate charts in $\P^1$,  we can reduce the problem to the case where  $f = f(u,x)$  is a function on $\R \times \R^2$ supported by $[-1/2,1/2] \times [-1/2,1/2]^2$.    Denote by $x = (x_1,x_2)$ the coordinates in the $\R^2$ factor.

Let $\varepsilon > 0$.  By Stone-Weierstrass Theorem,  there exists  $N = N(\varepsilon)$ and smooth functions 	$\widetilde\psi_i$ on $[-1,1]$ and $\widetilde\varphi_i$ on $[-1,1]^2$,  $i = 1, \ldots, N- 1$ such that  $$\Big|f(u,v)-\sum_{i=1}^{N -1}\widetilde\psi_i(u)\widetilde\varphi_i(x)\Big|\leq \varepsilon /100 \quad \text{for } \, (u,x) \in [-1,1] \times [-1,1]^2.$$
We can  can use here polynomials in $u$ and $x$.

Let  $0\leq \chi \leq 1$ be a smooth cut-off function on $\R$ supported by $[-1,1]$, with $\chi =1$ on $[-1/2,1/2]$.   Set $\psi_i^\pm(u):= \psi_i(u):= \chi(u) \widetilde\psi_i(u)$ and $\varphi^\pm_i(x) := \varphi_i(x) = \chi(x_1)\chi(x_2) \widetilde\varphi_i(x)$ for $i = 1, \ldots, N-1$. For $i = N$,  set  $\varphi^\pm_N = \pm \varepsilon \chi(x_1) \chi(x_2)/100$ and $ \psi^\pm_N = \chi$.

Then,  the functions $h^\pm (u,x):=\sum_{i=1}^{N} \psi^\pm_i(u)\varphi^\pm_i(x)$ satisfy $h^- \leq f \leq h^+ $ and $\|f-h^\pm \|_\infty\leq \varepsilon/50.$ Applying Case 2 to each term of $h^\pm$  using $\varepsilon/(100N)$ instead of $\varepsilon$,   we get functions $f^\pm$ of the desired form  such that $f^- \leq f \leq f^+ $ and $\|f-f^\pm \|_\infty\leq  \varepsilon /10 $.  Since $f$ and $f^\pm$ are supported by $[-1,1]^3$,  we get that $\|f^\pm - f\|_{L^1(\Leb \otimes \nu)} \leq  (8/10) \,  \varepsilon$ . This concludes the proof of the lemma.
\end{proof}

Now we can complete the proof of Theorem \ref{thm:LLT}.

\begin{proof}[Proof of Theorem \ref{thm:LLT}]
Let $f$ be as in the statement of Theorem \ref{thm:LLT}. Fix $\varepsilon > 0$. Then,  by Lemma \ref{lemma:weierstrass},  one can find an integer $N_\varepsilon >0$ and functions $f^\pm$ of the form  $f^\pm (u,x): =\sum_{1\leq j\leq N_\varepsilon} \psi_j^\pm(u) \varphi_j^\pm(x)$ where $\psi_j^\pm \in \mathscr H$ and $\varphi_j^\pm \in \Cc^1(\P^1) \subseteq \W$, such that $f^- \leq f \leq f^+$ and $\|f^\pm - f\|_{L^1(\Leb \otimes \nu)} \leq \varepsilon$. 

Set $$\mathcal A_n(f) := \sqrt {2 \pi n} \,  a \,  \E \Big( f\big( t + \sigma(S_n, x) - n \gamma , S_n \cdot x\big) \Big)$$ and $$\oG_n(f):=  e^{-\frac{t^2}{2a^2 n}}\int_{\R \times \P^1} f(s,  y)\, \diff s \, \diff \nu(y).$$ Then, we need to show that $\big|\mathcal  A_n(f) -\oG_n(f) \big|$ tends to zero as $n \to \infty$ uniformly in  $(t,x) \in \R \times \P^1$ (cf. eq. (\ref{eq:LLT})). From Lemma \ref{lemma:LLT-product},  the last property is true for the sequences $\big|\mathcal  A_n(f^\pm) -\oG_n(f^\pm) \big|$. We also have $\big|\oG_n(f) -\oG_n(f^\pm) \big| \leq \varepsilon$ for every $n \geq 1$ by the choice of $f^\pm$. Using this and the fact that $\mathcal A_n$ and $\oG_n$ are positive functionals, one has that, for $n$ large enough  
$$\mathcal A_n(f) - \oG_n(f) \leq \mathcal A_n(f^+) - \oG_n(f) \leq \mathcal A_n(f^+) - \oG_n(f^+) + \varepsilon \leq 2 \varepsilon $$ and 
$$- 2\varepsilon \leq  \mathcal A_n(f^-) -\oG_n(f^-) - \varepsilon \leq \mathcal  A_n(f^-) - \oG_n(f) \leq \mathcal A_n(f) -\oG_n(f).$$ We conclude that $\big| \mathcal A_n(f) - \oG_n(f) \big| < 2\varepsilon$ for $n$ large enough,  uniformly in  $(t,x) \in \R \times \P^1$. Since $\varepsilon > 0$ is arbitrary, one concludes that $\big| \mathcal A_n(f) -\oG_n(f) \big| \to 0$   as $n \to \infty$ uniformly in  $(t,x) \in \R \times \P^1$. The proof is complete.
\end{proof}

\begin{remark}[The LLT for the matrix norms] \label{rmk:LLT-norms}  Theorem \ref{thm:LLT} implies that, for any $b_1 < b_2$ one has $$\lim_{n \to \infty} \sqrt n \cdot \mathbf{P} \Big(\frac{1}{\sqrt n} \big(\sigma(S_n,x) - n \gamma \big) \in \Big [ t + \frac{b_1}{\sqrt n}, t + \frac{b_2}{\sqrt n} \Big] \Big) =  e^{- \frac{t^2}{2a^2}}  {b_2 - b_1\over  \sqrt{2 \pi}\, a}.$$

We claim that there exist a point $x \in \P^1$ and a subset $G_n^x$ of $G$ such that $\big|\sigma(g,x) - \log \|g\| \big | = O(\log n)$ for every $g \in G_n^x$ and $\mu^{\ast n}(G \setminus G_n^x) =  O(n^{-1})$.  The details are given below.

This allows us to obtain a ``local limit theorem'' with windows of larger size for the random variables $\log\|S_n\|$. More precisely, under the same assumptions of Theorem  \ref{thm:LLT}, one can show that if $v_n$ is a sequence of positive numbers such that $\lim_{n \to \infty}  v_n = 0$ and $\frac{\log n}{\sqrt n} = o(v_n)$, then
 $$\lim_{n \to \infty}  v_n^{-1} \cdot \mathbf{P} \Big(\frac{1}{\sqrt n} \big( \log\|S_n\| - n \gamma \big) \in   [ t + v_n b_1, t + v_n b_2   ] \Big) = e^{- \frac{t^2}{2a^2}} {b_2 - b_1\over  \sqrt{2 \pi}\, a},$$ uniformly in $t \in \R$.
 
The above set $G_n^x$ can be obtained as follows.  For $g \in G$,  fix a Cartan decomposition $g=k_g a_g \ell_g$ as in \eqref{eq:cartan-decomposition}  and let $z_g^m = \ell_g^{-1} e_2$  be the corresponding density point,  see \cite[Chapter 14]{benoist-quint:book}.  Let $\Sigma_n \subset G \times \P^1$ be the set of pairs $(g,x) \in G \times \P^1$ such that $d(x,z_g^m) \geq n^{-1}$.  By Fubini's Theorem and the fact that a disc of radius $n^{-1}$ in $\P^1$ has area $O(n^{-2})$,   we have that the $\mu^{\ast n} \otimes \Leb$ measure of $G \times \P^1 \setminus \Sigma_n$ is $O(n^{-2})$.   For fixed $x \in \P^1$,  we let $G_n^x$ be set of $g \in G$ such that $(g,x) \in \Sigma_n$.  Then,  by Fubini's Theorem again, there exists an $x$ such that $\mu^{\ast n}(G \setminus G_n^x) =  O(n^{-1})$.  Finally,  the estimate $\big|\sigma(g,x) - \log \|g\| \big | = O(\log n)$ for $g \in G_n^x$ follows from \cite[Lemma 14.2]{benoist-quint:book}.
\end{remark}

\section{LLT for the matrix coefficients} \label{sec:LLT-coeff}

This section is devoted to the proof of Theorem \ref{thm:LLT-coeff}.  We  follow \cite{DKW:BE-LLT-coeff}, which partly uses some ideas from \cite{grama-quint-xiao}. Our method allows us to  avoid the use of the zero-one law in \cite{grama-quint-xiao}. Differently from these works, which deal with the case of finite exponential moments, we need to use here some weaker large deviation estimates and work with the space $\W$ instead of the H\"older spaces used there. The results obtained in Sections \ref{sec:markov-op} to \ref{sec:W} are indispensable.
 
We will work in this section with the function space $\W$ defined in Section \ref{sec:W} for the value $p=2$ (see Theorem \ref{thm:spectral-gap-W}), even if we assume that $\mu$ has a finite moment of order three. In particular, $W^ {1,2} \cap \Cc^{\log^{1}} \subset \W \subset W^ {1,2}\cap \Cc^0$ and by Lemma \ref{lemma:Pt-regularity-3-moments}, the family $\xi \mapsto \oP_\xi$ acting on $\W$ is locally Lipschitz in $\xi$.
\medskip

We'll need the following large deviation estimate.  We'll apply it for $p=3$.   The lemma follows from Proposition 4.1 and Lemma 4.8 in \cite{benoist-quint:CLT},  see also \cite[Lemma 3.2]{cuny-dedecker-merlevede-peligrad-2}.  These results hold in any dimension.  In their original statement,  a $O(1/ n^{p-1})$ appear in the right hand side,   but this can be  easily improved to $o(1/ n^{p-1})$ by going back to the proof of \cite[Theorem 2.2]{benoist-quint:CLT}.  It is also possible to give a proof using the spectral gap property of Theorem \ref{thm:DKW-spectral-gap}.

\begin{lemma}\label{lemma:LDT-p-moment}
	Let $\mu$ be a non-elementary measure with a finite moment of order $p>1$. Then, 
	$$\sup_{z_1,z_2\in\P^1}\mu^{*\ell}\big\{g: \, d(gz_1,z_2)\leq e^{-n}\big\}=  o(1/n^{p-1}) \quad\text{for all} \quad \ell\geq n \quad \text{as } \, n \to \infty.$$
\end{lemma}

\medskip
We now begin the proof of Theorem \ref{thm:LLT-coeff}. Fix $x = [v] \in \P^1$ and, for a function $f$ on $\R \times \P^1$, let
\begin{equation*}
\oA_n(f)(t) := \sqrt{n}  \, \mathbf E \Big(  f \big(t+ \log{ |\lp  S_n v,w \rp | \over \norm{v} \norm{w}} - n \gamma, S_n x \big)\Big)
\end{equation*}
and
\begin{equation*}
\oG_n(f)(t)  : = \frac{1}{\sqrt{2 \pi}\,a} \, e^{ -\frac{t^2}{2 a^2 n}} \int_{\R \times \P^1} f(s,z)\, \diff s \, \diff \nu(z).
\end{equation*}

Our goal is to prove that
\begin{equation} \label{eq:LLT-main-limit}
\lim_{n\to \infty}\sup_{t\in\R} \big|\, \oA_n(f)(t)  -  \oG_n(f)(t)   \big| =0.
\end{equation}

Let $x:=[v]$ and $y:=[w^\star]$, where $w^\star\neq 0$ is orthogonal to $w$. 
Using  \eqref{eq:distance-def}, it is not hard to see that $d(S_nx,y)={|\lp S_n v, w  \rp  |\over \norm{S_n v}\norm{w} }$. Thus,
\begin{equation}\label{eq:coeff-split}
\log{ |\lp S_n v,w \rp | \over \norm{v} \norm{w}} = \sigma(S_n,x) + \log d(S_n x, y).
\end{equation}

We will prove Theorem \ref{thm:LLT-coeff} by using the above formula to replace random variable $\log{ |\lp  S_n v,w \rp | \over \norm{v} \norm{w}}$ by $\sigma(S_n,x) +\log d(S_n x, y)$. In this way, we can study the behaviour of $\sigma(S_n,x)$ using the perturbed Markov operators and estimate the term  $\log d(S_n x, y)$ using the large deviation estimates from Lemma \ref{lemma:LDT-p-moment} combined with the partition of unity described below.

\medskip

Recall that $\D(y,r)$ is the open disc of radius $r$ centered at $y$ with respect to the distance $d$ defined in \eqref{eq:distance-def}.
 Let $0<\zeta \leq 1$ be a constant. For integers $k \geq 0$, introduce
\begin{align}  \label{eq:def-annulus}
\Tc_k^\zeta := \big\{z \in \P^1 :\, e^{-(k+1)\zeta} < d(z,y) < e^{-(k-1)\zeta} \big\} = \D(y,e^{-(k-1) \zeta}) \setminus \overline{\D(y,e^{-(k+1)\zeta})}.
\end{align}
Note that,  since $\P^1$ has diameter one,  these open sets cover $\P^1$.

\begin{lemma} \label{lemma:partition-of-unity-2}
	Let $0<\zeta \leq 1$. Then, there exist non-negative smooth functions $\chi_k$ on $\P^1$, $k \geq 0$, such that
	\begin{enumerate}
		\item $\chi_k$ is supported by $\Tc_k^\zeta$;
		\item If $z \in \P^1 \setminus \{y\}$,  then  $\chi_k(z) \neq 0$ for at most two values of $k$;
		\item $\sum_{k\geq 0}  \chi_k=1$ on $\P^1 \setminus \{y\}$; 
		\item $\norm{\chi_k}_{\Cc^1}\leq C \zeta^{-1} e^{k\zeta}$  for some universal constant $C>0$.;
		\item $\norm{\chi_k}_{\W}\leq C'( \zeta^{-1}+k\zeta)$ for some universal constant $C'>0$.
	\end{enumerate}
\end{lemma}

\begin{proof}
It is easy to find a smooth function $0 \leq \widetilde \chi \leq 1$ supported by $(-1,1)$ such that $\widetilde \chi(t) = 1$ for $|t|$ small, $\widetilde \chi(t) + \widetilde \chi(t-1)= 1$ for $0 \leq t\leq 1$ and $\norm{\widetilde \chi}_{\Cc^1}\leq 4$. Define  $\widetilde \chi_k (t) := \widetilde \chi(t+k)$. We see that $\widetilde \chi_k$ is supported by $(-k-1,-k+1)$,  $\sum_{k\geq 0} \widetilde \chi_k=1$ on $\R_{\leq 0}$ and $\norm{\widetilde \chi_k}_{\Cc^1}\leq 4$.  
Define $$\chi_k(z):= \widetilde \chi_k \big(\zeta^{-1} \log d(z,y) \big).$$
 Clearly, $\chi_k$ satisfies (1)--(3). Since the function $\Psi(z) := \log d(z,y)$ satisfies  $\norm{\Psi|_{\Tc_k^\zeta}}_{\Cc^1}  \lesssim e^{k \zeta}$,  it follows that $\chi_k$ satisfies (4).
 
 We now compute $\norm{\chi_k}_{W^{1,2}\cap \log ^1}$. Obviously, $\norm{\chi_k}_{L^1} \leq 1$.  Using that $\chi_k$ is supported by $\Tc_k^\zeta$ and $\norm{\chi_k}_{\Cc^1}\leq C \zeta^{-1} e^{k\zeta}$, we get $\norm{\partial \chi_k}_{L^2}^2\lesssim e^{-2(k-1)\zeta}\zeta^{-2}e^{2k\zeta}=\zeta^{-2}e^{2\zeta}$. So $\norm{\chi_k}_{W^{1,2}}\lesssim \zeta^{-1}$.
 It remains to compute $\norm{\chi_k}_{ \log ^1}$. By definition, it is the supremum of 
 $$\big|\chi_k(z_1) - \chi_k(z_2)\big| \big(1+|\log d(z_1,z_2)|\big)$$
 over $z_1 \neq z_2$ in $\P^1$. If $d(z_1,z_2)>e^{-2k\zeta}$, then $|\log d(z_1,z_2)|< 2k\zeta$. Since $|\chi_k|\leq 1$,  the last quantity is bounded by $2(1+2k\zeta)$. Observe that $d$ is comparable with the distance $d'$ induced by the Fubini-Study metric on $\P^1$.  In particular, from the mean value theorem for $d'$ along geodesics, we deduce that $\big|\chi_k(z_1) - \chi_k(z_2)\big| \leq C'' \norm{\chi_k}_{\Cc^1}d(z_1,z_2)$ for some constant $C''>0$.  Then, if  $d(z_1,z_2)\leq e^{-2k\zeta}$,  the above quantity is bounded by a constant times
 \begin{align*}
 \norm{\chi_k}_{\Cc^1}d(z_1,z_2)\cdot\big(1+|\log d(z_1,z_2)|\big)=\norm{\chi_k}_{\Cc^1}d(z_1,z_2)+\norm{\chi_k}_{\Cc^1}\big(\sqrt{d(z_1,z_2)} \,\big)^2|\log d(z_1,z_2)|\\
 \leq C \zeta^{-1}e^{k\zeta}e^{-2k\zeta}+C \zeta^{-1}e^{k\zeta}e^{-k\zeta}\sqrt{d(z_1,z_2)}\,|\log d(z_1,z_2)|\lesssim \zeta^{-1}.
 \end{align*}
 Thus, we conclude that $\norm{\chi_k}_{W^{1,2}\cap \log ^1}\lesssim \zeta^{-1}+k\zeta$.  The estimate (5) follows. The proof of the lemma is complete.
\end{proof}

\medskip

We start with a particular case of Theorem \ref{thm:LLT-coeff}.

\begin{proposition} \label{prop:LLT-product-test-fcn}
Theorem \ref{thm:LLT-coeff} holds when $f(u,z) = \psi(u) \varphi(z)$ with $\psi$ is of class $\Cc^1$ with compact support in $\R$ and $\varphi \in \Cc^1(\P^1) \subset \W$.
\end{proposition}

We now prove the above proposition.   From now on, we fix $\psi$ and $\varphi$ as in the statement of Proposition \ref{prop:LLT-product-test-fcn}. We can assume that  $0 \leq \varphi\leq 2$,  $\oN_0 \varphi=1$ and $\norm{\varphi}_{\Cc^1}\leq 2$, since the problem is linear on $\varphi$ and these functions span the space $\Cc^1$. We can also assume that  $\psi$ is non-negative and $\|\psi\|_{\Lip} \leq 1$. In order to simplify the notation, let $$\oA_n(t):= \oA_n(\psi \cdot \varphi)(t) \quad \text{and }\quad \oG_n(t):= \oG_n(\psi \cdot \varphi)(t) =   \frac{1}{\sqrt{2 \pi}\,a} \, e^{ -\frac{t^2}{2 a^2 n}} \int_{\R} \psi(s) \, \diff s,$$
  where we have used the assumption that  $\oN_0 \varphi=1$.  The limit \eqref{eq:LLT-main-limit} will be obtained by dealing separately with the upper and lower estimate.

\subsection{Upper bound} \label{subsec:LLT-upper-limit}

 We first consider the upper bound in the limit \eqref{eq:LLT-main-limit}. The lower bound is analogous and will be treated in the end of the proof.

\begin{proposition} \label{prop:LLT-upper-limit}
	Let $\oA_n(t)$ be as above. Then, $$\limsup_{n\to \infty}\sup_{t\in\R}\big( \oA_n(t)  -   \oG_n(t) \big) \leq 0.$$
\end{proposition}

 Let $0 < \zeta \leq 1$ be a small constant and set $$\psi^\ast(u) := \sup_{|u' - u| \leq \zeta} \psi(u').$$
According to \eqref{eq:def-phi-delta-sup},  $\psi^\ast$ coincides with the function $\psi^+_{[\zeta]}$,  but we use here a different notation for simplicity.  We will also omit the dependence of $\psi^\ast$ on $\zeta$ in order to ease the notation.  From the definition and our assumptions on $\psi$,  it follows that $\psi^\ast$ is non-negative,  Lipschitz continuous and $\|\psi^\ast\|_{\Lip} \leq 1$.  Moreover,  it is not hard to check that $\|\psi^\ast - \psi\|_{L^1}$ tends to zero as $\zeta$ tends to zero.

 For later use, consider the translates of $\psi^\ast$ given by $$\psi^\ast_{t,k}(u):=\psi^\ast(u +t - k \zeta),$$
   where again we omit the dependence on $\zeta$.  Observe that, since $\psi$ is compactly supported, for fixed $t,s\in\R$, we have that $\psi^\ast_{t,k}(u)\neq 0$ for only finitely many $k$'s.

\medskip

  Fix another constant $0<\eta\leq 1$.  Applying Lemma \ref{lemma:LDT-p-moment} to $\lfloor \eta n^{1/4} \rfloor$ instead of $n$, $\ell = n$ and $p=3$ yields  
$$\mu^{*n}\big\{g: \, \log d(gx,y)\leq -\eta n^{1/4} \big\}= \eta^{-2} o(1/ \sqrt n) \quad\text{as}\quad n\to +\infty.$$

	In order to simplify the notation,  we introduce the linear functional
\begin{equation} \label{eq:En-def}
\oE_n\big(f\big):= \sqrt n \, \mathbf E\Big( f\big( \sigma(S_n,x)-n\gamma ,S_n x \big)  \Big),
\end{equation}
where $f$ is a function of $(u,z) \in \R \times \P^1$.

For $z \in \P^1$, set 
\begin{equation} \label{eq:Phi-star-def}
\Phi_n^\star (z):= \varphi(z) -  \sum_{0\leq k\leq  \zeta^{-1} \eta n^{1/4}}  (\chi_k \varphi) (z),
\end{equation}	
where $\chi_k$ are the functions from Lemma \ref{lemma:partition-of-unity-2} and define
$$\oB^\ast_n(t):=  \sum_{0\leq k\leq  \zeta^{-1} \eta n^{1/4}} \oE_n\big(\psi^\ast_{t,k}\cdot\chi_k \varphi  \big)+   \oE_n\big(\psi^\ast_{t,0}\cdot \Phi_n^\star \big).$$

\begin{lemma} \label{lemma:llt-ineq-1}
Let $\oA_n$ and $\oB^\ast_n$ be as above. Then, $$\oA_n(t) \leq \oB^\ast_n(t) +\eta^{-2} o(1)$$ as $n\to +\infty$, uniformly in $t \in \R$.
\end{lemma}

\begin{proof}
	Using   that $\mathbf P \big( \log d(S_n x,y)\leq  -\eta n^{-1/4} \big) =\eta^{-2} o(1 / \sqrt n)$ and the decomposition \eqref{eq:coeff-split}, we obtain
	\begin{align*}
	\oA_n(t) \leq \sqrt{n}  \, \mathbf E \Big(  \psi(t+ \sigma(S_n,x) +\log d(S_n x, y) - n \gamma) \mathbf 1_{\log d(S_n x, y)\geq -\eta  n^{1/4}} \varphi(S_n x)  \Big)    + \eta^{-2}o(1).
	\end{align*}
	Observe that,  when $S_n x \in\supp(\chi_k)$, we have  $-(k+1) \zeta \leq \log d(S_n x,y) \leq -(k-1)\zeta$,  so
	$$  \psi \big( t + \sigma(S_n,x) +\log d(S_n x, y) - n \gamma \big)   \leq \psi^\ast \big( t+ \sigma(S_n,x) - k\zeta - n \gamma \big) = \psi^\ast_{t,k}\big(\sigma(S_n,x)  - n \gamma\big),$$
	where we have used the definitions of $\psi^\ast$ and $\psi^\ast_{t,k}$.
	By construction, $\mathbf 1_{\log d(z, y)\geq -\eta  n^{1/4}} \leq \sum_{0\leq k \leq \zeta^{-1} \eta n^{1/4} + 1} \chi_k(z)$. So, it follows from the above, that
	\begin{align*}
	&\E \Big(\psi \big(t+ \sigma(S_n,x) +\log d(S_n x, y) - n \gamma \big) \mathbf 1_{\log d(S_n x, y)\geq -\eta  n^{1/4}} \varphi(S_n x)\Big) \\
	&\leq  \sum_{0\leq k\leq \zeta^{-1} \eta n^{1/4} + 1} \E \Big( \psi^\ast_{t,k}\big(\sigma(S_n,x)-n\gamma\big)(\chi_k \varphi)(S_n x) \Big)    \\
	 &\leq  \sum_{0\leq k\leq \zeta^{-1} \eta n^{1/4}} \E \Big(\psi^\ast_{t,k}\big(\sigma(S_n,x)-n\gamma\big) (\chi_k \varphi)(S_n x) \Big) +   2\E\Big(\chi_{k_0}(S_n x) \Big),
	\end{align*}
	where $k_0:= \lfloor  \zeta^{-1} \eta n^{1/4} \rfloor + 1$.  Recall that  $0 \leq \varphi\leq 2$,  $\psi^\ast_{t,k}$ is non-negative and $\norm{\psi^\ast_{t,k}}_\infty \leq 1$.
	
	From the fact that $\chi_{k_0} \leq \mathbf 1_{\D(y,e^{-(k_0-1) \zeta })}$, we see that the last term above is bounded by $2 \mathbf P \big( d(S_n x,y)\leq  e ^{1-\eta n^{1/4}} \big) = 2 \eta^{-2} o(1  /\sqrt  n)$ from Lemma \ref{lemma:LDT-p-moment} with $p=3$. Hence,
	\begin{equation*}
	\oA_n(t) \leq	 \sum_{0\leq k\leq \zeta^{-1} \eta n^{1/4}} \oE_n\big(\psi^\ast_{t,k}\cdot\chi_k \varphi \big)+\eta^{-2} o(1)\leq \oB^\ast_n(t)  +\eta^{-2}o(1),
	\end{equation*}
	proving the lemma.
\end{proof}

	By Lemma \ref{lemma:conv-Fourier-2} , for every $0<\delta\leq 1$, there exists a smooth function $(\psi^\ast)^+_{\delta}$  such that $\widehat {(\psi^\ast)^+_{\delta}}$ has support in $[-\delta^{-2},\delta^{-2}]$,  $$\psi^\ast \leq (\psi^\ast)^+_{\delta},\quad \lim_{\delta\to 0} (\psi^\ast)^+_{\delta} =\psi^\ast    \quad \text{and} \quad  \lim_{\delta\to 0} \big \|(\psi^\ast)^+_{\delta} -\psi^\ast \big \|_{L^1} = 0.$$ 
Moreover,  $\norm{(\psi^\ast)^+_{\delta}}_\infty$, $\norm{(\psi^\ast)^+_{\delta}}_{L^1}$ and $\|\widehat{(\psi^\ast)^+_{\delta}}\|_{\Cc^1}$ are bounded by a constant independent of $\delta$.  We fix from now until the end of the proof a  $0<\delta\leq 1$.

As above, for $t\in\R$ and $k\in\N$, we consider the translations
$$ \psi_{t,k}^+(u):=(\psi^\ast)^+_{\delta}(u+t - k \zeta)  . $$
We omit the dependence on $\delta$ and $\zeta$ in order to ease the notation.  Define also
\begin{equation} \label{eq:R-def}
	\oB_n^+(t):= \sum_{0\leq k\leq  \zeta^{-1} \eta n^{1/4}} \oE_n\big(\psi_{t,k}^+\cdot \chi_k \varphi  \big)+  \oE_n\big(\psi_{t,0}^+\cdot \Phi_n^\star  \big).	
\end{equation}

	Clearly, we have $\oB^\ast_n(t)\leq \oB_n^+(t)$. From the definition of $\oE_n$, Fourier inversion formula and Fubini's theorem, we have
	\begin{align*}
	\oE_n\big(\psi_{t,k}^+\cdot \chi_k \varphi \big)&=\sqrt n \, \int_{G} (\psi^\ast)^+_{\delta}\big(\sigma(g,x)-n\gamma+t - k \zeta \big) \cdot (\chi_k \varphi)(gx) \,\diff \mu^{*n}(g)\\
	&={\sqrt n\over 2\pi}\int_{G} \int_{-\infty}^{+\infty} \widehat{(\psi^\ast)^+_{\delta}}(\xi) e^{i\xi(\sigma(g,x)-n\gamma+t -k \zeta )} \cdot (\chi_k \varphi)(gx) \,\diff \xi\diff\mu^{*n}(g)\\
	&={\sqrt n\over 2\pi}\int_{-\infty}^{+\infty}  \widehat{(\psi^\ast)^+_{\delta}}(\xi) e^{i\xi(t-k \zeta)}\cdot e^{-i\xi n\gamma}\oP^n_{\xi}(\chi_k \varphi)(x) \,\diff \xi,
	\end{align*}
	where in the last step we have used \eqref{eq:P_t^n}.
	
	Recall that $\supp\big( \widehat{(\psi^\ast)^+_{\delta}} \big)\subset [-\delta^{-2},\delta^{-2}]$. So, after the change of variables $\xi \mapsto \xi / \sqrt n$, the above identity becomes 
	$$	\oE_n\big(\psi_{t,k}^+\cdot\chi_k \varphi \big) ={1\over 2\pi}\int_{-\delta^{-2} \sqrt n}^{\delta^{-2} \sqrt n} \widehat{(\psi^\ast)^+_{\delta}}\Big({\xi\over \sqrt n}\Big) e^{i\xi{t -k \zeta \over \sqrt n}}\cdot e^{-i\xi \sqrt n\gamma}\oP^n_{{\xi\over \sqrt n}}(\chi_k \varphi)(x) \,\diff \xi.$$
	
	A similar computation yields $$\oE_n\big(\psi_{t,0}^+ \cdot \Phi_n^\star \big) = {1\over 2\pi}\int_{-\delta^{-2} \sqrt n}^{\delta^{-2} \sqrt n} \widehat{(\psi^\ast)^+_{\delta}}\Big({\xi\over \sqrt n}\Big) e^{i\xi{t \over \sqrt n}}\cdot e^{-i\xi \sqrt n\gamma}\oP^n_{{\xi\over \sqrt n}}\Phi_{n}^{\star} (x) \,\diff \xi.$$

Define
\begin{equation} \label{eq:Phi-xi-def}
\Phi_{n,\xi} (z):= \sum_{0\leq k\leq \zeta^{-1}  \eta  n^{1/4}}  e^{ - i \xi{k \zeta \over \sqrt n}}(\chi_k \varphi)(z).
\end{equation}
 Using this notation and the above computations, \eqref{eq:R-def} becomes
\begin{equation} \label{eq:R-formula} 
\oB_n^+(t)= {1\over 2\pi}\int_{-\delta^{-2} \sqrt n}^{\delta^{-2} \sqrt n} \widehat{(\psi^\ast)^+_{\delta}}\Big({\xi\over \sqrt n}\Big) e^{i\xi{t\over \sqrt n}}\cdot e^{-i \xi \sqrt n\gamma}\oP_{{\xi\over \sqrt n}}^n(\Phi_{n,\xi} +\Phi_{n}^{\star}  )(x) \,\diff \xi.
\end{equation}

	\begin{lemma}  \label{lemma:norm-Phi-2}
	Let $0<\eta\leq 1$, $\Phi_{n,\xi}$ and $\Phi_{n}^{\star}$ be as above. Then, 
	\begin{equation} \label{eq:psi_xi+psi_T-2}
	\Phi_{n,\xi} + \Phi_{n}^{\star} = \varphi + \sum_{0\leq k\leq \zeta^{-1}  \eta  n^{1/4}} \big(  e^{- i \xi{k \zeta \over \sqrt n}} - 1 \big) \chi_k\varphi
	\end{equation}
	and there is a constant $C>0$ independent of $n,\eta,\zeta$ and $\xi$ such that 
	\begin{equation*}  \label{eq:norm-Phi-2}
	\norm{\Phi_{n,\xi} }_{\W}\leq  C  \zeta^{-2} \eta\sqrt n \quad\text{and}\quad  \norm{\Phi_{n}^{\star}}_{\W}\leq  C(1 +  \zeta^{-2} \eta\sqrt n).
	\end{equation*}
	Moreover, $\Phi_{n}^{\star}$ is supported by $\big\{z:\,\log d(z,y)\leq - \eta n^{1/4} + 1\big\}$.
\end{lemma}

\begin{proof}
	The first and last assertions are clear from the definitions of $\Phi_{n,\xi}$, $ \Phi_{n}^{\star}$ and parts (1) and (3) of Lemma \ref{lemma:partition-of-unity-2}. Since $\varphi$ is $\Cc^1$ and  $\norm{\varphi}_{\Cc^1}\leq 2$, it follows that $\norm{\chi_k \varphi}_{\W} \lesssim \norm{\chi_k}_{\W}$ and, from  Lemma \ref{lemma:partition-of-unity-2}-(5), the last quantity is $\lesssim \zeta^{-1}+k\zeta$. Therefore,
	$$\norm{\Phi_{n,\xi} }_{\W}\leq \sum_{0\leq k\leq  \zeta^{-1} \eta n^{1/4}}\norm{\chi_k \varphi}_{\W}\lesssim \sum_{0\leq k\leq \zeta^{-1} \eta  n^{1/4}}(\zeta^{-1}+k\zeta)\lesssim \zeta^{-2} \eta n^{1/4} + \zeta^{-1} \eta^2 \sqrt{n} \lesssim \zeta^{-2} \eta \sqrt n,$$
	because $0 < \eta,  \zeta  \leq 1$, yielding the first bound. As $\Phi_{n}^{\star} = \varphi - \Phi_{n,0}$ by  \eqref{eq:psi_xi+psi_T-2}, the bound on $\norm{\Phi_{n}^{\star}}_{\W}$ also follows. This completes the proof.
\end{proof}

	Define  
\begin{equation} \label{eq:S-def} 
\oG_n^+(t):={1\over 2\pi}\widehat{(\psi^\ast)^+_{\delta}}(0)  \int_{-\infty}^{+\infty} e^{i\xi{t\over \sqrt n}} e^{-{ a^2\xi^2 \over 2}}  \,\diff \xi ={1\over \sqrt{2\pi} \, a} e^{-{t^2\over 2a^2 n}}\int_{\R} (\psi^\ast)^+_{\delta} (u)\,\diff u,
\end{equation}
where in the second equality we have used the fact that the inverse Fourier transform of $e^{-{ a^2\xi^2 \over 2}}$ is $ {1\over \sqrt{2\pi} \,  a}  e^{-{t^2\over 2a^2}}$.

\begin{lemma}\label{lemma-R-S}
	Fix  $0<\zeta\leq 1$,   $0<\eta\leq 1$ and $0< \delta \leq 1$. Then, there exists a constant $C_{\delta}>0$ independent of $n$, $\eta$ and $\zeta$ such that, for all $n \geq 1$,  
	$$\sup_{t\in\R}\big|\oB_n^+(t)-\oG_n^+(t) \big| \leq C_\delta \big( \zeta^{-2}\eta+n^{-1/2}\big).$$
\end{lemma}

\begin{proof} 

Let $\xi_0>0$ be a constant satisfying Lemma \ref{lemma:lambda-estimates}. In particular, the  decomposition of $\oP_\xi$ on $\W$ in Corollary \ref{cor:P_t-decomp-W} holds for $|\xi| \leq \xi_0$. Using that decomposition, \eqref{eq:R-formula} and \eqref{eq:S-def}, we can write $$\oB_n^+(t)-\oG_n^+(t) = \Lambda_n^1(t) + \Lambda_n^2(t)  + \Lambda_n^3(t)  + \Lambda_n^4(t)  + \Lambda_n^5(t),$$ where
	$$\Lambda_n^1(t):={1\over 2\pi}\int_{-\xi_0 \sqrt n}^{\xi_0 \sqrt n} e^{i\xi{t\over \sqrt n}} \Big[ \widehat{(\psi^\ast)^+_{\delta}}\Big({\xi\over \sqrt n}\Big) e^{-i \xi \sqrt n\gamma}\lambda_{{\xi\over \sqrt n}}^n\oN_0(\Phi_{n,\xi} +\Phi_{n}^{\star}  )-\widehat{(\psi^\ast)^+_{\delta}}(0)    e^{-{  a^2\xi^2 \over 2}}\Big]  \,\diff \xi ,$$
	$$\Lambda_n^2(t):= {1\over 2\pi}\int_{-\xi_0 \sqrt n}^{\xi_0 \sqrt n} e^{i\xi{t\over \sqrt n}}\Big[ \widehat{(\psi^\ast)^+_{\delta}}\Big({\xi\over \sqrt n}\Big) e^{-i \xi \sqrt n\gamma}\lambda_{{\xi\over \sqrt n}}^n\big(\oN_{{\xi\over \sqrt n}}-\oN_0\big)(\Phi_{n,\xi} +\Phi_{n}^{\star}  ) (x) \Big]  \,\diff \xi , $$
	$$\Lambda_n^3(t):=  {1\over 2\pi}\int_{-\xi_0 \sqrt n}^{\xi_0 \sqrt n} e^{i\xi{t\over \sqrt n}}  \widehat{(\psi^\ast)^+_{\delta}}\Big({\xi\over \sqrt n}\Big)  e^{-i \xi \sqrt n\gamma} \oQ_{{\xi\over \sqrt n}}^n(\Phi_{n,\xi} +\Phi_{n}^{\star}  )(x)   \,\diff \xi,   $$
	$$ \Lambda_n^4(t):=  {1\over 2\pi}\int_{\xi_0\sqrt n \leq|\xi|\leq\delta^{-2} \sqrt n}e^{i\xi{t\over \sqrt n}} \widehat{(\psi^\ast)^+_{\delta}}\Big({\xi\over \sqrt n}\Big) e^{-i \xi \sqrt n\gamma}\oP_{{\xi\over \sqrt n}}^n(\Phi_{n,\xi} +\Phi_{n}^{\star}  )(x) \,\diff \xi   $$
	and 
	$$ \Lambda_n^5(t):=  - {1\over 2\pi}\widehat{(\psi^\ast)^+_{\delta}}(0)  \int_{|\xi|\geq \xi_0 \sqrt n} e^{i\xi{t\over \sqrt n}}  e^{-{  a^2\xi^2 \over 2}} \,\diff \xi. $$
	\medskip
	
	We will bound each $\Lambda_n^j$, $j=1,\ldots, 5$,  separately. We will use that 
	$$\norm{\Phi_{n,\xi} +\Phi_{n}^{\star}  }_{\W}\lesssim 1 + \zeta^{-2} \eta \sqrt n$$ uniformly in $\xi$,  according to Lemma \ref{lemma:norm-Phi-2}. 
	
	In order to bound $\Lambda_n^2$, we have, by Lemma \ref{lemma:Pt-regularity-3-moments}, $\norm{\oN_\xi-\oN_0}_{\W}=O(|\xi|)$ for $|\xi|\leq \xi_0$. Thus, for $|\xi|\leq \xi_0 \sqrt n$,
	 $$ \Big\| \big(\oN_{{\xi\over \sqrt n}}-\oN_0\big)(\Phi_{n,\xi} +\Phi_{n}^{\star}  ) \Big\|_\infty \lesssim  {|\xi|\over \sqrt n} \norm{\Phi_{n,\xi} +\Phi_{n}^{\star}  }_{\W}\lesssim {|\xi|\over \sqrt n}+ \zeta^{-2} \eta|\xi|.$$
	Recall, from Lemma \ref{lemma:lambda-estimates}, that $\big|\lambda_{{\xi\over \sqrt n}}^n\big|\leq e^{-{ a^2\xi^2\over 3}}$ for $|\xi|\leq \xi_0 \sqrt n$. 
		Since $\|  \widehat{(\psi^\ast)^+_{\delta}} \|_{\Cc^1}$ is bounded uniformly in $\delta$, we get  
	$$\sup_{t\in \R} \big|\Lambda_n^2(t)\big|\lesssim  \int_{-\infty}^{+\infty} e^{-{ a^2\xi^2\over 3}}  {|\xi|\over \sqrt n} \,\diff \xi+  \int_{-\infty}^{\infty} e^{-{ a^2\xi^2\over 3}}  \zeta^{-2} \eta |\xi| \,\diff \xi \lesssim {1\over \sqrt n}+ \zeta^{-2} \eta. $$
	
	For  $\Lambda_n^3$, we use that $\norm{\oQ^n_\xi}_{\W} \leq c \beta^n$ for $|\xi| \leq \xi_0$, where $c>0$ and $0<\beta<1$ are constants, see Corollary \ref{cor:P_t-decomp-W}-(5). Therefore, for $|\xi|\leq \xi_0\sqrt n$,  
	$$\Big\| \oQ_{{\xi\over \sqrt n}}^n(\Phi_{n,\xi} +\Phi_{n}^{\star}  ) \Big\|_\infty \lesssim \beta^n\norm{\Phi_{n,\xi} +\Phi_{n}^{\star}  }_{\W} \lesssim \beta^n (1+ \zeta^{-2} \eta \sqrt n),    $$
	which gives
	$$\sup_{t\in\R} \big|\Lambda_n^3(t)\big|\lesssim \int_{-\xi_0 \sqrt n}^{\xi_0 \sqrt n}  \beta^n (1+ \zeta^{-2} \eta \sqrt n) \,\diff \xi = 2\xi_0\sqrt n \beta^n (1+ \zeta^{-2} \eta \sqrt n)\lesssim {1\over \sqrt n}+ \zeta^{-2} \eta.$$

	In order to bound  $\Lambda_n^4$,  recall from \eqref{lambda-4-ineq}, there are constants $C'_\delta>0$ and $0<\rho_\delta<1$ such that $\norm{\oP^n_{\xi}(\Phi_{n,\xi} +\Phi_{n}^{\star}  ) }_{\infty}\leq C'_\delta \rho_\delta^n \norm{\Phi_{n,\xi} +\Phi_{n}^{\star} }_{\W}$ for all $\xi_0\leq |\xi|\leq \delta^{-2}$ and $n \geq 1$. Therefore,  
	$$\sup_{t\in\R}  \big|\Lambda_n^4(t)\big|\lesssim \int_{\xi_0\sqrt n \leq|\xi|\leq\delta^{-2} \sqrt n} C'_\delta \rho_\delta^n (1+ \zeta^{-2}\eta \sqrt n) \,\diff \xi \leq C''_\delta \Big( {1\over \sqrt n}+  \zeta^{-2} \eta \Big),$$
	for some constant $C''_{\delta}>0$.
	
	The modulus of the term $\Lambda_n^5$ is clearly $\lesssim 1/\sqrt n$, so it only remains to estimate $\Lambda_n^1$. For every $t \in \R$, we have $$\big| \Lambda_n^1(t) \big|\leq \Gamma_n^1+\Gamma_n^2+\Gamma_n^3,$$ where 
	$$\Gamma_n^1:= {1\over 2\pi}\int_{-\xi_0 \sqrt n}^{\xi_0 \sqrt n} \Big| \widehat{(\psi^\ast)^+_{\delta}}\Big({\xi\over \sqrt n}\Big) \Big| \, \big|\lambda_{{\xi\over \sqrt n}}^n \big| \cdot\Big| \oN_0(\Phi_{n,\xi} +\Phi_{n}^{\star}  )- 1 \Big|  \,\diff \xi , $$
	$$\Gamma_n^2:= {1\over 2\pi}\int_{-\xi_0 \sqrt n}^{\xi_0 \sqrt n} \big|\lambda_{{\xi\over \sqrt n}}^n\big|\cdot \Big| \widehat{(\psi^\ast)^+_{\delta}}\Big({\xi\over \sqrt n}\Big) -\widehat{(\psi^\ast)^+_{\delta}}(0) \Big|  \,\diff \xi $$
	and
	$$\Gamma_n^3:= {1\over 2\pi}\int_{-\xi_0 \sqrt n}^{\xi_0 \sqrt n} \big| \widehat{(\psi^\ast)^+_{\delta}}(0) \big| \cdot\Big| e^{-i \xi \sqrt n\gamma}\lambda_{{\xi\over \sqrt n}}^n-    e^{-{ a^2\xi^2 \over 2}}\Big|  \,\diff \xi. $$
	
	Recall that $\chi_k$ is bounded by $1$, supported by $\Tc_k^\zeta  \subset \D(y,e^{-(k-1)\zeta})$ and $0 \leq \varphi \leq 2$. Therefore,
	$$\oN_0 (\chi_k \varphi) = \int_{\P^1} \chi_k \varphi \, \diff \nu \leq 2\nu \big(\D(y,-(k-1)\zeta)\big)  = o\Big({1\over \zeta^2 k^2}\Big),$$
	where in the last step we have used Proposition \ref{regularity} for $p=3$.
	
	Using \eqref{eq:psi_xi+psi_T-2} and the assumption that $\oN_0 \varphi = 1$, we obtain
	\begin{align*}
	\Big| \oN_0(\Phi_{n,\xi} +\Phi_{n}^{\star}  )- 1 \Big| = \Big| \oN_0(\Phi_{n,\xi} +\Phi_{n}^{\star}  )-\oN_0 \varphi \Big|\leq  \sum_{0\leq k\leq \zeta^{-1}  \eta  n^{1/4}}  \big| e^{ - i \xi{k \zeta\over \sqrt n}}-1 \big|\oN_0 (\chi_k \varphi) \\
	\lesssim \sum_{1 \leq k\leq \zeta^{-1}  \eta  n^{1/4}} |\xi|{ k \zeta \over \sqrt n}\cdot o\Big({1\over \zeta^2 k^2 }\Big) \lesssim \sum_{0\leq k\leq \zeta^{-1}  \eta  n^{1/4}}  { \zeta^{-1} |\xi| \over  \sqrt n}  \leq   \zeta^{-2} \eta n^{-1/4}|\xi|.
	\end{align*}

	Using that $\big|\lambda_{{i\xi\over \sqrt n}}^n\big|\leq e^{-{ a^2\xi^2\over 3}}$ for $|\xi|\leq \xi_0 \sqrt n$ (Lemma \ref{lemma:lambda-estimates}) and that  $\|\widehat{(\psi^\ast)^+_{\delta}}\|_{\Cc^1}$  is uniformly bounded,  we get that 
	\begin{align*}
 \Gamma_n^1 \lesssim      \int_{-\xi_0 \sqrt n}^{\xi_0 \sqrt n}  \|\widehat{(\psi^\ast)^+_{\delta}}\|_{\infty} e^{-{ a^2\xi^2\over 3}}  \zeta^{-2} \eta n^{-1/4}|\xi| \,\diff \xi  \lesssim   \zeta^{-2} \eta n^{-1/4} \int_{-\infty}^\infty e^{-{ a^2\xi^2\over 3}}|\xi| \,\diff \xi \lesssim \zeta^{-2} \eta
 \end{align*} 
	and
	$$\Gamma_n^2\lesssim  \int_{-\xi_0 \sqrt n}^{\xi_0 \sqrt n} e^{-{ a^2\xi^2\over 3}} {|\xi|\over \sqrt n} \|\widehat{(\psi^\ast)^+_{\delta}}\|_{\Cc^1} \,\diff \xi\lesssim {1\over \sqrt n}.$$

	The bound $\Gamma_n^3\lesssim 1/\sqrt n$ follows by splitting the integral along the intervals $|\xi|\leq \sqrt[6] n$ and $\sqrt[6] n< |\xi| \leq \xi_0\sqrt n$ and using Lemma \ref{lemma:lambda-estimates}.
	
	We conclude that $$\sup_{t\in\R} \big|\Lambda_n^1(t)\big|\lesssim  \zeta^{-2} \eta + n^{-1/2}.$$
	
	\medskip
	
	Gathering the above estimates, we obtain $$\sup_{t\in\R}\big| \oB_n^+(t)-\oG_n^+(t)  \big|\lesssim (C''_\delta+1) (\zeta^{-2} \eta+n^{-1/2}).$$
	This finishes the proof of the lemma.
\end{proof}

	The above estimates are enough to obtain the desired upper bound.

\begin{proof}[Proof of Proposition \ref{prop:LLT-upper-limit}]

	Fix  $0<\zeta\leq 1$,   $0<\eta\leq 1$ and $0<\delta \leq 1$  as  before in this subsection.  Lemmas \ref{lemma:llt-ineq-1} and \ref{lemma-R-S} and the fact that $\oB^\ast_n(t) \leq \oB_n^+(t)$ give that
	   $$\oA_n(t) \leq \oG_n^+(t)  + \eta^{-2} o(1) +C_\delta \big(\zeta^{-2} \eta+n^{-1/2}\big) \quad \text{as} \quad n\to+\infty,$$ uniformly in $t\in\R$.
		
	Recall,   from \eqref{eq:S-def},  that $\oG_n^+(t) = {1\over \sqrt{2\pi} \, a} e^{-{t^2\over2 a^2 n}}\int_{\R} (\psi^\ast)^+_{\delta} (u)\,\diff u$. Hence, for every fixed $n$,  
	$$ \big|  \oG_n^+(t) - \oG_n(t)  \big| \leq {1\over \sqrt{2\pi} \, a}  \big\|   (\psi^\ast)^+_{\delta} -\psi\big\|_{L^1},$$ so
$$ \oA_n(t) -   \oG_n(t) \leq  \frac{1}{\sqrt{2\pi} a}   \big\|   (\psi^\ast)^+_{\delta} -\psi  \big\|_{L^1} +   \eta^{-2} o(1) +C_\delta \big(\zeta^{-2} \eta+n^{-1/2}\big)$$ as $n\to+\infty$.  Therefore,	
\begin{align*}  
\limsup_{n\to \infty}  \sup_{t\in\R} \big( \oA_n(t) -  \oG_n(t) \big) \leq   \frac{1}{\sqrt{2\pi} a} \big\|   (\psi^\ast)^+_{\delta} -\psi  \big\|_{L^1} +  C_\delta \,  \zeta^{-2}\eta  \\ \leq 
\frac{1}{\sqrt{2\pi} a} \big( \big\| (\psi^\ast)^+_{\delta}  - \psi^\ast   \big\|_{L^1}  +  \big\|  \psi^\ast - \psi   \big\|_{L^1} \big)  +  C_\delta \,  \zeta^{-2}\eta.
\end{align*}	

From the definition of $\psi^\ast$ and the fact that $\|\psi^\ast\|_{\Lip} \leq 1$,  it follows that $ \big\|  \psi^\ast - \psi   \big\|_{L^1} \leq A \zeta$   for some constant $A>0$.  Using again that $\psi^\ast$ is Lipschitz with $\|\psi^\ast\|_{\Lip} \leq 1$,  it follows from  the construction of  $ (\psi^\ast)^+_{\delta}$  via convolution (see the  proof of Lemma \ref{lemma:conv-Fourier-2}) that   $ \big\| (\psi^\ast)^+_{\delta}  - \psi^\ast   \big\|_{L^1}\leq B \delta$ for some constant $B>0$.  Therefore,	
\begin{equation*}  
\limsup_{n\to \infty}  \sup_{t\in\R} \big( \oA_n(t) -  \oG_n(t) \big) \leq   \frac{1}{\sqrt{2\pi} a} \big(  B \delta + A \zeta  \big)   +  C_\delta \,  \zeta^{-2}\eta.
\end{equation*}	

	Recall that $0<\zeta\leq 1$, $0<\eta\leq 1$ and $0<\delta\leq1$ are arbitrary and independent.   Thus, setting  $\eta = \zeta^3$ and letting $\zeta \to 0$ first, then $\delta \to 0$ yields the desired result.
 
\end{proof}

\subsection{Lower bound} \label{subsec:LLT-lower-limit}

We now deal with the lower bound for the limit in \eqref{eq:LLT-main-limit}. 

\begin{proposition} \label{prop:LLT-lower-limit}
	Let $\oA_n(t)$ and  $\oG_n(t)$ be as above. Then, $$\liminf_{n\to \infty}\inf_{t\in\R} \big( \oA_n(t)  -   \oG_n(t) \big) \geq 0.$$
\end{proposition}

The argument is a variation of the one used in the proof of Proposition \ref{prop:LLT-upper-limit}, but using approximations from below instead.  For $0 < \zeta \leq 1$  a small constant we set $$\psi_\ast(u) := \inf_{|u' - u| \leq \zeta} \psi(s')$$
and
$$\psi_{\ast,t,k}(u):=\psi_\ast(u +t - k \zeta).$$
As before, $\psi_\ast$ is non-negative,  Lipschitz continuous and $\|\psi^\ast\|_{\Lip} \leq 1$.  Moreover,  $\|\psi_\ast - \psi\|_{L^1}$ tends to zero as $\zeta$ tends to zero.  

Fix  $0<\eta\leq 1$, let $\chi_k$ be as in Lemma \ref{lemma:partition-of-unity-2} and $\Phi_n^\star$ be the function defined in \eqref{eq:Phi-star-def}. Set
$$\oB_{\ast,n}(t):=  \sum_{0\leq k\leq  \zeta^{-1} \eta n^{1/4}} \oE_n\big( \psi_{\ast,t,k} \cdot\chi_k \varphi  \big)+   \oE_n\big( \psi_{\ast,t,k} \cdot \Phi_n^\star \big).$$	
where $\oE_n$ is defined in \eqref{eq:En-def}.

\begin{lemma} \label{lemma:llt-ineq-2}
 Let $\oA_n$ and $\oB_{\ast,n}$ be as above. Then, $$\oA_n(t) \geq \oB_{\ast,n}(t) - \eta^{-2}o(1)$$ as $n\to +\infty$, uniformly in $t \in \R$.
\end{lemma}

\begin{proof}
	Using  \eqref{eq:coeff-split} and the definition of $\oA_n(t)$,  it follows that
	$$	\oA_n(t) \geq \sqrt{n}   \mathbf E \Big(  \psi \big(t+ \sigma(S_n,x) +\log d(S_n x, y) - n \gamma \big) \mathbf 1_{\log d(S_n x, y)\geq -\eta  n^{1/4}} \varphi(S_n x)  \Big).$$
		
When $S_n x \in\supp(\chi_k)$, we have  $-(k+1) \zeta \leq \log d(S_n x,y) \leq -(k-1)\zeta$,  so 
	$$  \psi\big(t+ \sigma(S_n,x) +\log d(S_n x,  y) - n \gamma\big) \geq \psi_\ast \big( t + \sigma(S_n,x)  -k \zeta- n \gamma \big) =  \psi_{\ast,t,k}\big(\sigma(S_n,x)  - n \gamma\big).      $$
	Since $\mathbf 1_{\log d(z,  y)\geq -\eta n^{1/4}}  \geq \sum_{0\leq k\leq  \zeta^{-1} \eta n^{1/4} - 1} \chi_k(z)$, we obtain
	\begin{align*}
	\psi \big(t+ \sigma(S_n,x) +\log d(S_n x,  y) - n \gamma \big) \mathbf 1_{\log d(S_n x,  y)\geq -\eta n^{1/4}}  \varphi(S_n x) \\
	\geq   \sum_{0\leq k\leq \zeta^{-1} \eta n^{1/4}-1} \psi_{\ast,t,k}\big(\sigma(S_n,x)-n\gamma\big)(\chi_k \varphi)(S_n x)  . 
	\end{align*} 
		Therefore, 	if  $k_0:= \lfloor \zeta^{-1} \eta n^{1/4} \rfloor$,  then
	\begin{equation*}
	\oA_n(t) \geq	 \sum_{0\leq k\leq  \zeta^{-1} \eta n^{1/4}-1} \oE_n\big(\psi_{\ast,t,k}\cdot\chi_k \varphi \big) = \oB_{\ast,n}(t) -   \oE_n\big(\psi_{\ast,t,k_0}\cdot\chi_{k_0} \varphi \big) -   \oE_n\big(\psi_{\ast,t,0}\cdot \Phi_n^\star \big).
	\end{equation*}
	Arguing as in the proof of Lemma \ref{lemma:llt-ineq-1}, we see that the last two terms above are $\eta^{-2}o(1)$ as $n\to +\infty$, uniformly in $t \in \R$. The lemma follows. 
\end{proof}		

Let $0<\delta\leq 1$. Lemma \ref{lemma:conv-Fourier-2} yields a smooth function $(\psi_\ast)^-_{\delta}$  such that $\widehat {(\psi_\ast)^-_{\delta}}$ has support in $[-\delta^{-2},\delta^{-2}]$,  $$ (\psi_\ast)^-_{\delta} \leq \psi_\ast,\quad \lim_{\delta\to 0}  (\psi_\ast)^-_{\delta} =  \psi_\ast    \quad \text{and} \quad  \lim_{\delta\to 0} \big \| (\psi_\ast)^-_{\delta} - \psi_\ast \big \|_{L^1} = 0.$$ 
Moreover, $\norm{ (\psi_\ast)^-_{\delta}}_\infty$, $\norm{ (\psi_\ast)^-_{\delta}}_{L^1}$ and $\|\widehat{  (\psi_\ast)^-_{\delta}}\|_{\Cc^1}$ are bounded by a constant independent of $\delta$.

For $t\in\R$ and $k\in\N$, set
$$ \psi_{t,k}^-(u):=  (\psi_\ast)^-_{\delta}(u+t - k \zeta)$$ and define
\begin{equation*}
	\oB_n^-(t):= \sum_{0\leq k\leq \zeta^{-1} \eta n^{1/4}} \oE_n\big(\psi_{t,k}^-\cdot \chi_k \varphi  \big)+  \oE_n\big(\psi_{t,0}^-\cdot \Phi_n^\star  \big)
\end{equation*}	
and 
\begin{equation*}
\oG_n^-(t):={1\over 2\pi}\widehat{(\psi_\ast)^-_{\delta}}(0)  \int_{-\infty}^{+\infty} e^{i\xi{t\over \sqrt n}} e^{-{ a^2\xi^2 \over 2}}  \,\diff \xi ={1\over \sqrt{2\pi} \, a} e^{-{t^2\over 2a^2 n}}\int_{\R} (\psi_\ast)^-_{\delta} (u)\,\diff u.
\end{equation*}

\begin{lemma} \label{lemma-R-S-2}
Fix  $0<\zeta\leq 1$,   $0<\eta\leq 1$ and $0< \delta \leq 1$. Then, there exists a constant $\widetilde C_{\delta}>0$ independent of $n$,  $\eta$ and $\zeta$ such that, for all $n \geq 1$,  
	$$\sup_{t\in\R}\big| \oB_n^-(t)- \oG_n^-(t) \big| \leq\widetilde  C_\delta \big( \zeta^{-2}\eta+n^{-1/2}\big).$$
\end{lemma}

\begin{proof}
	As in Subsection \ref{subsec:LLT-upper-limit}, we have the following identity (compare with \eqref{eq:R-formula})
	\begin{equation*}
\oB_n^-(t)= {1\over 2\pi}\int_{-\delta^{-2} \sqrt n}^{\delta^{-2} \sqrt n} \widehat{(\psi_\ast)^-_{\delta}}\Big({\xi\over \sqrt n}\Big) e^{i\xi{t\over \sqrt n}}\cdot e^{-i \xi \sqrt n\gamma}\oP_{{\xi\over \sqrt n}}^n(\Phi_{n,\xi} +\Phi_{n}^{\star}  )(x) \,\diff \xi,
\end{equation*}
	where $\Phi_{n,\xi}$ and $\Phi_{n}^{\star}$ are defined in \eqref{eq:Phi-xi-def} and \eqref{eq:Phi-star-def} respectively.  The proof of Lemma \ref{lemma-R-S} can be repeated by using  ${ (\psi_\ast)^-_{\delta}}$ instead of ${ (\psi^\ast)^+_{\delta}}$,  giving the desired estimate.
\end{proof}

We can now prove Proposition \ref{prop:LLT-lower-limit}.

\begin{proof}[Proof of Proposition \ref{prop:LLT-lower-limit}]
	Lemmas \ref{lemma:llt-ineq-2} and \ref{lemma-R-S-2} and the fact that $\oB_{\ast,n}(t) \geq \oB_n^-(t)$ give that,
	$$\oA_n(t) \geq \oG_n^-(t) - \eta^{-2}o(1) -\widetilde  C_\delta \big(\zeta^{-2}\eta+n^{-1/2}\big) \quad\text{as}\quad n\to+\infty,$$ uniformly in $t\in\R$.
	As in the proof of Proposition \ref{prop:LLT-upper-limit}, it is clear that, for every fixed $n$, 
	$$  \big| \oG_n^-(t) -  \oG_n(t) \big| \leq {1\over \sqrt{2\pi} \,  a}  \big\|  (\psi_\ast)^-_{\delta} - \psi \big \|_{L^1}.$$
	Therefore, 
	\begin{align*}
	\oA_n(t) -  \oG_n(t)  \geq    -  {1\over \sqrt{2\pi} \,  a}  \big\|  (\psi_\ast)^-_{\delta} - \psi \big \|_{L^1} - \eta^{-2}o(1) -\widetilde  C_\delta \big(\zeta^{-2} \eta+n^{-1/2}\big) \end{align*} uniformly in $t\in\R$, so
	\begin{align*}
	\liminf_{n\to \infty}  \inf_{t\in\R} \big( \oA_n(t) -  \oG_n(t)  \big)  \geq  -  {1\over \sqrt{2\pi} \,  a}  \big\|  (\psi_\ast)^-_{\delta} - \psi \big \|_{L^1}  -\widetilde  C_\delta \,  \zeta^{-2} \eta \\ \geq -
\frac{1}{\sqrt{2\pi} a} \big( \big\| (\psi_\ast)^-_{\delta}  - \psi_\ast  \big\|_{L^1}  +  \big\|  \psi_\ast - \psi   \big\|_{L^1} \big)   -\widetilde  C_\delta \,  \zeta^{-2} \eta.
\end{align*}

Arguing as in the end of the proof of Proposition \ref{prop:LLT-upper-limit},  we get that
\begin{equation*}  
\liminf_{n\to \infty}  \inf_{t\in\R} \big( \oA_n(t) -  \oG_n(t)  \big)  \geq -  \frac{1}{\sqrt{2\pi} a} \big(  B' \delta +  A' \zeta \big)   -\widetilde  C_\delta \,  \zeta^{-2} \eta,
\end{equation*}
for some constants $A', B' > 0$.	 We conclude as before by setting  $\eta = \zeta^3$ and letting $\zeta \to 0$ first, then $\delta \to 0$. 
\end{proof}

\begin{proof}[Proof of Proposition \ref{prop:LLT-product-test-fcn}]
	As observed above, the conclusion of Proposition \ref{prop:LLT-product-test-fcn} is equivalent to the limit \eqref{eq:LLT-main-limit} for $f=\psi \cdot \varphi$.  Propositions \ref{prop:LLT-upper-limit} and  \ref{prop:LLT-lower-limit} give that 
	$$\limsup_{n\to \infty}\sup_{t\in\R} \big(\oA_n(t)  -  \oG_n(t) \big) \leq 0 \quad \text{and} \quad \liminf_{n\to \infty}\inf_{t\in\R} \big(\oA_n(t)  -   \oG_n(t) \big) \geq 0$$ 
	respectively. This clearly implies \eqref{eq:LLT-main-limit} for $f=\psi \cdot \varphi$. The proposition follows.
\end{proof}

Now, the proof of  Theorem \ref{thm:LLT-coeff} in the general case can be concluded.

\begin{proof}[Proof of  Theorem \ref{thm:LLT-coeff}]
Let $f(u,z)$ be an arbitrary function on $\R \times \P^1$ which is continuous and has compact support. We need to show that \eqref{eq:LLT-main-limit} holds for $f$. Fix $\varepsilon > 0$.  From the proof of Lemma \ref{lemma:weierstrass},  one can find an integer $N >0$ and functions $h^\pm$ of the form  $h^\pm (u,z): =\sum_{1\leq j\leq N} \psi_j^\pm(u) \varphi_j^\pm(z)$ where $\psi^\pm_j$ is continuous with compact support in $\R$ and $\varphi_j^\pm \in \Cc^1(\P^1)$, such that $h^- \leq f \leq h^+$ and $\|h^\pm - f\|_{L^1(\Leb \otimes \nu)} \leq \varepsilon$. We can now argue as in the end of the proof of Theorem \ref{thm:LLT} in Section \ref{sec:LLT-norm} and conclude that  \eqref{eq:LLT-main-limit} holds for $f$.  It is not hard to see that all of our estimates are uniform in $x \in \P^1$.  The theorem follows.
\end{proof}

\section{Berry-Esseen bound for the norm cocycle} \label{sec:berry-esseen} 

In this section we prove Theorem \ref{thm:berry-esseen} as an application of Nagaev-Guivarc'h's spectral method.  Our proof is parallel to that of \cite[Theorem 3.7]{gouezel:spectral-methods}.  The original hypotheses of that theorem asks that the family of operators $\xi \mapsto \oP_\xi$ is of class $\Cc^3$,  which does not hold in our case under the third moment condition of Theorem  \ref{thm:berry-esseen}  (see e.g. Proposition \ref{prop:P_t-regularity}).  However,  we'll see below that the proof still works when the family is merely Lipschitz near zero,   which is true in our case,  and the leading eigenvalue has an order three expansion as in Lemma \ref{lemma:lambda_t-expansion-3}. Lemma \ref{lemma:lambda-estimates} is also crucial in our proof. As in Section \ref{sec:LLT-coeff}, we work here with the function space $\W$ defined in Section \ref{sec:W} for the value $p=2$.

\begin{proof}[Proof of  Theorem \ref{thm:berry-esseen}]
The Berry-Esseen Lemma (see \cite[XVI.3]{feller:book}, \cite[Theorem 3.7]{gouezel:spectral-methods} and Lemma \ref{lemma:BE-feller}  below) says that, for any interval $J \subset \R$ and any $T > 0$, the left hand side of \eqref{eq:berry-esseen} is bounded by a constant independent of $T,J,n$ and $x$  times
$$  \int_0^T \frac1\xi \, \Big | \E\Big(e^{i\frac{\xi}{\sqrt n}  (\sigma(S_n,x) - n \gamma)} \Big) - e^{-\frac{a^2 \xi^2}{2}} \Big|  \, \diff\xi + \frac{1}{T}.$$

Choosing $T:= \xi_0 \sqrt n$,  where $\xi_0 >0$ is a small constant,  the theorem  will follow once we show that
\begin{equation} \label{eq:berry-essen-main-estimate}
\int_0^{ \xi_0  \sqrt n} \frac 1\xi \, \Big | \E\Big(e^{i\frac{\xi}{ \sqrt n}  (\sigma(S_n,x) - n \gamma)} \Big) - e^{-\frac{a^2\xi^2}{2}} \Big| \, \diff \xi \lesssim \frac{1}{\sqrt n}.
\end{equation}

We have that $\E\big(e^{i\frac{\xi}{\sqrt n}  (\sigma(S_n,x) - n \gamma)} \big) = e^ {-i \xi \sqrt n \gamma}\oP^n_{\frac{\xi}{ \sqrt n}} \mathbf 1 (x)$ and $\oP^n_\xi = \lambda^n_\xi \oN_\xi + \oQ^n_\xi$ as operators on $\W$ for $\xi$ small enough by Corollary \ref{cor:P_t-decomp-W}.  
This gives $\E\big(e^{i \xi \sigma(S_n,x)}\big) = \lambda_\xi^n \oN_\xi \mathbf 1(x) + \oQ_\xi^n  \mathbf 1 (x)$, so 
$$\E\Big(e^{i\frac{\xi}{ \sqrt n}  (\sigma(S_n,x) - n \gamma)}\Big)= e^ {-i \xi \sqrt n \gamma}\lambda_{\frac{\xi}{ \sqrt n}}^n \oN_{\frac{\xi}{ \sqrt n}} \mathbf 1(x) + e^ {-i \xi \sqrt n \gamma} \oQ^n_{\frac{\xi}{ \sqrt n}}\mathbf 1 (x).$$ 

From Lemma \ref{lemma:Pt-regularity-3-moments}, the families $\xi \mapsto \oP_\xi$, $\oN_\xi$ and $\oQ_\xi$ acting on $\W$ are locally Lipschitz in $\xi$.  The  same  holds for the maps $\xi \mapsto \oN_\xi \mathbf 1(x)$ and $\xi \mapsto \oQ_\xi \mathbf 1(x)$ because,  as  $\| \cdot \|_{\infty} \lesssim \| \cdot \|_\W $, the evaluation map $\W \ni \psi \mapsto \psi(x)$ is continuous. 
Recall that $\oQ^n_0 \mathbf 1 = 0$. The Lipschitz property above gives that for $|\xi| \leq \xi_0 \sqrt n$, we have
\begin{equation} \label{eq: d_n(t)-estimate}
\big|\oQ^n_{\xi\over \sqrt n} \mathbf 1 (x)\big| \leq  B \beta^n {\xi\over \sqrt n}
\end{equation}
for some constants $B>0$ and $0<\beta<1$, see \cite[Remark 3.8]{gouezel:spectral-methods}.  Therefore,   $$\int_0^{ \xi_0  \sqrt n} \frac1\xi \, \Big |  e^ {-i \xi \sqrt n \gamma} \oQ^n_{\frac{\xi}{ \sqrt n}}\mathbf 1 (x) \Big|  \, \diff \xi \leq \xi_0  B  \beta^n  \lesssim  \frac{1}{\sqrt n}.$$

Thus, \eqref{eq:berry-essen-main-estimate} will follow if we show that $$ \int_0^{ \xi_0 \sqrt n} \frac1\xi \, \Big |e^ {-i \xi \sqrt n \gamma}\lambda_{\frac{\xi}{ \sqrt n}}^n \oN_{\frac{\xi}{ \sqrt n}} \mathbf 1(x)  - e^{-\frac{a^2\xi^2}{2}} \Big|  \, \diff \xi \lesssim  \frac{1}{\sqrt n}.$$
The above integral is bounded by
\begin{equation} \label{eq:B-E-two-integrals}
\int_0^{ \xi_0  \sqrt n} \frac1\xi  \, \big | \lambda_{\frac{\xi}{ \sqrt n}}\big|^n  \,   \Big| \oN_{\frac{\xi}{ \sqrt n}}  \mathbf 1 (x)  -  1 \Big|  \, \diff \xi  + \int_0^{ \xi_0  \sqrt n} \frac1\xi \, \Big | e^ {-i \xi \sqrt n \gamma}\lambda_{\frac{\xi}{\sqrt n}}^n   - e^{-\frac{a^2\xi^2}{2}} \Big|  \, \diff \xi .
\end{equation}

The Lipschitz property and the fact that $\oN_0 \mathbf 1= \mathbf 1$ gives that  $\big|\oN_{\frac{\xi}{ \sqrt n}}  \mathbf 1 (x) -  1 \big| \lesssim  \frac{\xi}{\sqrt n}$ for $0 \leq \xi \leq \xi_0  \sqrt n$.  Also,  $\big | \lambda_{\frac{\xi}{ \sqrt n}} \big| ^ n \leq e^{-{a^2\xi^2\over 3}}$ for $|\xi| \leq \xi_0  \sqrt n$ from Lemma \ref{lemma:lambda-estimates}.  We conclude that the first integral  in (\ref{eq:B-E-two-integrals}) is bounded by a constant times $\frac{1}{\sqrt n} \int_0^{+\infty}  e^{-\frac{a^2\xi^2}{3}} \, \diff \xi \lesssim  \frac{1}{\sqrt n}$.  For the second integral in (\ref{eq:B-E-two-integrals}), we use that $\big| e^ {-i \xi \sqrt n \gamma}\lambda_{\frac{\xi}{ \sqrt n}}^n   - e^{-\frac{a^2\xi^2}{2}}\big|$ is  $\frac{1}{\sqrt n} O\big(|\xi|^3 e^{ - \frac{a^2\xi^2}{2}}\big)$ for $0< \xi \leq \sqrt[6]{n}$ and $ \frac{1}{\sqrt n} O\big(e^{ - \frac{a^2 \xi^2}{4}}\big) \text{ for }  \sqrt[6]{n} <\xi \leq \xi_0  \sqrt{n}$,  as  in Lemma \ref{lemma:lambda-estimates}.  Therefore,   the second integral in (\ref{eq:B-E-two-integrals}) is bounded by a constant times  $$\frac{1}{\sqrt n} \int_0^{+\infty} \Big( \xi^2e^{-\frac{a^2\xi^2}{2}} + {1\over \sqrt[6] n} e^{ - \frac{a^2\xi^2}{4}} \Big) \, \diff \xi \lesssim  \frac{1}{\sqrt n}.$$  This completes the proof of the theorem.
\end{proof}

\begin{remark} \label{rmk:berry-esseen-norm}
Arguing as in Remark \ref{rmk:LLT-norms}, one can deduce from \eqref{eq:berry-esseen},  under a third moment condition,   an analogue bound for $\log\|S_n\|$, but with an $O(\log n \slash \sqrt n)$ rate instead. We leave the details to the reader. See  \cite{xiao-grama-liu:berry-eseen,cuny-dedecker-merlevede-peligrad} for similar results.
\end{remark}

\section{Berry-Esseen bound for the matrix coefficients}\label{sec:berry-esseen-coeff} 

This section is devoted to the proof of Theorem \ref{thm:BE-coeff}. We assume throughout this section that $\mu$ has a finite moment of order three. We follow the approach of \cite{DKW:BE-LLT-coeff}. As in Section \ref{sec:LLT-coeff},  we need to use here some weaker large deviation estimates and work with the space $\W$ instead of the H\"older spaces used there. In particular, the results from Sections \ref{sec:markov-op} to \ref{sec:W} will come into play.  As before, we work in this section with the space $\W$ defined in Section \ref{sec:W} for the value $p=2$,  even if we assume a finite moment of order three. In particular, $W^ {1,2} \cap \Cc^{\log^{1}} \subset \W \subset W^ {1,2}\cap \Cc^0$ and the family $\xi \mapsto \oP_\xi$ acting on $\W$ is locally Lipschitz in $\xi$ (Lemma \ref{lemma:Pt-regularity-3-moments}).

Our handling of the ``tail probabilities'' will be possible by Lemma \ref{lemma:LDT-p-moment} and the following large deviation estimates, due to Benoist-Quint,  which will be applied for $p=3$. 

\begin{proposition}[\cite{benoist-quint:CLT}--Proposition 4.1] \label{prop:BQLDT} 
	Let $\mu$ be a probability measure on $G=\SL_2(\C)$. Assume that $\mu$ is non-elementary and $\int_G \log^p \|g\| \, \diff \mu(g) < \infty$ for some $p > 1$. Let $\gamma > 0$ be the Lyapunov exponent of $\mu$. Then, for every $\epsilon>0$ there exist positive constants $C_{n,\epsilon}$ satisfying $\sum_{n\geq 1} n^{p-2}C_{n,\epsilon}< \infty$ such that, for every $x\in\P^1$,
	\begin{equation*}
	\mu^{*n}\big\{g\in G:\, | \sigma_g(x)-n\gamma| \geq \epsilon n\big\}\leq C_{n,\epsilon}.
	\end{equation*}
\end{proposition}

For convenience, we introduce the following definition.

\begin{definition} \label{def:conjugate-cf} \rm 
	For a real random variable $X$ with cumulative distribution function $F$ (c.d.f.\ for short), we define its \textit{conjugate characteristic function} by $$\phi_F(\xi):=\mathbf E\big(e^{-i\xi X}\big).$$
\end{definition}

We will use the following version of the Berry-Esseen lemma obtained in \cite{DKW:BE-LLT-coeff}. See also \cite[XVI.3]{feller:book}.

\begin{lemma}  \label{lemma:BE-feller} 
Let $F$ be a c.d.f.\ of some real random variable and let  $H$  be a differentiable real-valued function  on $\R$ with derivative $h$ such that $H(-\infty)=0,H(\infty)=1,|h(u)|\leq m$ for some constant $m >0$. Let $D>0$ and $0<\delta<1$ be real numbers such that $\big|F(u)-H(u) \big|\leq D \delta^2$ for $|u|\geq \delta^{-2}$. 
	Assume moreover that $h\in L^1$, $\widehat h\in\Cc^1$ and that $\phi_F$ is differentiable at zero.
	Then, there exist a constant $C>0$ (resp. $\kappa > 1$) independent of $F,H,\delta$ (resp. $D, F,H,\delta$), such that
	$$\sup_{u\in\R}\big|F(u)-H(u) \big|\leq {1\over \pi} \sup_{|u|\leq \kappa \delta^{-2}}    \Big|\int_{-\delta^{-2}}^{\delta^{-2}} {\Theta_u(\xi) \over \xi}     \,\diff \xi\Big|  +C\delta^2,$$
	where  $\Theta_u(\xi):=e^{iu\xi}\big(\phi_F(\xi)- \widehat {h}(\xi) \big)\widehat{\vartheta_\delta}(\xi)$,
	and $\vartheta_\delta$ is defined in Subsection \ref{subsec:regularization}.
\end{lemma}

Let $x:=[v]$ and $y:=[w^\star]$, where $w^\star\neq 0$ is orthogonal to $w$. As in Section \ref{sec:LLT-coeff}, the strategy to prove Theorem \ref{thm:BE-coeff} is to use the identity \eqref{eq:coeff-split} to work with the random variable $\sigma(S_n,x) +\log d(S_n x, y)$ instead of $\log{ |\lp S_n v, w \rp | \over \norm{v} \norm{w}} $, use the perturbed Markov operators to study $\sigma(S_n,x)$ and large deviation estimates together with a well-chosen partition of unity to deal with the term $\log d(S_n x, y)$.

For integers $k \geq 0$, let $\Tc_k : = \Tc_k^1$ be the annulus centered at $y$ defined in \eqref{eq:def-annulus} for $\zeta = 1$.   Let $\chi_k$, $k \geq 0$ be the functions obtained from Lemma \ref{lemma:partition-of-unity-2} with  $\zeta = 1$.  In particular, we have $\norm{\chi_k}_\W \lesssim 1 + k$.

\medskip

We now begin the proof of Theorem \ref{thm:BE-coeff}.  It  suffices to prove the result  for intervals of the type $J=(-\infty,b]$ with $b \in \R$, as the case of an arbitrary interval can be obtained as a consequence. For example, the case $(b,+\infty)$ follows directly by considering its complement. The case of $[b, +\infty)$ can be deduced by approximating it by $(b \pm \varepsilon, + \infty)$ and the case $(-\infty,b)$  follows  by taking the complement. Finally, by considering differences of the previous cases,  we obtain the result for bounded intervals. 
\medskip

Recall that  $\mu$ has a finite moment of order $3$ by assumption. Applying Lemma \ref{lemma:LDT-p-moment} to $\lfloor \sqrt[4] n \rfloor$ instead of $n$, $\ell = n$ and $p=3$, gives that  
$$\mu^{*n}\big\{g: \, \log d(gx,y)\leq - \sqrt[4] n  \big\}=  o(1/ \sqrt n) \quad\text{as}\quad n\to +\infty.$$
It follows that $\log d(S_n x,y)\leq -\sqrt[4] n $  with probability $o(1/\sqrt n)$. In particular, we can assume that $\log d(S_n x,y)> -\sqrt[4] n$ throughout the proof. In other words, in order to prove Theorem \ref{thm:BE-coeff}, it is enough to show that
\begin{equation}\label{goal-1-varphi}
\Big| \oL_n(b)  -  \frac{1}{\sqrt{2 \pi}  \, a }  \int_{-\infty}^b e^{-\frac{s^2}{2  a ^2}} \, \diff s \Big| \lesssim \frac{1}{\sqrt n},
\end{equation}
uniformly in $b$, where
$$\oL_n(b):=  \mathbf E \Big(  \mathbf 1_{{\sigma(S_n,x) +\log d(S_n x, y) - n \gamma\over \sqrt n}\leq b}   \mathbf 1_{ \log d(S_n x,y)>  -\sqrt[4] n   } \Big).$$

Let $\chi_k$ be as above.  It is clear from their definition, that
$$\sum_{0\leq k\leq \sqrt[4] n  - 1}\chi_k(z) \leq \mathbf 1_{  \log d(z,y)> - \sqrt[4] n } \leq  \sum_{0\leq k\leq \sqrt[4] n +1}\chi_k(z)$$ as functions on $\P^1$. Using that $\chi_k$ is supported by $\Tc_k$,  it follows that 
\begin{align} \label{eq:Ln-two-sided-bound}
\sum_{0\leq k\leq \sqrt[4] n -1} \hspace{-7pt} \E\Big(  \mathbf 1_{{\sigma(S_n,x) - n \gamma - k+1\over \sqrt n}\leq b} \, \chi_k(S_n x) \Big) \leq \oL_n(b)\leq  \sum_{0\leq k\leq \sqrt[4] n +1}   \mathbf E\Big(  \mathbf 1_{{\sigma(S_n,x) - n \gamma - k-1\over \sqrt n}\leq b} \, \chi_k(S_n x) \Big).
\end{align}

For $z \in \P^1$, let $$\Phi_{n}^{\star}  (z):= 1 - \sum_{0\leq k\leq \sqrt[4] n } \chi_k(z).$$ We use the same notation as in Section \ref{sec:LLT-coeff} to denote a slightly different function.  Define, for $b \in \R$
\begin{align*}
F_n(b):&=    \sum_{0\leq k\leq \sqrt[4] n } \E\Big(  \mathbf 1_{{\sigma(S_n,x) - n \gamma - k\over \sqrt n}\leq b}  \chi_k (S_n x) \Big)  +  \E \Big( \mathbf 1_{{\sigma(S_n,x) - n \gamma \over \sqrt n}\leq b} \, \Phi_{n}^{\star} (S_n x) \Big). 
\end{align*}

Notice that $F_n$ is non-decreasing, right-continuous, $F_n(-\infty)=0$ and $F_n(\infty)=1$. Therefore, it is the c.d.f.\ of some probability distribution.  Even if, as we'll see, the term involving $\Phi_{n}^{\star}$ has a negligible impact in our estimates, its presence in the definition of $F_n$ will make our computations more tractable. This could already be observed in Section \ref{sec:LLT-coeff}.

\begin{lemma} \label{lemma:Ln-Fn}
	Let $\oL_n$ and $F_n$ be as above. Then,
	$$ F_n(b - 1 / \sqrt n)  -  o(1/\sqrt n)  \leq \oL_n(b) \leq F_n(b + 1 / \sqrt n)  + o(1/\sqrt n)$$ as $n\to +\infty$, uniformly in $b \in \R$.
\end{lemma}

\begin{proof}
	Notice first that $\Phi_{n}^{\star}$ is non-negative, bounded by one and supported by an open ball $\mathbf \D_{n,y}$ of center $y$ and radius $O(e^{-\sqrt[4] n })$. As noted above, the probability that $S_n x$ belongs to $\mathbf \D_{n,y}$ is $o(1 / \sqrt n)$. This yields the following bounds for the second term in the definition of $F_n$:  $$0 \leq \E \Big( \mathbf 1_{{\sigma(S_n,x) - n \gamma \over \sqrt n}\leq b} \, \Phi_{n}^{\star} (S_n x) \Big) \leq \E \Big(\Phi_{n}^{\star}(S_n x) \Big)  =o(1 / \sqrt n).$$
	
	Therefore, in order to prove the lemma, we can replace $F_n$ by the function
	\begin{equation} \label{eq:def-Fn-tilde}
	\widetilde F_n(b) := \sum_{0\leq k\leq \sqrt[4] n } \E\Big(  \mathbf 1_{{\sigma(S_n,x) - n \gamma - k \over \sqrt n}\leq b}  \chi_k (S_n x) \Big).
	\end{equation}
	
	From the second inequality in \eqref{eq:Ln-two-sided-bound}, we have
	\begin{align*}
	\oL_n(b) - \widetilde F_n(b + 1 / \sqrt n)  \leq   \E\Big(  \mathbf 1_{{\sigma(S_n,x) - n \gamma - k^+-1\over \sqrt n}\leq b}  \chi_{k^+} (S_n x) \Big) \leq   \E\Big( \chi_{k^+} (S_n x) \Big),
	\end{align*}
	where $k^+:= \lfloor \sqrt[4] n  \rfloor + 1$. Since  $\chi_k \leq \mathbf 1_{\D(y,e^{-k+1})}$ and, by arguments similar to the ones used before,  the probability that $\log d(S_n x,y)\leq - \sqrt[4] n   +1 $  is $o( 1/\sqrt n)$, the above quantity is $o(1/\sqrt n)$. This gives the second inequality in the lemma.
	
	Using now the first inequality in \eqref{eq:Ln-two-sided-bound} and letting $k^-:= \lfloor  \sqrt[4] n   \rfloor$, we obtain
	\begin{align*}
	\widetilde F_n(b -  1 /\sqrt n)  - \oL_n(b) \leq   \E\Big(  \mathbf 1_{{\sigma(S_n,x) - n \gamma - k^- +1\over \sqrt n}\leq b}  \chi_{k^-} (S_n x) \Big) \leq   \E\Big(\chi_{k^-} (S_n x) \Big),
	\end{align*}
	which is $o( 1/\sqrt n)$ by the same arguments as before. The lemma follows.
\end{proof}

Introduce $$\Phi_{n,\xi} (z):= \sum_{0\leq k\leq  \sqrt[4] n  }  e^{i \xi{k \over \sqrt n}}\chi_k(z),$$ where, again, we use the same notation as in Section \ref{sec:LLT-coeff} to denote a slightly different function.  We have the following expression for the conjugate characteristic function of $F_n$. 

\begin{lemma} \label{lemma:char-function-Fn}
	The conjugate characteristic function of $F_n$ (cf.\ Definition \ref{def:conjugate-cf}) is given by
	$$\phi_{F_n}(\xi)=  e^{i\xi\sqrt n \gamma} \oP_{-{\xi\over \sqrt n}}^n \Phi_{n,\xi} (x) + e^{i\xi\sqrt n \gamma} \oP_{-{\xi\over \sqrt n}}^n \Phi_{n}^{\star}   (x) = e^{i\xi\sqrt n \gamma} \oP_{-{\xi\over \sqrt n}}^n \big( \Phi_{n,\xi} +  \Phi_{n}^{\star} \big)  (x).$$
	In particular,  $\phi_{F_n}$ is differentiable near zero.
\end{lemma}

\begin{proof}
The formula for $\phi_{F_n}$ is proven in \cite[Lemma 3.5]{DKW:BE-LLT-coeff}. Its differentiability follows directly from the fact that the family $\xi \to \oP_\xi$ acting on $\Cc^0$ is differentiable, which can be easily proved using the arguments of Section \ref{sec:perturbed-markov-op}, see e.g.\ Proposition \ref{prop:P_t-regularity}.
\end{proof}

The following result is the analogue of Lemma \ref{lemma:norm-Phi-2} in the current setting. Recall that we are working with the space $\W$ defined in Section \ref{sec:W} for the value $p=2$ and $\norm{\chi_k}_\W\lesssim k$.

\begin{lemma}  \label{lemma:norm-Phi}
	Let $\Phi_{n,\xi}, \Phi_{n}^{\star}$ be the functions on $\P^1$ defined above.  Then,  $\Phi_{n}^{\star} $ is supported by $\big\{z:\,\log d(z,y)\leq - \sqrt[4] n  + 1\big\}$ and  the following identity holds
	\begin{equation} \label{eq:psi_xi+psi_T}
	\Phi_{n,\xi} + \Phi_{n}^{\star} = \mathbf 1 + \sum_{0\leq k\leq  \sqrt[4] n  } \big(  e^{i \xi{k \over \sqrt n}} - 1 \big) \chi_k.
	\end{equation}
	Moreover, there  is a constant $C>0$ independent of $\xi$ and $n$ such that   
	$$\norm{\Phi_{n,\xi} }_{\W}\leq C \,  \sqrt n  ,\quad  \norm{\Phi_{n}^{\star}  }_{\W} \leq C \,  \sqrt n,$$
	and
	\begin{equation}  \label{eq:norm-Phi}
	\norm{\Phi_{n,\xi} + \Phi_{n}^{\star}}_{\W}\leq C (1+|\xi|). 
	\end{equation}

\end{lemma}

\begin{proof}
All the assertions are clear from the definitions of $\Phi_{n,\xi}$, $ \Phi_{n}^{\star}$ and parts (1) and (3) of Lemma \ref{lemma:partition-of-unity-2} except \eqref{eq:norm-Phi}. It is enough to estimate  $\norm{\Phi_{n,\xi} +\Phi_{n}^{\star} }_{W^{1,2}\cap \log^1}$, since  $\norm{\cdot}_{\W} \lesssim \norm{\cdot}_{W^{1,2}\cap \log ^1}$.  For the $W^{1,2}$ part, we use that $\norm{\chi_k}_{W^{1,2}}\lesssim 1$, from the proof Lemma \ref{lemma:partition-of-unity-2}. Then,
	\begin{align*}
	\norm{\Phi_{n,\xi} +\Phi_{n}^{\star} }_{W^{1,2}}\leq 1+\sum_{0\leq k\leq \sqrt[4] n}\big|e^{i \xi{k\over \sqrt n}} - 1 \big|\norm{\chi_k}_{W^{1,2}}
\lesssim 1+\sum_{0\leq k\leq \sqrt[4] n} |\xi|{k\over \sqrt n} \lesssim 1+|\xi|.   
	\end{align*}
	We now estimate the $\Cc^{\log^1}$-norm. It is clear that $\norm{\Phi_{n,\xi} +\Phi_{n}^{\star} }_{\infty} \lesssim 1$, so it is enough to estimate $[\Phi_{n,\xi} +\Phi_{n}^{\star}]_{\log^1}$.  Let $z_1 \neq z_2 \in \P^{d-1}$. Using \eqref{eq:psi_xi+psi_T} and the fact that, at given point, $\chi_k$ is non-zero for at most two values of $k$, there are indices  $0 \leq k_1,k_2,k_3,k_4 \leq \sqrt[4] n$ such that $$\big| (\Phi_{n,\xi} + \Phi_{n}^{\star})(z_1) - (\Phi_{n,\xi} + \Phi_{n}^{\star})(z_2) \big| \leq \sum_{1 \leq \ell \leq 4} |e^{i \xi{k_\ell \over \sqrt n}} - 1 \big| \big| \chi_{k_\ell}(z_1) - \chi_{k_\ell}(z_2) \big|.$$
Using that $[\chi_k]_{\log^1} \lesssim 1 + k$ from Lemma \ref{lemma:partition-of-unity-2}-(5), we conclude that
\begin{align*}
\big| (\Phi_{n,\xi} + \Phi_{n}^{\star})(z_1) - (\Phi_{n,\xi} + \Phi_{n}^{\star})(z_2) \big| \big(1+|\log d(z_1,z_2)|\big) \lesssim \sum_{1 \leq \ell \leq 4} |\xi| {k_\ell \over \sqrt n} (1 + k_\ell) \lesssim |\xi|
\end{align*}
where in the last step we have used that $k_\ell \leq  \sqrt[4] n$. This shows that $\norm{\Phi_{n,\xi} +\Phi_{n}^{\star} }_{\log^1} \lesssim 1 + |\xi|$ and finishes the proof of the lemma.
\end{proof}

Let $$H(b): =\frac{1}{\sqrt{2 \pi}  \, a } \int_{-\infty}^b e^{-\frac{s^2}{2  a ^2}} \, \diff s$$ be the c.d.f.\ of the normal distribution $\cali N(0; a ^2)$. In the notation of Lemma \ref{lemma:BE-feller}, we have $h(u): =\frac{1}{\sqrt{2 \pi}  \, a }   e^{-\frac{u^2}{2  a ^2}}$ and $\widehat h(\xi) =  e^{-{ a ^2 \xi^2\over 2}}$. 

\medskip

Let $\xi_0$ be a small constant. We choose it so that  Lemma \ref{lemma:lambda-estimates} holds and the decomposition  of $\oP_\xi$ from Corollary \ref{cor:P_t-decomp-W} is valid for $|\xi| \leq \xi_0$.

\medskip

\begin{lemma} \label{lemma:F_n-estimate}
	Let $F_n$ and $H$ be as above. Then, $\big|F_n(b)-H(b)\big|\lesssim 1/\sqrt n$ for $|b|\geq \xi_0 \sqrt n$, uniformly in $b$.
\end{lemma}

\begin{proof}
	We only consider the case where $b\leq -\xi_0\sqrt n$. The case $b \geq \xi_0\sqrt n$ can be treated similarly using $1-F_n$ and $1-H$ instead of $F_n$ and $H$. We can also assume that $n$ is large enough. Clearly, $H(b)\lesssim 1/\sqrt n$ for $b\leq -\xi_0\sqrt n$,  so it is enough to bound $F_n(b)$.  
	
	For $0\leq k\leq  \sqrt[4] n$, we have
	\begin{align*}
	\mathbf P\Big(  {\sigma(S_n,x) - n \gamma - k \over \sqrt n}\leq -\xi_0\sqrt n  \Big)&=\mathbf P\Big(\sigma(S_n,x) - n \gamma  \leq -\xi_0 n +k  \Big)\\
	&\leq \mathbf P\Big(\sigma(S_n,x) - n \gamma  \leq  -  \xi_0 n / 2\Big),
	\end{align*}
	since $n$ is large. By Proposition \ref{prop:BQLDT} applied with $p=3$ and $\ep= \xi_0 / 2$, there exists a sequence of positive constants $C_{n,\xi_0/2}$ with $\sum_{n \geq 1} n C_{n,\xi_0/2} < \infty $ such that the last quantity is bounded by $C_{n,\xi_0/2}$. In particular, we have that  $C_{n,\xi_0/2} = o (1 / n)$.
	
	Using the definition of $F_n$ and the fact that $\E \big(\Phi_{n}^{\star}(S_n x) \big)  =o( 1 / \sqrt n)$ (see the proof of Lemma \ref{lemma:Ln-Fn}),  it follows that 
	$$ F_n (-\xi_0\sqrt n)\leq \sum_{0\leq k\leq  \sqrt[4] n }C_{n,\xi_0/2}+ o(1/\sqrt n)  \leq \sqrt[4] n \, o(1/n) + o(1/\sqrt n )=o( 1/\sqrt n).$$ 
	As $F_n(b)$ is non-decreasing in $b$, one gets that $F_n(b) =o( 1 / \sqrt n)$ for all $b\leq -\xi_0\sqrt n$. The lemma follows.
\end{proof}

Lemmas \ref{lemma:char-function-Fn} and \ref{lemma:F_n-estimate}  imply that $F_n$ satisfies the conditions of Lemma \ref{lemma:BE-feller} with $\delta_n:=(\xi_0 \sqrt n)^{-1/2}$.  Let $\kappa > 1$ be the constant appearing in that lemma. For simplicity, by taking a smaller $\xi_0$ if necessary, one can assume that $2\kappa \xi_0 \leq 1$.  Then,  Lemma \ref{lemma:BE-feller} gives that
\begin{equation} \label{eq:BE-main-estimate}
\sup_{b \in\R}\big|F_n(b)-H(b)\big|\leq  {1\over \pi} \sup_{|b|\leq \sqrt n}    \Big|  \int_{-\xi_0\sqrt n}^{\xi_0\sqrt n} {\Theta_b(\xi) \over \xi}    \,\diff \xi  \Big|+{C\over \sqrt n},
\end{equation}
where $C>0$ is a constant independent of $n$, 
$$\Theta_b(\xi):=e^{ib\xi}\big( \phi_{F_n}(\xi)- \widehat {h}(\xi) \big)\widehat{\vartheta_{\delta_n}}(\xi),$$
and $\vartheta_\delta$ is defined in Subsection \ref{subsec:regularization}.

\medskip

We now estimate the integral in the right hand side of \eqref{eq:BE-main-estimate}.  Motivated by Lemma \ref{lemma:lambda-estimates},  Corollary \ref{cor:P_t-decomp-W} and Lemma \ref{lemma:char-function-Fn}, in order to approximate $\phi_{F_n}$, introduce the following function:
$$\widetilde h_n(\xi):=  \big(\oN_0 \Phi_{n,\xi} \big)  e^{-{ a ^2 \xi^2\over 2}} +    \big(\oN_0 \Phi_{n}^{\star} \big)    e^{-{ a ^2 \xi^2\over 2}} =  e^{-{ a ^2 \xi^2\over 2}}  \oN_0 (  \Phi_{n,\xi} +  \Phi_{n}^{\star} ).$$ Notice that  $\oN_0 (  \Phi_{n,\xi} +  \Phi_{n}^{\star} )$ is a constant independent of $x$.  Define also
$$\Theta_b^{(1)}(\xi):= e^{ib\xi}\big( \phi_{F_n}(\xi)- \widetilde {h}_n(\xi) \big)\widehat{\vartheta_{\delta_n}}(\xi) \quad\text{and} 
\quad \Theta_b^{(2)}(\xi):= e^{ib\xi}\big(  \widetilde {h}_n(\xi)-\widehat h(\xi) \big)\widehat{\vartheta_{\delta_n}}(\xi),$$
so that $\Theta_b =  \Theta_b^{(1)} + \Theta_b^{(2)}$. 
\medskip

\begin{lemma} \label{lemma:theta-1-bound} We have
	$$\sup_{|b|\leq \sqrt n}    \Big|  \int_{-\xi_0\sqrt n}^{\xi_0\sqrt n} {\Theta_b^{(1)}(\xi) \over \xi}    \,\diff \xi  \Big| \lesssim {1\over \sqrt n}.$$
\end{lemma}

\begin{proof}
	Using Lemma \ref{lemma:char-function-Fn} and the decomposition of $\oP_\xi$ on $\W$ from Corollary \ref{cor:P_t-decomp-W}, we write
	$$\Theta_b^{(1)}=\Lambda_1+\Lambda_2+\Lambda_3,$$
	where 
	$$ \Lambda_1(\xi):= e^{ib\xi}\Big(  e^{i\xi\sqrt n \gamma} \lambda_{-{\xi\over \sqrt n}}^n \oN_{-{\xi\over \sqrt n}} (\Phi_{n,\xi} +\Phi_{n}^{\star}  ) (x) -      e^{-{ a ^2 \xi^2\over 2}}   \oN_0 (\Phi_{n,\xi} +\Phi_{n}^{\star}  )  (x)   \Big) \widehat{\vartheta_{\delta_n}}(\xi), $$ 
	$$\Lambda_2(\xi):=e^{ib\xi}\Big(  e^{i\xi\sqrt n \gamma} \oQ_{-{\xi\over \sqrt n}}^n  (\Phi_{n,\xi} +\Phi_{n}^{\star}  )(x) -e^{i\xi\sqrt n \gamma}\oQ_0^n  (\Phi_{n,\xi} +\Phi_{n}^{\star}  )(x)  \Big)\widehat{\vartheta_{\delta_n}}(\xi)$$
	and
	$$ \Lambda_3(\xi):=e^{ib\xi}e^{i\xi\sqrt n \gamma}\oQ_0^n  (\Phi_{n,\xi} +\Phi_{n}^{\star}  )(x)  \,\widehat{\vartheta_{\delta_n}}(\xi).$$
	
	We will estimate the integral  of $\Lambda_j(\xi) / \xi$,  $j=1,2,3$, separately.  Notice that, from \eqref{eq:psi_xi+psi_T}, we have that $\Phi_{n,0} + \Phi_{n}^{\star} = \mathbf 1 $.   Together with the fact that $\lambda_0=1$, $\oN_0 \mathbf 1 = \mathbf 1$, $\oQ_0 \mathbf 1 = 0$ and the Lipschitz regularity of $\xi \mapsto \oP_\xi$ (Lemma \ref{lemma:Pt-regularity-3-moments}), we get that  $\Lambda_j$ is a locally Lipschitz function vanishing at the origin for  $j=1,2,3$.  In particular, $\Lambda_j(\xi) / \xi$ is a bounded function of $\xi$ for $j=1,2,3$. Moreover, $\Lambda_3$ is differentiable. We see here why the presence of the function $\Phi_{n}^{\star}$ is useful.
	
	\vskip 5pt
	
	In order to estimate $\Lambda_2$ observe that, for $|\xi|$ small, the norm of the operator $\oQ^n_\xi - \oQ^n_0$ is bounded by a constant times $|\xi| n \beta^n$ for some $0<\beta<1$. This can be seen by writing the last difference as $\sum_{\ell=0}^{n-1}  \oQ_\xi^{n-\ell-1}(\oQ_\xi - \oQ_0)\oQ_0^\ell$, applying Corollary \ref{cor:P_t-decomp-W}-(5) and using the fact that $\norm{\oQ_\xi - \oQ_0}_{\W} \lesssim |\xi|$ by Lemma \ref{lemma:Pt-regularity-3-moments}. Therefore, we have 
	$$  \Big|  \oQ_{-{\xi\over \sqrt n}}^n  (\Phi_{n,\xi} +\Phi_{n}^{\star}  )(x) - \oQ_0^n (\Phi_{n,\xi} +\Phi_{n}^{\star}  )(x) \Big| \lesssim {|\xi|\over \sqrt n} n \beta^n \norm{ \Phi_{n,\xi} +\Phi_{n}^{\star}  }_{\W}.$$
	Using Lemma \ref{lemma:norm-Phi}, this gives
	\begin{align*}
	\Big|\int_{-\xi_0\sqrt n}^{\xi_0\sqrt n} {\Lambda_2(\xi) \over \xi}    \,\diff \xi\Big|
	&\leq\int_{-\xi_0\sqrt n}^{\xi_0\sqrt n} {1\over |\xi|} \cdot \Big| \oQ_{-{\xi\over \sqrt n}}^n  (\Phi_{n,\xi} +\Phi_{n}^{\star}  )(x) -  \oQ_0^n (\Phi_{n,\xi} +\Phi_{n}^{\star}  )(x) \Big|\,\diff\xi \\
	&\lesssim     \int_{-\xi_0\sqrt n}^{\xi_0\sqrt n} {1\over |\xi|} \cdot  |\xi| \sqrt n \beta^n \norm{ \Phi_{n,\xi} +\Phi_{n}^{\star}  }_{\W} \,\diff\xi \lesssim n \beta^n \sqrt n  \lesssim {1\over \sqrt n}.     
	\end{align*}
	
	We now estimate $\Lambda_3$ using its derivative $\Lambda'_3$. Recall that $\norm{\widehat{\vartheta_{\delta_n}}}_{\Cc^1}\lesssim 1,|b|\leq \sqrt n$ and  $\big|\oQ_0^n  (\Phi_{n,\xi} +\Phi_{n}^{\star}  )(x) \big|\lesssim \beta^n\norm{ \Phi_{n,\xi} +\Phi_{n}^{\star}  }_{\W}$, where $0<\beta<1$ is as before. A direct computation using \eqref{eq:psi_xi+psi_T} gives for $|\xi|\leq \xi_0 \sqrt n$,
	\begin{align*}
	| \Lambda_3'(\xi)| &\lesssim  \Big|b \oQ_0^n  (\Phi_{n,\xi} +\Phi_{n}^{\star}  )(x)\Big|+    \Big|\sqrt n \gamma \oQ_0^n  (\Phi_{n,\xi} +\Phi_{n}^{\star}  )(x)\Big|  \\
	& \quad \quad + \sum_{1 \leq k \leq \sqrt[4] n } \Big|{k\over \sqrt n} e^{i \xi \frac{k}{\sqrt n}}\oQ_0^n \chi_k (x)    \Big|  + \Big|\oQ_0^n  (\Phi_{n,\xi} +\Phi_{n}^{\star}  )(x)\Big| \norm{\widehat{\vartheta_{\delta_n}}}_{\Cc^1} \\
	&\lesssim \sqrt n  \beta^n \sqrt n  + \sum_{1 \leq k \leq \sqrt[4] n } {k^2\beta^n\over \sqrt n}  +   \beta^n \sqrt n   \lesssim \sqrt n  \beta^n \sqrt n ,
	\end{align*}
	where we have used that $|b| \leq \sqrt n$, $\norm{\chi_k}_{\W} \lesssim k$ and $\norm{ \Phi_{n,\xi} +\Phi_{n}^{\star}  }_{\W} \lesssim \sqrt n $, see Lemma \ref{lemma:norm-Phi}.
	
	Applying the mean value theorem, we get
	\begin{align*}
	\Big|\int_{-\xi_0\sqrt n}^{\xi_0\sqrt n} {\Lambda_3(\xi) \over \xi}    \,\diff \xi\Big|\leq 2\xi_0\sqrt n\sup_{|\xi|\leq \xi_0\sqrt n} | \Lambda_3'(\xi)|\lesssim n   \beta^n \sqrt n\lesssim {1 \over \sqrt n}.
	\end{align*}
	\vskip 5pt
	
	It remains to estimate the term involving $\Lambda_1$. We have
	$$ \Big|\int_{-\xi_0\sqrt n}^{\xi_0\sqrt n} {\Lambda_1(\xi) \over \xi}   \diff \xi\Big| \leq \int_{-\xi_0\sqrt n}^{\xi_0\sqrt n} {1\over |\xi|}\cdot \Big|  e^{i\xi\sqrt n \gamma} \lambda_{-{\xi\over \sqrt n}}^n \oN_{-{\xi\over \sqrt n}} (\Phi_{n,\xi} +\Phi_{n}^{\star}  ) (x) -   e^{-{ a ^2 \xi^2\over 2}}   \oN_0 (\Phi_{n,\xi} +\Phi_{n}^{\star}  )(x)    \Big|   \diff \xi.$$
	We split the last integral into two integrals using
	$$ \Gamma_1(\xi):=   e^{i\xi\sqrt n \gamma} \lambda_{-{\xi\over \sqrt n}}^n \oN_{-{\xi\over \sqrt n}}  (\Phi_{n,\xi} +\Phi_{n}^{\star}  ) (x) -      e^{i\xi\sqrt n \gamma} \lambda_{-{\xi\over \sqrt n}}^n  \oN_0 (\Phi_{n,\xi} +\Phi_{n}^{\star}  )(x)   $$
	and
	$$  \Gamma_2(\xi):=   e^{i\xi\sqrt n \gamma} \lambda_{-{\xi\over \sqrt n}}^n  \oN_0 (\Phi_{n,\xi} +\Phi_{n}^{\star}  ) (x) -      e^{-{ a ^2 \xi^2\over 2}}   \oN_0 (\Phi_{n,\xi} +\Phi_{n}^{\star}  ) (x) .   $$
	\vskip 5pt
	
	From   Lemmas \ref{lemma:Pt-regularity-3-moments} and  \ref{lemma:norm-Phi}, one has 
	$$\Big\|\big(\oN_{-{\xi\over \sqrt n}}-\oN_0 \big)  (\Phi_{n,\xi} +\Phi_{n}^{\star}  ) \Big\|_\infty \lesssim {|\xi| \over \sqrt n} \norm{ \Phi_{n,\xi} +\Phi_{n}^{\star}  }_{\W}\lesssim    {|\xi| \over \sqrt n}(1 + |\xi|). $$
	
 Also, $\big|\lambda_{-{\xi\over \sqrt n}}^n\big|\leq e^{-{ a ^2\xi^2\over 3}}$ from Lemma \ref{lemma:lambda-estimates}, so	
	$$ \int_{-\xi_0\sqrt n}^{\xi_0\sqrt n}  {1\over |\xi|}\cdot \big| \Gamma_1(\xi) \big|  \,\diff \xi \lesssim  \int_{-\infty}^{+\infty}  {1\over |\xi|} \cdot  e^{-{ a ^2\xi^2\over 3}} {|\xi|\over \sqrt n}   (1+|\xi|)  \,\diff \xi \lesssim {1 \over \sqrt n}. $$
	
	We now claim that, for $|\xi| < \xi_0 \sqrt n$, we have
		\begin{equation} \label{eq:lambda-estimate}
\Big| e^{i\xi\sqrt n \gamma} \lambda_{-{\xi\over \sqrt n}}^n-e ^{-{ a ^2\xi^2\over 2}}  \Big|\lesssim {1\over \sqrt n} |\xi|^3 e^{-{ a ^2\xi^2\over 4}}.
	\end{equation}
	Indeed, Lemma \ref{lemma:lambda-estimates} directly gives \eqref{eq:lambda-estimate} when $|\xi|\leq \sqrt[6] n$ and when $\sqrt[6] n<|\xi|\leq \xi_0\sqrt n$, the same lemma gives the bound ${1 \over \sqrt n}e^{-{a^2\xi^2\over 4}}$ which is bounded by  ${1\over \sqrt n} |\xi|^3 e^{-{ a ^2\xi^2\over 4}}$ since $|\xi| \geq 1$.
	
	By definition of $\Phi_{n,\xi}$ and Lemma \ref{lemma:norm-Phi}, $ \big| \Phi_{n,\xi} +\Phi_{n}^{\star} \big| \leq \Phi_{n,0} +\Phi_{n}^{\star} = 1$, so  $\big|\oN_0  (\Phi_{n,\xi} +\Phi_{n}^{\star}  )\big| \leq 1$. Therefore, using \eqref{eq:lambda-estimate}, we obtain
	$$ \int_{-\xi_0\sqrt n}^{\xi_0\sqrt n}  {1\over |\xi|}\cdot \big| \Gamma_2(\xi) \big|  \,\diff \xi \lesssim \int_{-\xi_0\sqrt n}^{\xi_0\sqrt n} {1\over |\xi|}\cdot |\xi|^3  {1\over \sqrt n} e^{-{ a ^2\xi^2\over 4}}   \,\diff \xi  \lesssim { 1\over \sqrt n}.$$
We deduce that 
	$$\Big|\int_{-\xi_0\sqrt n}^{\xi_0\sqrt n} {\Lambda_1(\xi) \over \xi}    \,\diff \xi\Big|  \lesssim {1 \over \sqrt n},$$ which ends the proof of the lemma. \end{proof}

\begin{lemma} \label{lemma:theta-2-bound}
	We have
	$$\sup_{|b|\leq  \sqrt n}    \Big|  \int_{-\xi_0\sqrt n}^{\xi_0\sqrt n} {\Theta_b^{(2)}(\xi) \over \xi}   \,\diff \xi  \Big| \lesssim { 1\over \sqrt n}.$$
\end{lemma}

\begin{proof}
From the proofs of  Theorem 2.2 and Proposition 4.5 in \cite{benoist-quint:CLT}, there exist positive constants $C_k$, $k \geq 1$ such that $\nu\big(\D(x,e^{-k+1})\big) \leq C_k$  and $\sum_{k \geq 1} k C_k < \infty$. Since $\chi_k$ is bounded by $1$ and is supported by $ \Tc_k \subset \D(y,e^{-k+1})$,  we get that  
	$$\sum_{k\geq 0} k (\oN_0 \chi_k) <+\infty.$$
	
	Recall that  $\widehat h(\xi) =  e^{-{ a ^2 \xi^2\over 2}}$ and $\widetilde h_n(\xi)=  e^{-{ a ^2 \xi^2\over 2}} \oN_0 (\Phi_{n,\xi} +  \Phi_{n}^{\star})$. Using \eqref{eq:psi_xi+psi_T} and $\oN_0 \mathbf 1 = \mathbf 1$, we get
	\begin{align*}
	\Theta_b^{(2)}(\xi) &= e^{ib \xi} e^{-{ a ^2 \xi^2\over 2}} \big( \oN_0 (\Phi_{n,\xi} +  \Phi_{n}^{\star}) - 1 \big) \widehat{\vartheta_{\delta_n}}(\xi) \\ &=  e^{ib \xi} e^{-{ a ^2 \xi^2\over 2}} \sum_{0\leq k\leq \sqrt[4] n } \big(  e^{i \xi{k \over \sqrt n}}-1 \big) \big( \oN_0 \chi_k \big) \cdot \widehat{\vartheta_{\delta_n}}(\xi).
	\end{align*}
	As $\norm{\widehat{\vartheta_{\delta_n}}}_\infty \lesssim 1$, we obtain 
	\begin{align*}
	\big|\Theta_b^{(2)}(\xi)\big| \lesssim  e^{-{ a ^2 \xi^2\over 2}} \sum_{0\leq k\leq \sqrt[4] n } \big|  e^{i \xi{k \over \sqrt n}}-1\big| \big( \oN_0\chi_k \big)    \lesssim e^{-{ a ^2 \xi^2\over 2}}\sum_{1\leq k\leq \sqrt[4] n } | \xi| {k \over \sqrt n} \big( \oN_0\chi_k \big) \lesssim e^{-{ a ^2 \xi^2\over 2}} {|\xi|\over \sqrt n}.
	\end{align*}
	 Therefore,
	$$   \int_{-\xi_0\sqrt n}^{\xi_0\sqrt n} \Big|{\Theta_b^{(2)}(\xi)\over \xi} \Big|  \,\diff \xi   \lesssim  \int_{-\xi_0\sqrt n}^{\xi_0\sqrt n}  {1\over |\xi|} \cdot e^{-{ a ^2 \xi^2\over 2}} {|\xi|\over \sqrt n}  \,\diff \xi \lesssim {1 \over \sqrt n},$$
	thus proving the lemma.
\end{proof}

Gathering the above estimates, we can finish the proof.

\begin{proof}[End of proof of Theorem \ref{thm:BE-coeff}] 
	Recall that our goal is to prove \eqref{goal-1-varphi}. Estimate \eqref{eq:BE-main-estimate} together with Lemmas \ref{lemma:theta-1-bound} and \ref{lemma:theta-2-bound} give that
	$\big|F_n(b)-H(b)\big| \leq C'/ \sqrt n$ for all $b \in \R$, where $C' > 0$ is a constant. Since $H(b): =\frac{1}{\sqrt{2 \pi}  \, a } \int_{-\infty}^b e^{-\frac{s^2}{2  a ^2}} \, \diff s$,  the last estimate together with Lemma \ref{lemma:Ln-Fn} and the easy fact that $\sup_{b \in \R} \big|H(b)-H(b \pm 1/ \sqrt n)\big| \lesssim 1/ \sqrt n$ give that $ \big|\oL_n(b)-H(b)\big| \leq C'' /  \sqrt n$ for some constant $C'' > 0$. We have thus shown that   \eqref{goal-1-varphi} holds. Observe that all of our estimates are uniform in $x,y \in \P^1$ and $b\in\R$. This completes the proof of the theorem.
\end{proof}

\appendix

\section{Properties of functions in $W^{1,2}$}

In this appendix, we recall some properties of functions in the Sobolev space $W^{1,2}$ that are used in this work. Most of them also hold for functions in the higher dimensional analogues $W_*$ of $W^{1,2}$ introduced by Sibony and the first author. These spaces are defined in \cite{dinh-sibony:decay-correlations} for any K\"ahler manifold and are very useful in applications to dynamics and complex analysis. For simplicity, we only deal here with the special case of $\P^1$.

By definition, elements of $W^{1,2}$ are equivalence classes of measurable functions with finite $W^{1,2}$-norm, where we identify functions that agree almost everywhere with respect to the Lebesgue measure. However, we'll see below that functions on $W^{1,2}$ admit canonically defined pointwise values outside a polar subset of $\P^1$.

We start by recalling some notation. We denote $\diff^c:=  \frac{i}{2\pi} (\overline \partial - \partial)$, so that $\ddc = \frac{i}{\pi} \partial \bar \partial$. Recall that the standard Fubini-Study form $\omegaFS$ on $\P^1$ is the unique smooth closed $(1,1)$-form of unit integral that is invariant by the transitive action of the unitary group. A function $\psi: \P^1 \to \R \cup \{-\infty\}$ is \textit{quasi-subharmonic} if it locally the difference of a subharmonic function and a smooth one. If $\psi$ is quasi-subharmonic, there is a constant $c>0$ such that $c \, \omegaFS + \ddc \psi$ is a positive measure on $\P^1$. A set $E \subset \P^1$ is \textit{polar} if there is a quasi-subharmonic $\psi$ such that $E \subset \{\psi = - \infty\}$.

There are various useful notions of convergence of functions in $W^{1,2}$.

\begin{definition}\rm
Let $(u_k)_{k \geq 1}$ be a sequence in $W^{1,2}$ and $u \in W^{1,2}$.
\begin{itemize}
\item We say that $u_k \to u$ \textit{strongly} or \textit{in $W^{1,2}$-norm} if $\|u_k - u\|_{W^{1,2}} \to 0$ as $k \to \infty$.
\item We say that $u_k \to u$ \textit{weakly} if $u_k$ converge to $u$ in the sense of currents (or Schwartz distributions) and $\|u_k\|_{W^{1,2}}$ is uniformly bounded.
\item We say that $u_k \to u$ \textit{locally nicely} if $u_k \to u$ weakly and if for every $x \in \P^1$ there exist an open neighborhood $U$ of $x$ and a sequence of subharmonic functions $\varphi_k$ on $U$ decreasing pointwise to a function $\varphi$ so that, for every $k\geq 1$, one has
\begin{equation} \label{eq:locally-nice}
i \del u_k\big|_U \wedge \overline{\del u_k}\big|_U \leq \ddc \varphi_k
\end{equation}
and the same for $\overline {u_k}$.

\item We say that $u_k \to u$ \textit{globally nicely} if $u_k \to u$ weakly and if there exist a number $c>0$ and a sequence of quasi-subharmonic functions $\psi_k$ on $\P^1$ decreasing pointwise to a function $\psi$ so that, for every $k\geq 1$, one has
\begin{equation} \label{eq:globally-nice}
i \del u_k \wedge \overline{\del u_k } \leq c \, \omegaFS + \ddc \psi_k
\end{equation}
and the same for $\overline {u_k}$.
\end{itemize}
\end{definition}

\medskip

Notice that the functions $\psi_k$ above satisfy $ c \, \omegaFS + \ddc \psi_k \geq 0$. It follows that the limit function $\psi$ is also  quasi-subharmonic and satisfies the same inequality. Observe that when the $u_k$ are real, the estimates \eqref{eq:locally-nice} and \eqref{eq:globally-nice} for $u_k$ and $\overline {u_k}$ are equivalent.

\begin{remark}
Dinh-Marinescu-Vu only work with local nice convergence. It is clear that global nice convergence implies local nice convergence. It is enough to take $\varphi_k = c \, \mathbf g + \psi_k|_U$ and $\varphi =c \, \mathbf g + \psi |_U $ where $\mathbf g$ is a local potential of $\omegaFS$ on $U$, that is, $\omegaFS = \ddc \mathbf g$ on $U$.
\end{remark}

Using that $\P^1$ is homogeneous under the holomorphic action of $\SL_2(\C)$, we can regularize functions on $\P^1$ by convolution in a standard way, see \cite{dinh-sibony:acta}. Using this regularization, the following lemma can be proved in the same way as \cite[Lemma 2.3]{dinh-marinescu-vu}. See also \cite[Lemma 5]{vigny:dirichlet}.

\begin{lemma} \label{lemma:standard-reg}
Let $u \in W^{1,2}$. Then, there exists a sequence $(u_k)_{k \geq 1}$ of functions on $\P^1$ such that $u_k$ is smooth for every $k$ and $u_k \to u$ globally nicely as $k \to \infty$. Moreover, if $u$ is bounded then $u_k$ are uniformly bounded.
\end{lemma}

\begin{theorem}[\cite{dinh-marinescu-vu}--Theorem 2.10] \label{thm:good-rep} Let $u \in W^{1,2}$. Then there exists a Borel function $\widetilde u$ defined everywhere on $\P^1$ such that $\widetilde u = u$ almost everywhere (in particular, $\widetilde u$ and $u$ define the same element in $W^{1,2}$) and the following holds:
\begin{itemize}
\item[(i)] For every open set $U \subset \P^1$ and every sequence $(u_k)_{k \geq 1}$ in $W^{1,2}(U) \cap \Cc^0(U)$ such that $u_k \to u$ nicely, there is a subsequence $(u_{k_\ell})_{\ell_\geq 1}$  and a polar set $E \subset U$ such that $u_{k_{\ell}}$ converges pointwise to $\widetilde u$ everywhere on $U \setminus E$.

\item[(ii)] If $\widetilde v$ is another Borel function satisfying (i) then $\widetilde v = \widetilde u$ on $U$ except on a polar set.
\end{itemize}
\end{theorem}

\begin{definition} \rm \label{def:good-rep}
Let $u \in W^{1,2}$. Then a function $\widetilde u$ satisfying (i) from Theorem \ref{thm:good-rep} is called a \textit{good representative} of $u$.
\end{definition}

Notice that, by definition, a good representative is always defined everywhere on $\P^1$ except on a polar set. Any function that is both a nice limit and a pointwise limit of a sequence of continuous functions in $W^{1,2}$ is a good representative. If a function $u \in W^{1,2}$ is continuous in some open set $U$ then $\widetilde u = u$ on $U$ outside a polar set. Also, by Theorem \ref{thm:good-rep}-(ii), any two good representatives agree at every point of $\P^1$ except on a polar set. From \cite{dinh-marinescu-vu}, when $u$ is real, the functions $\widetilde u^+, \widetilde u^-$  are good representatives of $u^+,u^-$. Here $u^+:=\max(u,0)$ and $u^-:=\max(-u,0)$. It also follows from the arguments in \cite{dinh-marinescu-vu} that if $\widetilde u$ is a good representative of $u$ then the set $\{|\widetilde u| = + \infty \}$ is polar and $\Re \widetilde u$, $\Im \widetilde u$ and $|\widetilde u|$ are good representatives of $\Re   u$, $\Im u$ and $|u|$ respectively. 

The next results show that good representatives are preserved by the action of the perturbed Markov operators $\oP_\xi$ defined in (\ref{eq:def-P_t}).

\begin{lemma} \label{lemma:good-rep-P_t-2}
Let $\mu$ be a probability measure on $G=\SL_2(\C)$ with a finite first moment, i.e.,  $\int_G \log \|g\| \, \diff \mu(g) <\infty$. Let $\widetilde u$ be a good representative of a bounded function $u \in W^{1,2}$.  Then, the integral $\int_G e^{i \xi \sigma_g(x)} \widetilde u (gx) \, \diff \mu(g)$ is well-defined and finite for all $x$ outside a polar set.
\end{lemma}

\begin{proof}
By definition $\widetilde u$ is well-defined and finite outside a polar set $E$. Define $$G_x :=\{g \in G: gx \in E\} \subset G \quad \text{ and } \quad E':=\{x \in \P^1 : \mu(G_x) \neq 0 \}.$$

Notice that the integral $\int_G e^{i \xi \sigma_g(x)}\widetilde u (gx) \, \diff \mu(g)$ is well-defined for $x \notin E'$, since it coincides with $\int_{G \setminus G_x} e^{i \xi \sigma_g(x)} \widetilde u (gx) \, \diff \mu(g)$ and the integrand in the last integral is well-defined and finite.

  We claim that $E'$ is polar. In order to see this, let $\varphi$ be a quasi-subharmonic function such that  $E \subset \{\varphi = - \infty\}$. We may assume that $\varphi\leq -2$ and $\varphi$ is the limit of a decreasing sequence of negative smooth functions $\varphi_n$ with $\ddc \varphi_n \geq -\omega$. One can show that $\Psi := \log (-\varphi)$ belongs to $W^{1,2}$ and is the increasing limit of the sequence $\Psi_n :=  \log (-\varphi_n)$ which is bounded in $W^{1,2}$, see \cite{dinh-sibony:decay-correlations} and  \cite[Example 1]{vigny:dirichlet}. Observe that, for every constant $A>0$, there is a neighborhood $U_A$ of $E'$ such that $\oP \Psi_n \geq A$ on $U_A$ for $n$ large enough. because $\lim_{n \to \infty} \oP \Psi_n = + \infty$ on $E'$. We deduce that $\widetilde {\oP \Psi} \geq A$ on $U_A$ outside a polar set. Hence $\widetilde{\oP \Psi} = +\infty$ on $E'$ outside a polar set. We have seen that $\widetilde {\oP \Psi}$ is finite outside a polar set, so $E'$ must be polar. This concludes the proof.
\end{proof}

\begin{proposition} \label{prop:good-rep-P_t}
Let $\mu$ be a probability measure on $G=\SL_2(\C)$ with a finite first moment, i.e.,  $\int_G \log \|g\| \, \diff \mu(g) <\infty$. Let $\widetilde u$ be a good representative of a bounded function $u \in W^{1,2}$.  Then $\oP_\xi \widetilde u$ is a good representative of $\oP_\xi u$ and its value at $x$ coincides with the integral $\int_G e^{i \xi \sigma_g(x)}\widetilde u (gx) \, \diff \mu(g)$ for all $x$ outside a polar set.
\end{proposition}

\begin{proof}
 Let $u_k$ be as in Lemma \ref{lemma:standard-reg}. Then $u_k$ are uniformly bounded, say by $M >0$, and converge to $\widetilde u$ globally nicely and pointwise outside a polar set, see Theorem \ref{thm:good-rep}. Recall that $\oP_\xi: W^{1,2} \to W^{1,2}$ is a bounded operator (cf. Proposition \ref{prop:P_t-bounded}) and $\oP_\xi u_k$ are continuous.

\vskip5pt

\noindent \textbf{Claim.} $\oP_\xi u_k$ converges globally nicely to $\oP_\xi \widetilde u$ as $k \to \infty$.
\proof[Proof of Claim]
It is easy to see that $\oP_\xi u_k \to \oP_\xi \widetilde u$ weakly as $k \to 0 $. To prove the nice convergence, it remains to show that
\begin{equation*}
i \del \oP_\xi u_k \wedge \overline{\del \oP_\xi u_k } \leq c' \, \omegaFS + \ddc \chi_k,
\end{equation*}
for some constant $c'>0$ and quasi-subharmonic functions $\chi_k$ decreasing to some function $\chi$ as $k \to \infty$, and an analogous estimate for $\overline{\oP_\xi u_k}$ instead of $\oP_\xi u_k$. We'll only prove the first one, as the second one follows from the same arguments.

We have $$\del \oP_\xi u_k = \int_G i \xi \, e^{i \xi \sigma_g} \del \sigma_g \, g^* u_k \, \diff \mu(g) +  \int_G e^{i \xi \sigma_g} g^* \del u_k \, \diff \mu(g): = \phi_1 + \phi_2.$$

By Cauchy-Schwarz inequality, we have $$i \, \phi_1 \wedge \overline{\phi_1} \leq \xi^2 \int_G |g^*u_k|^2 \, i \del \sigma_g \wedge \overline{\del \sigma_g}\, \diff \mu(g) \leq M^2 \xi^2 \check \rho \, \omegaFS,$$ where $i \del \sigma_g \wedge \overline{\del \sigma_g} = \rho_g \, \omegaFS$ and $\check \rho = \int_G \rho_g \, \diff \mu (g)$. We have from Lemma \ref{lemma:sigma-estimates} that $M^ 2 \xi^2 \check \rho \, \omegaFS$ is a positive measure of finite mass on $\P^1$, so we can write it as $c_1 \omegaFS + \ddc \mathbf v_1$ for some constant $c_1 >0$ and a quasi-subharmonic function $\mathbf v_1$. We conclude that $i \, \phi_1 \wedge \overline{\phi_1} \leq c_1 \omegaFS + \ddc \mathbf v_1$.

Using (\ref{eq:globally-nice}) one gets $$i\, e^{i \xi \sigma_g} g^* \del u_k \wedge \overline{e^{i \xi \sigma_g} g^* \del u_k} = g^*(i \, \del u_k \wedge \overline{\del u_k}) \leq c g^* \omegaFS + \ddc (\psi_k \circ g),$$
where $\psi_k$ are quasi-subharmonic and decrease to $\psi$.

Let $\Omega : = \int_G g^* \omegaFS \, \diff \mu(g)$ and $\eta_k := \int_G \psi_k \circ g \, \diff \mu(g)= \oP_0 \psi_k$. By the above estimate and Cauchy-Schwarz inequality we get $$i \, \phi_2 \wedge \overline{\phi_2} \leq  \int_G e^{i \xi \sigma_g} g^* \del u_k \wedge \overline{e^{i \xi \sigma_g} g^* \del u_k} \, \diff \mu(g) \leq c\Omega + \ddc \eta_k.$$ 

Since $\Omega$ has mass one, we can write $c\Omega = c\omegaFS + \ddc \mathbf v_2$ for some quasi-subharmonic function $\mathbf v_2$. Setting $\chi_k=2 \mathbf v_2 + 2\eta_k + 2\mathbf v_1$ yields $$i \, \phi_2 \wedge \overline{\phi_2} \leq c\omegaFS + \ddc (\chi_k \slash 2 -  \mathbf v_1).$$ 

Using Cauchy-Schwarz inequality and combining with the estimate on $i \,\phi_1 \wedge \overline{\phi_1}$, we get $$i \del \oP_\xi u_k \wedge \overline{\del \oP_\xi u_k } \leq 2i \,\phi_1 \wedge \overline{\phi_1} + 2i \,\phi_2 \wedge \overline{\phi_2} \leq  2(c_1 + c) \, \omegaFS + \ddc \chi_k,$$
and $\chi_k$ decrease to $\chi := 2(\mathbf v_1 + \mathbf v_2 + \oP_0 \psi)$. This proves the claim. \endproof

\vskip5pt

Define now $v(x) := \oP_\xi \widetilde u (x) = \int_G e^{i \xi \sigma_g(x)} \widetilde u (gx) \, \diff \mu(g)$. By Lemma \ref{lemma:good-rep-P_t-2}, $v$ is well-defined outside a polar set.  Lebesgue's dominated convergence theorem implies that $\oP_\xi u_k$ converge to $v$ pointwise outside a polar set. Finally, the above claim and Theorem \ref{thm:good-rep} implies that $v = \widetilde{\oP_\xi u}$ outside a polar set. This finishes the proof.
\end{proof}

We end this appendix by proving some results that were used in Section \ref{sec:spec-Pt}.

\begin{proposition} \label{p:WPC}
Let $\nu$ be a Schwarz distribution on $\P^1$.  Then,  the following are equivalent:
\begin{enumerate}
\item There is a constant $c>0$ such that $|\langle\nu, u \rangle|\leq c\| u \|_{W^{1,2}}$  for every smooth test function $ u $ on $\P^1$.

\item There exist a $(1,1)$-form with bounded coefficients $\alpha$ and a $(1,0)$-form with $L^2$  coefficients $\eta$ on $\P^1$ such that $\nu=\alpha+\dbar \eta$.
\end{enumerate}

In particular,  if $\nu$ satisfies one of the properties above,  the linear functional $ u \mapsto \langle\nu, u \rangle$ defined for smooth $u$ can be extended to a functional on $ W^{1,2}$ that is continuous with respect to both strong and weak topologies on $W^{1,2}$.
\end{proposition}

\begin{proof}
We first prove that (1) implies (2).  Statement (1) means that $\nu$ belongs to the dual space $H^{-1}$ of $W^{1,2}$. Notice that $\nu$ is a $(1,1)$-current on $\P^1$ which, by degree reasons,  is $\overline \partial$-closed.  Let  $\alpha$ be a smooth $(1,1)$-form cohomologous to $\nu$. Then  $\nu=\alpha+\dbar \eta$ for some $(1,0)$-current $\eta$. By degree reasons $\del \eta = 0$ and since the Dolbeault cohomology group $H^{1,0}(\P^1,\C)$ is trivial, we can write $\eta = i \del v$ for some Schwartz distribution $v$. Then, $i \del \overline \del v = \alpha - \nu$. Notice that, in a local holomorphic coordinate, $i \del \overline \del v$ coincides with the Laplace operator up to a constant multiple. Since $\alpha - \nu \in H^{-1}$, it follows from standard elliptic estimates for the Laplace operator that $v \in H^1 = W^{1,2}$. Therefore, $\eta = i \del v$ has $L^2$ coefficients, thus giving (2).

\vskip3pt

Let us now prove that (2) implies (1).   If $\nu=\alpha+\dbar \eta$,  then 
\begin{equation} \label{eq:nu-hodge}
\langle\nu, u \rangle = \langle\alpha, u \rangle + \langle\eta, \dbar  u \rangle.
\end{equation} holds for $u$ smooth.  It is clear that $| \langle\alpha, u \rangle | \lesssim \|u\|_{W^{1,2}}$.  On the other hand,  by Cauchy-Schwarz inequality, we have that $|\langle\eta, \dbar  u \rangle|\leq \|\eta\|_{L^2} \|\dbar u\|_{L^2} \leq \|\eta\|_{L^2} \|u\|_{W^{1,2}}$, so (1) follows. 

\vskip3pt

For the last assertion, observe that we only need to consider the weak topology as the case of strong topology is a direct consequence.  Assuming that  $\eta$ belongs to $L^2_{(1,0)}$,  the right hand side of \eqref{eq:nu-hodge} defines a linear functional on $W^{1,2}$. Since $\alpha$ has bounded coefficients and weak convergence in $W^{1,2}$ implies $L^1$ convergence  \cite[Corollary 1.5]{dinh-marinescu-vu}, it follows  that $u \mapsto \langle\alpha, u \rangle$ is continuous for the weak topology.  The functional $u \mapsto \langle\eta, \dbar  u \rangle$ is  also  continuous for the weak topology on $W^{1,2}$ because this is true for $\eta$ smooth and, in the general case,  one can approximate $\eta$ in $L^2_{(1,0)}$ by smooth forms.  This finishes the proof of the proposition.
\end{proof}

\begin{corollary} \label{c:WPC}
Let $\nu$ be a positive measure on $\P^1$ and $\varphi$ a function in $L^\infty(\nu)$. Assume that $\nu$ satisfies the properties from Proposition \ref{p:WPC}. Then $\varphi\nu$ satisfies the same properties.
\end{corollary}

\begin{proof}
We will show that $\varphi\nu$ satisfies Property (1) in Proposition \ref{p:WPC}. For this purpose, we can assume that $0\leq \varphi\leq 1$. Let $u$ be a real-valued continuous function of unit norm in $W^{1,2}$ and write $u=u^+-u^-$ with $u^\pm:=\max(\pm u,0)$. The functions $u^\pm$ are continuous and have bounded $W^{1,2}$ norms. It is now easy to see, using Property (1) for $\nu$ and that $0\leq \varphi\leq 1$, that $\langle \varphi\nu, u\rangle=\langle \varphi\nu, u^+\rangle-\langle \varphi\nu, u^-\rangle$ is bounded. This shows that $\varphi\nu$ satisfies (1) for real-valued functions $u$. The general case follows by decomposing $u$ into its real and imaginary parts.
\end{proof}

\begin{lemma} \label{lemma:int-good-rep}
Let $\nu$ be a positive measure on $\P^1$. Assume that $\nu$ satisfies the properties from Proposition \ref{p:WPC}.  Let $\widetilde u$ be a good representative of $u$. Then $\widetilde u \in L^1(\nu)$ and $\lp \nu, u \rp = \int_{\P^1} \widetilde u \, \diff \nu$.  
\end{lemma}

\begin{proof}
As above, it is enough to prove the result for $u$ non-negative. Assume first that $u$ is bounded. Let $u_k$ be as in Lemma \ref{lemma:standard-reg}. On the one hand, we have that $\lp \nu, u_k \rp \to \lp \nu, u \rp$ as $k \to \infty$, because $u_k \to u$ weakly and $\nu$ is continuous with respect to weak convergence. On the other hand, we have that $\lp \nu, u_k \rp = \int_{\P^1}  u_k \, \diff \nu$, since $u_k$ is smooth.  Arguing as in in the proof of  \cite[Theorem 2.1]{DKW:PAMQ}, we see that $\nu$ has no mass on polar sets. By Lemma \ref{lemma:standard-reg} and Theorem \ref{thm:good-rep}, $u_k \to \widetilde u$ globally nicely and pointwise outside a polar set along a subsequence. As $u_k$ are uniformly bounded, we can apply Lebesgue's dominated convergence theorem. This gives that $\int_{\P^1}  u_{k} \, \diff \nu \to \int_{\P^1}  \widetilde u \, \diff \nu$. This completes the proof when $u$ is bounded.

Let now $u$ be an arbitrary non-negative function in $W^{1,2}$. For $M \geq 1$ let $u_M:= \min(u,M)$. Then, by \cite[Lemma 2.4]{dinh-marinescu-vu}, $u_M \in W^{1,2}$ and has $\min (\widetilde u,M)$ as  a good representative. Moreover, $\|u_M\|_{W^{1,2}}$ is bounded independently of $M$. It follows that $u_M \to u$ weakly and hence $\lp \nu, u_M \rp \to \lp \nu, u \rp$. By the first part of the proposition  $\lp \nu, u_M \rp = \int_{\P^1}  \min (\widetilde u,M) \, \diff \nu$. Letting $M \to +\infty$ and using the monotone convergence theorem  yield $\lp \nu, u \rp = \int_{\P^1} \widetilde u \, \diff \nu$.  Observe that, from Property (1) of Proposition \ref{p:WPC},  the last quantity is finite.  This completes the proof.
\end{proof}


\vspace{-7pt}

\begin{thebibliography}{CDMP21b}


\bibitem[BQ16a]{benoist-quint:CLT}
Yves Benoist and Jean-Fran\c{c}ois Quint.
\newblock Central limit theorem for linear groups.
\newblock {\em Ann. Probab.}, 44(2):1308--1340, 2016.

\bibitem[BQ16b]{benoist-quint:book}
Yves Benoist and Jean-Fran\c{c}ois Quint.
\newblock {\em Random walks on reductive groups}, volume~62 of {\em Ergebnisse
  der Mathematik und ihrer Grenzgebiete. 3. Folge. A Series of Modern Surveys
  in Mathematics [Results in Mathematics and Related Areas. 3rd Series. A
  Series of Modern Surveys in Mathematics]}.
\newblock Springer, Cham, 2016.

\bibitem[BD21]{bianchi-dinh}
Fabrizio Bianchi and Tien-Cuong Dinh.
\newblock Existence and properties of equilibrium states of holomorphic endomorphisms of $\mathbb{P}^k$.
\newblock {\em {\tt arXiv:2007.04595}}, 2021. 
  
  \bibitem[BL85]{bougerol-lacroix}
Philippe Bougerol and Jean Lacroix.
\newblock {\em Products of random matrices with applications to
  {S}chr\"{o}dinger operators}, volume~8 of {\em Progress in Probability and
  Statistics}.
\newblock Birkh\"{a}user Boston, Inc., Boston, MA, 1985.

\bibitem[CG02]{conze-guivarch}
Jean-Pierre Conze and Yves Guivarc'h.
\newblock Densit\'{e} d'orbites d'actions de groupes lin\'{e}aires et
  propri\'{e}t\'{e}s d'\'{e}quidistribution de marches al\'{e}atoires.
\newblock In {\em Rigidity in dynamics and geometry ({C}ambridge, 2000)}, pages
  39--76. Springer, Berlin, 2002.
  
\bibitem[CDJ17]{cuny-dedecker-jan}
Christophe Cuny, J\'{e}r\^{o}me Dedecker, and Christophe Jan.
\newblock Limit theorems for the left random walk on {${\rm GL}_d(\Bbb R)$}.
\newblock {\em Ann. Inst. Henri Poincar\'{e} Probab. Stat.}, 53(4):1839--1865,
  2017.

\bibitem[CDM18]{cuny-dedecker-merlevede}
Christophe {Cuny}, J\'er\^ome {Dedecker}, and Florence {Merlev\`ede}.
\newblock {On the Koml\'os, Major and Tusn\'ady strong approximation for some
  classes of random iterates}.
\newblock {\em {Stochastic Processes Appl.}}, 128(4):1347--1385, 2018.

\bibitem[CDMP21]{cuny-dedecker-merlevede-peligrad}
Christophe {Cuny}, J\'er\^ome {Dedecker}, Florence {Merlev\`ede}, and Magda
  {Peligrad}.
\newblock Berry-{E}sseen type bounds for the left random walk on
  {${\rm GL}_d(\Bbb R)$} under polynomial moment conditions.
\newblock {\em {\tt hal-03329189f}}. To appear in {\em Annals of Probability},  2021.

\bibitem[CDMP22]{cuny-dedecker-merlevede-peligrad-2}
Christophe Cuny, J\'{e}r\^{o}me Dedecker, Florence Merlev\`ede, and Magda
  Peligrad.
\newblock Berry-{E}sseen type bounds for the matrix coefficients and the
  spectral radius of the left random walk on {$GL_d(\Bbb R)$}.
\newblock {\em C. R. Math. Acad. Sci. Paris}, 360:475--482, 2022.

\bibitem[DKW20]{DKW:IJM}
Tien-Cuong Dinh, Lucas Kaufmann, and Hao Wu.
\newblock Dynamics of holomorphic correspondences on {R}iemann surfaces.
\newblock {\em Internat. J. Math.}, 31(5):2050036, 21, 2020.

\bibitem[DKW21a]{DKW:PAMQ}
Tien-Cuong Dinh, Lucas Kaufmann, and Hao Wu.
\newblock Products of random matrices: a dynamical point of view.
\newblock {\em Pure Appl. Math. Q.}, 17(3):933--969, 2021.

\bibitem[DKW21b]{DKW:fourier}
Tien-Cuong Dinh, Lucas Kaufmann, and Hao Wu.
\newblock Decay of Fourier coefficients for Furstenberg measures.
\newblock {\em {\tt arXiv:2108.06006}}. To appear in {\em Trans. Amer. Math. Soc.}, 2021.


\bibitem[DKW22]{DKW:BE-LLT-coeff}
Tien-Cuong Dinh, Lucas Kaufmann, and Hao Wu.
\newblock Berry–esseen bound and local limit theorem for the coefficients of
  products of random matrices.
\newblock {\em Journal of the Institute of Mathematics of Jussieu}, in press, available online {\tt doi:10.1017/S1474748022000561}, 2022.

\bibitem[DMV20]{dinh-marinescu-vu}
Tien-Cuong Dinh, George Marinescu, and Duc-Viet Vu.
\newblock {M}oser-{T}rudinger inequalities and complex {M}onge-{A}mp\` ere
  equation.
\newblock {\em {\tt arXiv:2006.07979}}. To appear in {\em Ann. Sc. Norm. Super. Pisa Cl. Sci.}, 2020.

\bibitem[DS06]{dinh-sibony:decay-correlations}
Tien-Cuong Dinh and Nessim Sibony.
\newblock Decay of correlations and the central limit theorem for meromorphic
  maps.
\newblock {\em Comm. Pure Appl. Math.}, 59(5):754--768, 2006.

\bibitem[DS09]{dinh-sibony:acta}
Tien-Cuong Dinh and Nessim Sibony.
\newblock Super-potentials of positive closed currents, intersection theory and
  dynamics.
\newblock {\em Acta Math.}, 203(1):1--82, 2009.


\bibitem[Fel71]{feller:book}
William Feller.
\newblock {\em An introduction to probability theory and its applications.
  {V}ol. {II}}.
\newblock John Wiley \& Sons, Inc., New York-London-Sydney, second edition,
  1971.

\bibitem[Fur63]{furstenberg}
Harry Furstenberg.
\newblock Noncommuting random products.
\newblock {\em Trans. Amer. Math. Soc.}, 108:377--428, 1963.

\bibitem[FK60]{furstenberg-kesten}
Harry Furstenberg and Harry Kesten.
\newblock Products of random matrices.
\newblock {\em Ann. Math. Statist.}, 31:457--469, 1960.


\bibitem[Gou15]{gouezel:spectral-methods}
S\'{e}bastien Gou\"{e}zel.
\newblock Limit theorems in dynamical systems using the spectral method.
\newblock In {\em Hyperbolic dynamics, fluctuations and large deviations},
  volume~89 of {\em Proc. Sympos. Pure Math.}, pages 161--193. Amer. Math.
  Soc., Providence, RI, 2015.
  
  \bibitem[GLLP20]{grama-condi}
Ion Grama, Ronan Lauvergnat, and \'{E}mile Le~Page.
\newblock Conditioned local limit theorems for random walks defined on finite
  {M}arkov chains.
\newblock {\em Probab. Theory Related Fields}, 176(1-2):669--735, 2020.

\bibitem[GQX22]{grama-quint-xiao}
Ion Grama, Jean-Fran\c{c}ois Quint, and Hui Xiao.
\newblock A zero-one law for invariant measures and a local limit theorem for
  coefficients of random walks on the general linear group.
\newblock {\em Ann. Inst. Henri Poincar\'{e} Probab. Stat.}, 58(4):2321--2346,
  2022.


\bibitem[Kat95]{kato:book}
Tosio Kato.
\newblock {\em Perturbation theory for linear operators}.
\newblock Classics in Mathematics. Springer-Verlag, Berlin, 1995.
\newblock Reprint of the 1980 edition.


\bibitem[LP82]{lepage:theoremes-limites}
\'{E}mile Le~Page.
\newblock Th\'{e}or\`emes limites pour les produits de matrices al\'{e}atoires.
\newblock In {\em Probability measures on groups ({O}berwolfach, 1981)}, volume
  928 of {\em Lecture Notes in Math.}, pages 258--303. Springer, Berlin-New
  York, 1982.

\bibitem[Li18]{li:fourier}
Jialun Li.
\newblock Decrease of {F}ourier coefficients of stationary measures.
\newblock {\em Math. Ann.}, 372(3-4):1189--1238, 2018.

\bibitem[Li20]{li:fourier-2}
Jialun Li.
\newblock Fourier decay, {R}enewal theorem and spectral gaps for random walks
  on split semisimple {L}ie groups.
\newblock {\em {\tt arXiv:1811.06484}, to appear in Ann. Sci. de l'ÉNS}, 2020.

\bibitem[Mos71]{moser:trudinger}
J\"urgen Moser.
\newblock A sharp form of an inequality by {N}. {T}rudinger.
\newblock {\em Indiana Univ. Math. J.}, 20:1077--1092, 1970/71.

\bibitem[Vig07]{vigny:dirichlet}
Gabriel Vigny.
\newblock Dirichlet-like space and capacity in complex analysis in several
  variables.
\newblock {\em J. Funct. Anal.}, 252(1):247--277, 2007.

\bibitem[Wal82]{walters:ergodic-theory}
Peter Walters.
\newblock {\em An introduction to ergodic theory}, volume~79 of {\em Graduate
  Texts in Mathematics}.
\newblock Springer-Verlag, New York-Berlin, 1982.

\bibitem[XGL19]{xiao-grama-liu}
Hui Xiao, Ion Grama, and Quansheng Liu.
\newblock Berry-Esseen bound and precise moderate
deviations for products of random matrices
\newblock {\em {\tt arXiv:1907.02438}}, 2019.

\bibitem[XGL20]{xiao-grama-liu:berry-eseen}
Hui Xiao, Ion Grama, and Quansheng Liu.
\newblock Berry-{E}sseen bounds and moderate deviations for the random walk on
  {${\rm GL}_d(\Bbb R)$}.
\newblock {\em {\tt hal-02911533}}, 2020.


\end{thebibliography}

\end{document}